\documentclass[12pt]{article}

\usepackage[
    linktocpage=true,
    colorlinks=true,
    linkcolor=cyan,   % page numbers in ToC
    citecolor=cyan,   % bibliography links
    urlcolor=cyan     % URLs
]{hyperref}

\usepackage{graphicx} % Required for inserting images
\usepackage{amsmath}
\usepackage{hyperref}
\usepackage{upgreek}
\usepackage{xcolor}

\usepackage[toc,title,page]{appendix}
\usepackage{cleveref}
\usepackage[caption=false]{subfig}
\usepackage[margin=3cm]{geometry}
\usepackage{todonotes}
\usepackage{amsthm}
\usepackage{amssymb}
\usepackage{enumitem}
\usepackage{relsize}
\setlist[enumerate]{nosep}

\usepackage[backend=biber,
            style=numeric-comp,
            sortcites=true,
            giveninits=true, % Use initials in bibliography
            useprefix=true,
            sorting=nyt,
            maxbibnames=9]{biblatex}
\DeclareFieldFormat{titlecase}{#1}
\crefname{equation}{equation}{equations}
\crefname{figure}{Figure}{Figures}

\numberwithin{equation}{section}

\newcommand{\dd}{\mathrm{d}}

\newcommand{\defeq}{\mathrel{\mathop:}=}
\newcommand{\eqdef}{\mathrel{\mathop=}:}

\DeclareMathOperator{\Residu}{Res}
\DeclareMathOperator{\coker}{coker}
\DeclareMathOperator{\Range}{range}
\DeclareMathOperator{\dist}{dist}
\DeclareMathOperator{\Span}{span}

\DeclareMathOperator{\card}{card}

\newcommand{\fB}{\hat{\mathsf{B}}}
\newcommand{\fM}{\hat{\mathsf{M}}}

\newcommand{\fU}{\hat{\mathsf{U}}}
\newcommand{\fV}{\hat{\mathsf{V}}}
\newcommand{\vK}{ {\boldsymbol{\mathsf{k}}} }
\newcommand{\vL}{ {\boldsymbol{\ell}} }
\newcommand{\vM}{ {\boldsymbol{\mathsf{m}}} }
\newcommand{\cEll}{\mathcal{E}^c_{|\vK|}}
\newcommand{\fOpR}{\mathsf{R}}

\DeclareMathOperator{\bis}{\flat}

\newcommand{\antipodreflOp}{\mathrm{R}^{\pi}}
\newcommand{\swapOp}{\mathrm{S}}

\newcommand{\LinOp}{\mathrm{L}}
\newcommand{\LinOpFk}{\mathsf{L}}

\newcommand{\Real}{\mathrm{Re}}%{\mathfrak{R}}
\newcommand{\Imag}{\mathrm{Im}}%{\mathfrak{I}}

\newcommand{\Normpsiphi}{N^k}
\newcommand{\polyP}{P}

\newcommand{\lmbk}{\uplambda_k}
\newcommand{\sigk}{\upsigma_k}
\newcommand{\tauk}{\uptau_k}

\newcommand{\x}{x}

\newcommand{\thth}{\theta\theta}

\newcommand{\bdth}{\mathrm{\theta}}
\newcommand{\bdthth}{\mathrm{\theta \theta}}

\newcommand{\F}{\mathrm{F}}
\newcommand{\B}{\mathrm{B}}

\renewcommand{\H}{H}
\newcommand{\bPhi}{\mathrm{\Phi}}

\newcommand{\dR}{\mathbb{R}}
\newcommand{\dZ}{\mathbb{Z}}
\newcommand{\dN}{\mathbb{N}}
\newcommand{\dC}{\mathbb{C}}

\newcommand{\Torus}{\mathbb{T}}

\newcommand{\fW}{\mathsf{W}}
\newcommand{\fX}{\mathsf{X}}
\newcommand{\fY}{\mathsf{Y}}
\newcommand{\fC}{\mathsf{C}}

\newcommand{\Aek}{\mathsf{e}_{k}}

\newcommand{\Ac}{\mathsf{c}}
\newcommand{\Ay}{\mathsf{y}}
\newcommand{\Ax}{\mathsf{x}}
\newcommand{\Afk}{\mathsf{f}_{k}}

\newcommand{\Azoutk}{\mathsf{z}^{\mathrm{out}}_k}
\newcommand{\UnitCircleResThm}{\partial \mathbb{D}}

\newcommand{\Aint}{\mathrm{I}}

\newcommand{\Azink}{\mathsf{z}^{\mathrm{in}}_k}

\newtheorem{prop}{Proposition}
\newtheorem{thm}{Theorem}
\newtheorem{definition}{Definition}
\newtheorem{lem}{Lemma}
\newtheorem{cor}{Corollary}
\numberwithin{prop}{section}
\numberwithin{lem}{section}
\numberwithin{thm}{section}
\numberwithin{definition}{section}
\numberwithin{cor}{section}

\theoremstyle{remark}

\title{Existence of spot and lane stationary solutions for an ant active matter PDE model}

\author{Matthias Rakotomalala\thanks{Technical University of Munich,
Chair of Analysis and Modelling,  Boltzmannstr. 3,
85748 Garching b. M\"unchen, Germany.}, Oscar de Wit\thanks{Université Claude Bernard Lyon 1, Institut Camille Jordan, 69622 Villeurbanne, France.}}

\begin{document}

    \maketitle
    
    \begin{abstract}
        This paper studies the existence of multiple non-trivial stationary solutions of a partial differential equation (PDE) model introduced in~\cite{bertucci2024curvature}, motivated by collective ant behavior. Previous work suggested the presence of two types of non-trivial stationary solutions for this PDE system: spot and lane solutions. In this paper, we establish the existence of these families of solutions along a bifurcation sequence as the interaction strength grows, with progressively increasing numbers of clusters and parallel lanes, respectively. Finally, we show that, for small values of the anticipation parameter, the first bifurcating spot solutions are locally dynamically stable, while the lane solutions are unstable.
    \end{abstract}

    %in which the population concentrates into localized clusters, and lane solutions, in which the population concentrates and aligns along trail-like structures
    
    \tableofcontents
    
    \section{Introduction}
In this paper, we address the existence of probability density solutions $f$ of the following nonlinear PDE system
\begin{subequations}
    \label{eq:fullFstat}
    \begin{align}
       \label{eq:Masterf}
         0=& \  \sigma_x \Delta_\x f + \sigma_\theta \partial_{\bdthth} f
         - \lambda\, v_{\bdth} \cdot \nabla_\x f
         - \chi\, \partial_{\bdth}\left( \B_\tau[f]f \right), \\[2mm] \label{eq:MasterBdef}
         \B_\tau[f] =&\
         v_{\bdth}^\perp \cdot \nabla_\x c
         + \tau\, v^\perp_{\bdth} \cdot (\nabla_\x^2 c)\, v_{\bdth},\\
         \label{eq:Mastercdef}0 =& \ 
          \gamma c - \sigma_c \Delta_\x c - \int f \dd\theta.
    \end{align}
    \end{subequations}
The equations are considered with periodic boundary conditions: the first two are posed on $\Torus^2_1\times \Torus_{2\pi}$ and the third on $\Torus^2_1$. The domain variables are here denoted by $(x,\theta)=(x_1,x_2,\theta) \in\mathbb{T}^2_1\times\mathbb{T}_{2\pi}$.
% We consider $f(x,\theta)$ as a positive solution of mass one to this equation, that is, $f\in L^1(\Torus_1^2\times \Torus_{2\pi}, \dR)$ and $f \geq 0, \int f \dd \x \dd \theta = 1$.

This PDE corresponds to the stationary solutions of the time-dependent problem
\begin{subequations}\label{eq:ftPDE}
    \begin{align}
        \partial_t f&=\; 
     \sigma_x \Delta_\x f + \sigma_\theta \partial_{\bdthth} f
     - \lambda\, v_{\bdth} \cdot \nabla_\x f
     - \chi\, \partial_{\bdth}\!\left( \B_\tau[f]f \right),\\
     f|_{t=0}&=f^0(x,\theta).
    \end{align}
\end{subequations}
This time-dependent problem has been studied analytically and numerically in \cite{bertucci2024curvature,rakodeWit2025,bruna2024lane,bruna2025convergence}. The motivation for studying the PDE \eqref{eq:ftPDE} is multifold.

\textbf{Introduction of the model.} The PDE \eqref{eq:ftPDE} can be derived as the evolution equation of the law of a stochastic interacting particle system, as the number of particles goes to infinity, $N\to\infty$, also called the McKean--Vlasov mean-field limit. The interacting particle system reflects $N$ diffusive ant-particles moving at a fixed speed $\lambda$ in the direction of their orientation $v_\theta=(\cos\theta,\sin\theta)^\mathsf{T}$. The ants interact, as in \eqref{eq:MasterBdef}, such that their orientations follow the local \emph{anticipated} chemical gradient, $\nabla_\x c+\tau(\nabla^2_\x c) v_\bdth=\nabla_\x c(x+\tau v_\bdth)+O(\tau^2)$, where $\tau$ is the anticipation parameter. This parameter can also be thought of as the length of the ant-antenna. This is the only interaction that the ants undergo: there are no external stimuli such as pre-laid pheromone trails, attractive nests or food sources. For this reason, the model reflects the spontaneous self-organization of a collective of ants.

Results on the dynamics of the model \eqref{eq:ftPDE} so far are as follows. It was observed in \cite{bertucci2024curvature,bruna2024lane,bruna2025convergence} that, besides the trivial homogeneous solution $f=\frac{1}{2\pi}$, there are two types of non-trivial stationary states for \eqref{eq:ftPDE} in the linearly unstable regime. As discussed in those papers, the two families of stationary states can be interpreted, respectively, as aggregation \emph{spots} and bidirectional \emph{lanes}. The emergence of either type of stationary state critically depends on the size of $\tau$. That is, the numerical simulations indicated that there is a critical damping value $\tau^\ast$ (fixing all other parameters) in the linear unstable regime, such that below, $\tau<\tau^\ast$, one has the emergence of spots and above, $\tau>\tau^\ast$, one has the emergence of lanes. In \cite{rakodeWit2025}, the existence of the global attractor was proven, and its dimension in the linear unstable regime can be bounded from below by four, giving a first rigorous proof for the existence of non-trivial dynamics. The main result of this paper concerns the existence of a sequence of two types of stationary solutions -- spots and lanes --, making the earlier observations based on numerics rigorous and extending the result of the preceding paper \cite{rakodeWit2025}. As such, the paper provides a rigorous mathematical foundation for self-organized lane formation in collectives of ants.

\textbf{State of the art.} The PDE model \eqref{eq:ftPDE} can be considered as an active matter model or non-equilibrium model. Specifically, to maintain the fixed speed $\lambda$ (activity) the system is thermodynamically open. 
% Furthermore, the generator associated to the PDE \eqref{eq:ftPDE} linearized in $c$ is not self-adjoint and hence the PDE cannot be time-reversible at equilibrium \cite[Section 4.6]{pavliotis2014stochastic}. 
Moreover, the entropy $\int f(t)\log f(t)\dd x\dd\theta$ for \eqref{eq:ftPDE} does not always decrease monotonically, independent of the initial data, unlike, for example, gradient flow PDEs. Within the active matter physics literature, the model \eqref{eq:ftPDE} is an extension of the so-called Active Attractive Alignment (AAA) model, as coined in the physics review paper \cite{liebchen2019interactions}. The extension of the AAA model as in \eqref{eq:ftPDE} can be formally derived with a local first-order Taylor series expansion of the anticipated chemical gradient.

% {\color{pink} For the analysis of active matter PDEs, the tool of free-energy functionals is not readily available. Free-energy functionals can be used to prove the existence of different stationary states and the long-time convergence towards stationary states. }

For several active matter PDE models, different from this ant PDE model \eqref{eq:ftPDE}, in \cite{dietert2024nonlinear,albritton2023stabilizing,alasio2024regularity,merino2025stability}, nonlinear stability around the homogeneous stationary state on periodic domains was obtained, using a combination of regularity estimates and linear methods.

Some results have been obtained regarding the analysis of non-homogeneous states for non-equilibrium PDE models.
% In \cite{ohm2022weakly}, different stationary states are predicted for an active biological fluid model using a formal weakly nonlinear analysis. 
For example, using dynamical systems methods, in \cite{luccon2020emergence,giacomin2012transitions}, the existence of periodic and non-homogeneous stationary solutions for multiple non-equilibrium PDE models was proven, for sufficiently small individual dynamics. In \cite{cormier2021hopf}, using the Lyapunov--Schmidt method, the existence of a Hopf bifurcation for a non-equilibrium neuronal integrate-and-fire PDE is shown. In \cite{haragus2010local} and references therein, the existence of bifurcations for stationary and periodic solutions is shown for multiple non-equilibrium hydrodynamic models. In \cite{bogachev2019convergence}, convergence towards stationary solutions is shown using probabilistic methods, including a case with non-uniqueness of the stationary solutions.

\textbf{Structure of the paper.} We use the Lyapunov--Schmidt method to reduce problem \eqref{eq:fullFstat} to a finite dimensional implicit equation. To do this, we study the linearized operator and derive explicit integral formulas in the $\sigma_\theta\to 0$ limit. In Section~\ref{sec:Preliminaries}, we provide preliminary properties related to the symmetry of the equation. Next, in Section~\ref{sec:MainRes}, we state the main result of the paper, which is about the existence of the two families of non-homogeneous stationary solutions, spots and lanes. As a corollary, we also derive the stability properties of these stationary states for $\tau$ smaller than a constant above which the model loses super-criticality of the bifurcation. The remaining sections are devoted to the proof of the results of Section~\ref{sec:MainRes}. In Section~\ref{sec:fourdecomppos} we derive bases for the kernel and cokernel of the linearized operator and give a formula for the inverse of the linearized operator. In Section~\ref{sec:BirfucationCurves}, we derive the Lyapunov--Schmidt reduced equation and compute explicitly the leading order integral coefficients as $\sigma_\theta\to 0$. The explicit computations are given in full detail in Appendix~\ref{sec:AppendixIntegrals}. We then derive the asymptotic behavior of the Lyapunov--Schmidt integral coefficients as $\sigk=2\pi k\sigma_x/\lambda\to 0$. Using these results we can derive the main result. Finally, in Section~\ref{sec:StabilityAnalysis}, using the Principle of Reduced Stability (see~\cite{kielhofer1983principle}), we derive the stability properties in the super-critical case.

%\partial_\theta B_\tau[f] &= -(\cos_\theta \partial_{x_1} c+ \sin_\theta \partial_{x_2} c)\\
%    &\hspace{.5em}+ \tau (-\cos_{2\theta} \partial_{x_1 x_1} c +\cos_{2\theta} \partial_{x_2 x_2} c -2\sin_{2\theta}\partial_{x_1 x_2} c),\\    
% 0& = \sigma_x \partial_{xx}f - \lambda \cos(\theta)\partial_x f - \chi \partial_\theta\left(B_\tau[f] \left(1/2\pi+ f\right)\right) + \sigma_\theta \partial_{\thth}f\\
% 0 &= \gamma c  - \sigma_c \partial_{xx} c - \int f d\theta,\\
% B_\tau[f] &= -\sin(\theta) \partial_x c - \frac{\tau}{2} \sin(2\theta) \partial_{xx} c. %\sin(\theta)\cos(\theta) 

    \section{Preliminaries}
\label{sec:Preliminaries}
In this first section, we list definitions and introduce the notations and the functional framework used throughout the paper. Subsequently, we present the symmetries of the equation that inform the choice of the solution space for \eqref{eq:fullFstat}.

\subsection{Notations and Functional equation}
We use the functional Lebesgue and Sobolev spaces, $L^2_0(\Torus^2_1\times \Torus_{2\pi}, \dR),L^2_0(\Torus_{2\pi}, \dC), L^2(\Torus_{2\pi},\\ \dC)$, corresponding to $L^2$-integrable functions, and $H^2_0(\Torus^2_1\times \Torus_{2\pi},\dR), H^2_0(\Torus_{2\pi}, \dC), H^2(\Torus_{2\pi},\\ \dC)$, corresponding to the Sobolev spaces of $L^2$-integrable functions with $L^2$-integrable second derivatives. The $0$-subscript, $\cdot_0$, indicates that we consider the linear subspace of functions with zero average, $\int f=0$. These are real- and complex-valued function spaces and are equipped with their canonical Hilbertian structure. We use the scalar product on the $L^2$-spaces, and write, without ambiguity, $\langle \cdot, \cdot\rangle$ for the $L^2$-scalar product. For the sake of notational compactness, we denote, in order, the above spaces as $L^2_{0}, L^2_\theta(\dC), L^{2}_{\theta;0}(\dC)$ and $ H^2_{0}, H^2_{\theta}(\dC), H^{2}_{\theta;0}(\dC)$.
Similarly, for $0 \leq s$, we denote $H^{s}_{0}$ the Sobolev space $H^{s}_0(\Torus^2_1\times \Torus_{2\pi}, \dR)$, defined in Fourier space. Finally, we use the notation $\cdot^\star$ for the adjoint of an operator, and $\cdot^*$ for complex conjugation.

% As equation~\eqref{eq:fullFstat} is associated with a population density, we consider solutions that are probability densities through 

We define the equation functional $\F: H^2_{0} \times \dR  \to L^2_{0}$, as
\begin{equation*}
    \F(f,\chi) = \sigma_x \Delta_\x f + \sigma_\theta \partial_{\bdthth}f - \lambda v_{\bdth} \cdot \nabla_\x f - \chi \partial_{\bdth} \left(\B_\tau[f] \left(\tfrac{1}{2\pi}+ f\right)\right).
\end{equation*}
In our analysis, the interaction strength parameter $\chi$, serves as the bifurcation parameter. As the bifurcation analysis gives us solution curves $s \mapsto ( f_s,\chi_s) \in H^2_{0} \times \dR $ around $(0, \chi^*)$ of the equation,
\begin{equation}
    \label{eq:MainSimpleform}
    \F(f, \chi) = 0.
\end{equation}

From the Morrey embedding in dimension three we have the compact embedding $H^2_{0} \subset \subset L^\infty(\Torus^2_1\times\Torus_{2\pi},\dR)$ and as $\B_\tau [\tfrac{1}{2\pi} + f_s] = \B_\tau [f_s]$, we obtain that for small $s$, $\tfrac{1}{2\pi} + f_s$ is a positive probability density solution to equation~\eqref{eq:fullFstat}.

\subsection{Symmetries of the equation}

We now introduce the antipodal-reflection operator and the swap operator, which preserve the structure of the equation. We leverage these symmetry properties to study the bifurcation equation in the following sections.

\begin{definition}
    We define the \textbf{antipodal-reflection operator}, $\antipodreflOp: L^2_0\to L^2_0 $, as
    \begin{equation*}
        (\antipodreflOp f)(x,\theta) = f(-x, \theta + \pi).
    \end{equation*}    
\end{definition}
We have the following commutation relation between the equation functional and the antipodal-reflection operator.

\begin{prop}
    \label{antipodalReflectionCommutations}
    Given $f\in H^2_0$, we have the following commutation relation,
    \begin{equation*}
        \antipodreflOp  \F(f,\chi) = \F(\antipodreflOp f, \chi).
    \end{equation*}
    Furthermore, $\antipodreflOp$ commutes with any of the terms involved in $\F$, and distributes on products,
    \begin{equation*}        
        \antipodreflOp (fg) = \antipodreflOp f \antipodreflOp g,
    \end{equation*}
    for $f,g \in L^2_0$.
\end{prop}
The proof is in the Appendix~\ref{sec:AppendixSymetries}.
\begin{definition}
    \label{def:antipodaleven}
    We say that a function $f$ is \textbf{antipodal even} (resp. \textbf{antipodal odd}) if, %{\color{pink} (or \textbf{antipodal reflection invariant}, resp. \textbf{antipodal anti-reflection invariant} ) }
    \begin{equation*}
        \antipodreflOp  f = f,
    \end{equation*}
    (resp. $\antipodreflOp f = - f$), or equivalently,
    \begin{equation}
        f(x,\theta) = f(-x,\theta + \pi),
    \end{equation}
    (resp. $f(x,\theta) = -f(-x,\theta + \pi)$).
    We use the abbreviation \textbf{a.e.} for antipodal even and \textbf{a.o.} for antipodal odd.
\end{definition}
We now introduce the swap operator, corresponding to exchanging the $x_1$- and $x_2$-axes with the adequate rotation in the angle variable. 
\begin{definition}
    We define the \textbf{swap operator}, $\swapOp : L^2_0 \to L^2_0$, as
    \begin{equation*}
        (\swapOp f)(x_1,x_2,\theta) = f\left(-x_2,-x_1,-\theta-\tfrac{\pi}{2}\right).
    \end{equation*}
    We call a function $g \in L^2_0$ swap symmetric, if $\swapOp g = g$.
\end{definition}
The swap operator satisfies the following properties.

\begin{prop}
    \label{prop:SwapOperator}
    The swap operator $\swapOp$, is an \textbf{involution} $\swapOp^{-1} = \swapOp$, it is \textbf{self-adjoint} $\swapOp^\star = \swapOp$, and it \textbf{commutes} with the equation functional $\F$, $\F(\swapOp f,\chi) = \swapOp \F(f,\chi)$.
\end{prop}
The proof is in the Appendix~\ref{sec:AppendixSymetries}.

\section{Statement of the Main Result}

\label{sec:MainRes}

This section is devoted to the statement of the main result. We also present a corollary on the local stability of the spot-solutions and discuss these results. We first introduce the following notion.
% \begin{definition}
%     We say that a positive integer $k\in\mathbb{N}^\ast$ is \textbf{non-Pythagorean} if the only solutions of $k^2_1+k^2_2=k^2$ in $\mathbb{Z}^2$ are $(k_1,k_2)=(\pm k,0)$ and $(0,\pm k)$.
% \end{definition}

\begin{definition}
    We say that a positive integer $k\in\mathbb{N}^\ast$ is \textbf{non-Pythagorean}, if $k^2 \neq l^2+m^2$, for all $l, m \in \dN^\ast$.
\end{definition}

% \begin{thm}
%     There exists a constant $\mathsf{c} > 0$, such that for all $ \gamma > 0 , \sigma_c > 0, \sigma_x >0, \lambda >0$ and any \textbf{non-Pythagorean} wave number $k\in \dN^*$, satisfying
%     \begin{equation*}
%         2\pi k\sigma_x < \mathsf{c} \lambda,
%     \end{equation*}
%     there exists $\sigma^k_\theta >0 $, such that for any $0 < \sigma_\theta < \sigma_\theta^k$, there exist
%     \begin{equation*}
%         \chi^k > 0, \uptau^k_{\Lambda} > \uptau^k_{\Xi}>0, 
%     \end{equation*}
%     such that for any $\tau \geq 0$ with $\tau \neq \uptau^k_\Lambda$ and $\tau \neq \uptau^k_\Xi$, equation \eqref{eq:MainSimpleform} admits two types of local analytic solution branches around $(0,\chi^k)$,
%     \begin{equation*}
%         U \ni s \mapsto \left(f^\Lambda_s, \chi^\Lambda_s\right), \; \; 
%         U \ni s \mapsto \left(f^\Xi_s, \chi^\Xi_s\right),
%     \end{equation*}
%     for some neighborhood $U$ of $0$ in $\dR$, satisfying the following expansions,
%     \begin{align*}
%            f^\Lambda_s &= s\phi^{\vK_1} + o(s),\; \; &\chi^\Lambda_s = \chi^k + \mathrm{a}^\Lambda_\tau s^2 + o(s^2),\\
%            f^\Xi_s &= s\left(\phi^{\vK_1}+ \phi^{\vK_2}\right) + o(s),\; \; &\chi^\Xi_s = \chi^k + \mathrm{a}^\Xi_\tau s^2 + o(s^2).
%     \end{align*}
%     Furthermore, the solutions $f^\Lambda_s$ are constant in the $\x_2$ variable, and the solutions $f^\Xi_s$ are swap symmetric.
% \end{thm}

The following statement is the main result of this paper and synthesizes the results from Theorem~\ref{thm:LyapunovSchmidt} and Theorem~\ref{thm:ExistenceSolCurves}.

\begin{thm}\label{thm:main}
    There is a constant $\mathrm{c}>0$, such that the following holds.
    
    For any $\gamma > 0 , \sigma_c > 0, \sigma_x >0, \lambda >0$ and any non-Pythagorean wave number $k\in \dN^*$, satisfying
    \begin{equation*}
        2\pi k\sigma_x < \mathrm{c}\lambda,
    \end{equation*}
    there exists $\sigma^k_\theta>0$, such that for all $0 < \sigma_\theta < \sigma_\theta^k$, there exist $\chi^k > 0, \uptau^k_{\Lambda} > \uptau^k_{\Xi}>0$, such that, provided that $\tau\geq 0$ and $\tau\neq \uptau^k_{\Lambda}$ and $\tau\neq \uptau^k_{\Xi}$,  equation \eqref{eq:MainSimpleform} admits two distinct local analytic solution branches around $(f,\chi)=(0,\chi^k)$.

    Specifically, there exist a neighborhood $U$ of $0$ in $\dR$ and curves for $s\in U$
    \begin{equation*}
        s \mapsto \left(f^\Lambda_s, \chi^\Lambda_s\right), \; \; 
        s \mapsto \left(f^\Xi_s, \chi^\Xi_s\right),
    \end{equation*}
    such that $\F(f^\Lambda_s,\chi^\Lambda_s) = \F(f^\Xi_s,\chi^\Xi_s)=0$. 
    
    Furthermore, the solutions $f^\Lambda_s$ are constant in the $\x_2$-variable, and the solutions $f^\Xi_s$ are swap symmetric.
\end{thm}

Specifically, the solutions satisfy the following expansions, 
\begin{align*}
    &(f^\Lambda_s,\chi^\Lambda_s)=(s\phi^{\vK_1} + o(s), \chi^k + c^\Lambda_\tau s^2 + o(s^2)),\\
    &(f^\Xi_s, \chi^\Xi_s)= (s(\phi^{\vK_1}+ \phi^{\vK_2}) + o(s)\chi^k + c^\Xi_\tau s^2 + o(s^2)),
\end{align*}
for some $c^\Lambda_\tau$ and $c^\Xi_\tau$ and where $\phi^{\vK_1},\phi^{\vK_2}$ are purely in, respectively, mode $\vK_1 = (k,0)$ and $\vK_2 = (0,k)$ in partial positional Fourier space. They are defined in Section \ref{sec:fourdecomppos}.  From the profiles of $\phi^{\vK_1}$ and $\phi^{\vK_2}$, the $x_2$-invariance of $f^\Lambda$ and the swap symmetry of $f^\Xi$, we refer to the $\Lambda$-solutions as \textbf{lane}-solutions and to the $\Xi$-solutions by \textbf{spot}-solutions.

\begin{cor}
    \label{cor:NonLinStabAnalysis}
    Under the hypotheses of Theorem \ref{thm:main}, when $0\leq \tau < \uptau^k_\Xi$ (resp. $0\leq \tau < \uptau^k_\Lambda$) the bifurcation branch $(s \mapsto (f^\Xi_s, \chi^\Xi_s))$ (resp. $(s \mapsto (f^\Lambda_s, \chi^\Lambda_s))$) is \textbf{super-critical}, and if $\uptau^k_\Xi < \tau$ (resp. $\uptau^k_\Lambda < \tau$) the bifurcation branch is \textbf{sub-critical}.
    
    Furthermore, 
    \begin{equation*}
       \text{ if } k = 1 \text{ and }0\leq \tau < \uptau^k_\Xi,
    \end{equation*}
    for fixed $s$ the solutions $(\tfrac{1}{2\pi} + f^\Xi_s, \chi^\Xi_s)$ are uniformly locally stable for the time-dependent problem \eqref{eq:ftPDE}, in the space $ H^{1+\varepsilon}_0\cap \{\antipodreflOp  f = f\}$, for any $0<\varepsilon<1 $, and the $\Lambda$-branch is unstable.
\end{cor}

The proof of Corollary~\ref{cor:NonLinStabAnalysis} follows directly from the linear stability analysis of Theorem~\ref{thm:StabilitySupercritical}, applying the nonlinear stability result of Henry~\cite[Theorem 5.1.1]{henry2006geometric} (with $\alpha = (1+\varepsilon)/2$,  $X = L^2_0 \cap \{\antipodreflOp f = f\}$) and using that the nonlinearity is controlled as $\|\partial_\theta (\B[f]f)\|_{L^2}\lesssim \|f\|^2_{H^{1+\varepsilon}_0}$ by the Sobolev embedding theorem. The nonlinear instability follows similarly; further details were already given in~\cite{rakodeWit2025} for the instability case.

\begin{figure}[!ht]
    \centering
    \includegraphics[width=.8\linewidth]{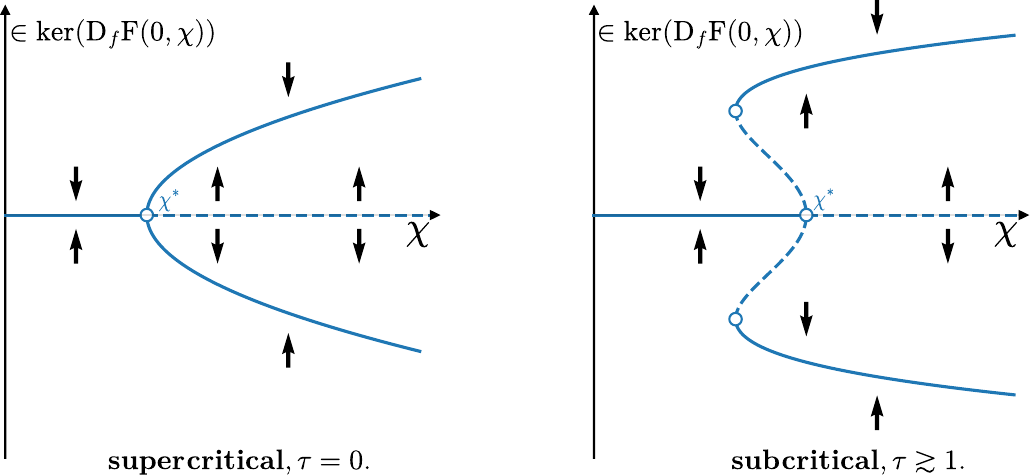}
    \caption{Sub-critical and super-critical bifurcation diagrams.}
    \label{fig:subsupercriticiality}
\end{figure}
 
There are several points of discussion regarding the main results that we clarify.

First, the idea of the proof is to extend the strategy of ~\cite{rakodeWit2025}, that is, we use the decoupling in positional Fourier space and resolvent estimates in terms of $\sigma_x$ and $\sigma_\theta$. In this paper, we obtain a bifurcation sequence for wave numbers $k$ that are non-Pythagorean; this restriction on the wave numbers fixes the dimension of the Lyapunov--Schmidt reduced equations to two. For Pythagorean wave numbers, the dimension is higher, and would require a finer symmetry study when constructing the bifurcation curves. Then, we explicitly compute the singular leading-order of the Lyapunov--Schmidt coefficients in terms of $\sigma_\theta$ as $\sigma_\theta\to 0^+$. Finally, to make the bifurcation analysis tractable in terms of $\tau$, we expand these coefficients in terms of $\sigk =2\pi k\sigma_x/\lambda$ as $\sigk\to 0^+$, hence requiring the inequality $\sigk<\mathrm{c}$, for some $\mathrm{c}$ independent of the parameters and $k = |\vK|$.

Second, due to the translation invariance of the PDE problem~\eqref{eq:fullFstat}, the two types of stationary solutions that are constructed in Theorem~\ref{thm:main}, both yield distinct infinite families of stationary solutions, as translates of the profiles. Moreover, the lane solutions yield a double family by the swap-symmetry. This translation invariance requires imposing the symmetry $\antipodreflOp f = f $, in the stability analysis of Corollary~\ref{cor:NonLinStabAnalysis}, in order to quotient the space $H^{1+\varepsilon}_0$ by translation symmetries.

Third, the sub-criticality of the bifurcation for large $\tau$, was not yet observed in the earlier preceding papers. Specifically, we obtain that when $ 0 \leq \tau< \uptau_\Xi$, both the bifurcation curves of spots and lanes are super-critical, when $\uptau_\Xi < \tau< \uptau_\Lambda$, spot solution curves are sub-critical and lane solution curves are super-critical, and, finally, when $\uptau_\Lambda < \tau$, both are sub-critical. In the sub-critical regime, solutions exist before the bifurcation point, and the dynamical stability relies on a global analysis of the bifurcation curves, as explained in Figure~\ref{fig:subsupercriticiality}. As lane solutions are locally unstable in the super-critical regime, as shown in Corollary~\ref{cor:NonLinStabAnalysis}, the exchange of stability between spot and lane solutions in terms of $\tau$ is not addressed here and is left for future work. Furthermore, in Appendix~\ref{sec:AppendixNumerics} we show numerical evidence of the sub-critical bifurcation for $\tau=0.5$ and the super-criticality for $\tau=0$ for $x_2$-invariant solutions.

    \section{Linearized Operator}
\label{sec:fourdecomppos}
% In this section, we introduce definitions and establish certain properties of the functional equation. We do this in order to apply the Lyapunov--Schmidt method in the following section and to make the integrals involved in the Lyapunov--Schmidt reduction asymptotically explicit. We introduce function names for the Fourier transforms of the terms (Fourier multipliers) in the linearized operator and the functions that form the (co)kernel basis at the bifurcation points. We also show symmetry properties of these Fourier multipliers and basis functions. We also recall the resolvent estimate from the previous paper \cite{rakodeWit2025}. Then, we show that the basis functions indeed form the (co)kernel bases. We also give a Fourier-wise characterisation of the inverse of the linearized operator and finally we show a spectral gap.

% {\color{red} I CHANGED $\sigma_\theta$ to $\sigma_\theta^2$ in the definition of the Operator to ensure the Fredholmness}
% {\color{pink} 
% Argue about Fredholmness at $\sigma_\theta = 0, \chi >0 $
% SOLUTION: Change for $\sigma_\theta^2$}
% {\color{red} when would a problem appear when eigenvalues constant in x, would appear : THE COEFFICIENTS WOULD BE ZERO}
% {\color{red} are your perturbation changing the mass ? should you restrict to the subspace of constant mass equal to zero ?}

In this section, we prepare the analysis of the Lyapunov--Schmidt reduction by establishing preliminary properties of the functional equation. First, we need the following result on the equation operator $\F$. This is a classical result from elliptic theory, see for example~\cite[Chapter III]{kielhofer2006bifurcation}.
\begin{prop}% in the sense of Definition I.2.1 in~\cite{kielhofer2006bifurcation}
    The equation functional $\F : H^2_0 \times \dR \to L^2_0$ is a non-linear Fredholm operator, and admits the linearized operator $D_f\F \in C(H^2_0 \times \dR, L(H^2_0,L^2_0))$, given by,
    \begin{equation*}
        D_f\F(f, \chi )[\phi]  = \sigma_x \Delta_{\x}\phi + \sigma_\theta \partial_{\bdthth}\phi- \lambda v_{\bdth}\cdot\nabla_{\x} \phi - \chi\partial_{\bdth}( \B_\tau[\phi] (\tfrac{1}{2\pi} + f) + \B_\tau[f]\phi).
    \end{equation*}
\end{prop}

We denote by $\LinOp_{\chi}$ the linearized operator around the homogeneous solution, $\LinOp_{\chi}  =D_f\F(\frac{1}{2\pi},\chi)[\cdot]= \sigma_x \Delta_{\x} 
+ \sigma_\theta \partial_{\bdth\bdth} 
- \lambda\, v_{\bdth} \cdot \nabla_{\x} 
- \frac{\chi}{2\pi} \, \partial_{\bdth} \B_\tau[\cdot ]$. The adjoint of $\LinOp_{\chi}$ is explicitly given by $(\LinOp_{\chi})^\star = \sigma_x \Delta_{\x} 
+ \sigma_\theta \partial_{\bdth\bdth}  
+ \lambda\, v_{\bdth} \cdot \nabla_{\x}  
- \frac{\chi}{2\pi} \, (\partial_{\bdth} \B_\tau)^\star[\cdot]
$, with
\[
(\partial_{\bdth} \B_\tau)^\star[\psi] =
\int_0^{2\pi} \Bigl[
v_{\bdth} \cdot \nabla_{\x} \Tilde{c}
+ \tau \Bigl( v^\perp_{\bdth} \cdot (\nabla^2_{\x} \Tilde{c}) v^\perp_{\bdth} 
- v_{\bdth} \cdot (\nabla^2_{\x} \Tilde{c}) v_{\bdth} \Bigr)
\Bigr] \, \mathrm{d}\theta, \quad
\Tilde{c} = (\gamma - \sigma_c \Delta_x)^{-1} \psi.
\]
Here, we extend the operator as
\((\gamma - \sigma_c \Delta_x)^{-1} : L^2(\Torus^2_1)\to H^2(\Torus^2_1)\)
to
\((\gamma - \sigma_c \Delta_x)^{-1} : L^2(\Torus_{2\pi}, L^2(\Torus^2_1)) \to L^2(\Torus_{2\pi}, H^2(\Torus^2_1))\).
Finally, using the Hilbertian structure, we have
\[
\coker(\LinOp_\chi) = \ker((\LinOp_\chi)^{\star}) = (\Range(\LinOp_\chi))^\perp.
\]
% We will denote by $\LinOp_{\chi}$ the linearized operator around the homogeneous solution,
% \begin{equation*}
%     \LinOp_{\chi} \phi  = \sigma_x \Delta_{\x}\phi + \sigma_\theta \partial_{\bdthth}\phi- \lambda v_{\bdth }\cdot\nabla_{\x} \phi - \tfrac{\chi}{2\pi}\partial_{\bdth} \B_\tau[\phi].
% \end{equation*}
% Its adjoint is explicitly given by,
% \begin{equation*}
%     (\LinOp_{\chi})^\star \psi = \sigma_x \Delta_{\x}\psi +  \sigma_\theta \partial_{\bdthth}\psi + \lambda v_{\bdth} \cdot\nabla_{\x} \psi - \tfrac{\chi}{2\pi}(\partial_{\bdth} \B_\tau)^\star[\psi],
% \end{equation*}
% where,
% \begin{align*}
%     (\partial_{\bdth} \B_\tau)^\star[\psi] &= \int^{2\pi}_0 v_{\bdth} \nabla_{\x} \Tilde{c} + \tau (v^\perp_{\bdth} (\nabla^2_{\x} \Tilde{c}) v^\perp_{\bdth} - v_{\bdth} (\nabla^2_{\x} \Tilde{c}) v_{\bdth}  )\dd\theta,\\
%     \Tilde{c} &= (\gamma - \sigma_c \Delta_x)^{-1}\psi,
% \end{align*}
% and where we extend: $(\gamma - \sigma_c \Delta_x)^{-1} : L^2_x \to H^2_x$ to $ (\gamma - \sigma_c \Delta_x)^{-1}:L^2_{x, \theta}\to L^2_\theta(H^2_x)$.

% Finally, using the Hilbertian structure, we have that,
% \begin{equation*}
%     \coker(\LinOp_\chi) = \ker((\LinOp_\chi)^{\star}) = (\Range(\LinOp_\chi))^\perp.
% \end{equation*}

\subsection{Partial Fourier decomposition}

In this paper, we use the following Fourier transform conventions. Given a function $g \in L^2_0$, the partial Fourier transform of $g$ in the positional variable, is defined for $\vK \in \dZ^2$ as,
\begin{equation*}
    \hat{g}_\vK(\theta) \defeq \mathcal{F}_x\{g\}_\vK (\theta)  =  \int g(x, \theta) e^{- i 2\pi 
 \, \vK\cdot x}\dd x.
\end{equation*}
% \begin{equation*}
%     \frac{1}{L} \int f g^* \dd y = \sum_{n \in \dZ} \hat{f}_n \hat{g}^*_n.
% \end{equation*}

% Furthermore, if $f$ and $g$ are real valued, then
% \begin{equation*}
%     \frac{1}{L}\int fg\dd y = \sum_{n \in \dZ} \hat{f}_n \hat{g}_{-n}
% \end{equation*}

% Let $\psi$ and $\phi$ be real valued, then for $L=1$,
% \begin{equation*}
%     \langle \phi, \psi \rangle_{x,\theta} = 2\sum_{k\geq 1} \Real \langle \hat{\phi}_k, \hat{\psi}^*_k \rangle_\theta + \langle \hat{\phi}_0, \hat{\psi}_0 \rangle_\theta
% \end{equation*}
This implies the Parseval identity for $f,g \in L^2_0$, 
\begin{equation*}
    \left\langle f, g\right\rangle \defeq \int f g \dd x \dd \theta = \sum_{\vK\in \dZ^2} \int \hat{f}^*_\vK \hat{g}_\vK \dd \theta,
\end{equation*}
where $\hat{f}_\vK^\ast, \hat{g}_\vK \in L^2_{\theta}(\dC)$ and $\hat{f}_0 , \hat{g}_0 \in L^2_{\theta;0}(\dC)$.
% \begin{definition}
%     We define the central reflection operator $R : \dL^2_{x,\theta}\to \dL^2_{x,\theta}$, defined as,
%     \begin{equation}
%         (\reflOp f)(x,\theta) = f(-x,\theta+\pi)
%     \end{equation}
%     \begin{equation}
%         \reflOp = \reflOp^\star = \reflOp^{-1}
%     \end{equation}
% \end{definition}

We introduce the following functions depending on $\theta$, which correspond to the Fourier multipliers in the $x$-variable of the terms in the operator $\mathrm{L}_\chi$.

The Fourier multiplier of the inverse elliptic operator is $\cEll = \gamma + 4\pi^2 |\vK|^2\sigma_c$. That is, if $-\sigma_c\Delta_\x c+\gamma c=\rho$, then $\hat{c}_\vK =\hat{\rho}_\vK /\mathcal{E}^c_{|\vK|}$. The Fourier multipliers of the operators $\B_\tau$ and $ - \sigma_x \Delta_\x + \lambda v \cdot \nabla_\x$ are given by,
\begin{align*}
    \fB_{\vK}(\theta) &= \frac{1}{\cEll}\left(2\pi i v^\perp(\theta)\cdot \vK  -  4\pi^2\tau v^\perp(\theta) \cdot (\vK \otimes \vK) v(\theta) \right),\\
    \fM_{\vK}(\theta) &= 4\pi^2 \sigma_x |\vK|^2 +  i 2\pi \lambda v(\theta)\cdot \vK.
\end{align*}
We naturally observe the following properties,
\begin{equation}
    \label{prop:ConjugateBM}
    (\fB_{\vK})^* = \fB_{-\vK}, \; \; (\fM_{\vK})^* = \fM_{-\vK},
\end{equation}
as we consider real-valued operators. 

We denote the restriction of the operator $\LinOp_\chi$ for each positional Fourier mode $\vK \in \dZ^2$ by $\LinOpFk^\vK_\chi$ acting on functions depending on $\theta$,
\begin{equation*}
    \LinOpFk_\chi^\vK \defeq \mathcal{F}_x\{\LinOp_\chi\}_\vK =  -\fM_\vK + \sigma_\theta \partial_{\thth} - \chi \partial_\theta \fB_{\vK} \int \cdot\; \dd \theta.
\end{equation*}
We also define, for $k \in \dN^*$, the \textbf{$k$-rescaled constants},
\begin{equation}
    \label{eq:krescaledConstants}
    \lmbk=\ 2\pi\lambda |k|, \hspace{3em} \sigk= \ 2\pi\frac{\sigma_x}{\lambda}|k|, \hspace{3em} \tauk= \ 2\pi \tau |k|.
\end{equation}
Then, for the specific wave numbers $\vK_1=(k,0),\vK_2=(0,k)$, we have explicitly the Table~\ref{tb:FourierMultipliers}.
\begin{table}[!ht]
\caption{Fourier Multipliers}
\centering
\label{tb:FourierMultipliers}
\renewcommand{\arraystretch}{1.6}
\begin{tabular}{|c||c|c|}
\hline
 & $\vK_1=(k,0)$ & $\vK_2=(0,k)$ \\
\hline
$\hat{\mathsf{B}}_{\vK}$ 
&
$\displaystyle \frac{2\pi k}{\cEll}\!\left(-i\sin\theta+\tauk\tfrac{\sin 2\theta}{2}\right)$
&
$\displaystyle \frac{2\pi k}{\cEll}\!\left(i\cos\theta-\tauk\tfrac{\sin 2\theta}{2}\right)$
\\
\hline
$\hat{\mathsf{M}}_{\vK}$
&
$\displaystyle \lmbk(\sigk+i\cos\theta)$
&
$\displaystyle \lmbk(\sigk+i\sin\theta)$
\\
\hline
$\partial_\theta\hat{\mathsf{B}}_{\vK}$
&
$\displaystyle \frac{2\pi k}{\cEll}\!\left(-i\cos\theta+\tauk\cos 2\theta\right)$
&
$\displaystyle \frac{2\pi k}{\cEll}\!\left(-i\sin\theta-\tauk\cos 2\theta\right)$
\\
\hline
$\partial_\theta\hat{\mathsf{M}}_{\vK}$
&
$\displaystyle -i\lmbk \sin\theta$
&
$\displaystyle i\lmbk\cos\theta$
\\
\hline
% $\hat{\mathsf{B}}_{-\vK}$ 
% &
% $\displaystyle \frac{2\pi k}{\cEll}\!\left(i\sin\theta+\tauk\tfrac{\sin 2\theta}{2}\right)$
% &
% $\displaystyle \frac{2\pi k}{\cEll}\!\left(-i\cos\theta-\tauk \tfrac{\sin 2\theta}{2} \right)$
% \\
% \hline
% $\hat{\mathsf{M}}_{-\vK}$
% &
% $\displaystyle \lmbk(\sigk-i\cos\theta)$
% &
% $\displaystyle \lmbk(\sigk-i\sin\theta)$
% \\
% \hline
% $\partial_\theta\hat{\mathsf{B}}_{-\vK}$
% &
% $\displaystyle \frac{2\pi k}{\cEll}\!\left(i\cos\theta+\tauk\cos 2\theta\right)$
% &
% $\displaystyle \frac{2\pi k}{\cEll}\!\left(i\sin\theta-\tauk\cos 2\theta\right)$
% \\
% \hline
% $\partial_\theta\hat{\mathsf{M}}_{-\vK}$
% & $\displaystyle i\lmbk \sin\theta$
% &
% $\displaystyle -i\lmbk\cos\theta$
% \\
% \hline
\end{tabular}
\end{table}

% We will also extends the above definitions for $\vK \in \dR^2 \setminus\{0\}$ or $k \in \dR_+$.

For computing closed-form integrals associated with the Lyapunov--Schmidt reduction, we use the following constants,
\begin{equation}
    \label{eq:defekzink}
    \Azink \defeq -(2\sigk^2+1) + \sqrt{(2\sigk^2+1)^2-1},\; \; \Aek \defeq \left(\sqrt{(2\sigk^2+1)^2-1}\right)^{-1} 
\end{equation}

We recall the following resolvent estimate, which was proven in Lemma 5.2 in~\cite{rakodeWit2025}.
\begin{prop}
    \label{prop:ResolventProperties}
    For $\vL \in \dZ^2\setminus\{0\}$ and $\sigma_\theta \geq 0$, we define the operator $\fOpR^\vL_{\sigma_\theta} : L^2_{\theta}(\dC) \to H^2_{\theta}(\dC)$, as the resolvent,
    \begin{equation*}
        \fOpR^\vL_{\sigma_\theta} = \left(-\hat{\mathsf{M}}_{\vL} + \sigma_\theta \partial_{\thth}\right)^{-1}.
    \end{equation*}
    This resolvent operator satisfies the following complex conjugation property. For all $g \in L^2_{\theta}(\dC)$,
    \begin{equation}
        \label{prp:ConjugateResolvent}
         (\fOpR^\vL_{\sigma_\theta} g)^*= \fOpR^{-\vL}_{\sigma_\theta} g^*.
    \end{equation}
    Furthermore, for regular functions $g\in H^2_\theta(\dC)$, we have the following uniform convergence. That is, for any $\sigma_1 \geq \sigma_2 \geq 0$,
    \begin{equation}
        \label{est:ResolventContinuity}
        \sup_{\vL \in \dZ^2 \setminus\{0\}} \|(\fOpR^\vL_{\sigma_1} - \fOpR^\vL_{\sigma_2})g\|_{L^2}\leq C_{\sigma_2,\|g\|_{H^2_\theta}} |\sigma_1 - \sigma_2|.
    \end{equation}
\end{prop}

The proof directly follows from an energy estimate using the regularity of the function $g$. The conjugation property follows by conjugating the equation.

We now define the following smooth functions depending on $\theta$,
\begin{equation*}
    \fU^\ell_{\sigma_\theta}(\theta)  = \fOpR^\ell_{\sigma_\theta} \partial_\theta \fB_\ell, \ \ \fV^\ell_{\sigma_\theta}(\theta)  = -\fOpR^{-\ell}_{\sigma_\theta} 1.
\end{equation*}
These functions satisfy, by \eqref{prp:ConjugateResolvent} and \eqref{prop:ConjugateBM}, the following identities,
\begin{equation*}    
    (\fU^\ell_{\sigma_\theta})^* = \fU^{-\ell}_{\sigma_\theta}, \ \ 
    (\fV^\ell_{\sigma_\theta})^* =\fV^{-\ell}_{\sigma_\theta}.
\end{equation*}

\begin{prop}
    \label{prop:RotationUKukomega}
    Let $\vK \in \dZ^2\setminus\{0\}$, $ \vK= |\vK| (\cos(\omega),\sin(\omega))$, and denote $\vK_1 = (|\vK|, 0)$. 
    Then, the following identities hold,
    \begin{align}
        \fB_\vK &= \fB_{\vK_1} (\cdot -\omega),\\
        \fM_\vK &= \fM_{\vK_1} (\cdot -\omega),\\ 
        \fU^\vK_{\sigma_\theta} &= \fU^{\vK_1}_{\sigma_\theta} (\cdot - \omega).
    \end{align}
\end{prop}

\begin{proof}
    We first check the property on $\fB_\vK$,
    \begin{equation*}
            \fB_\vK(\theta) = \frac{1}{\cEll}( 2\pi i |\vK| v^\perp(\theta)\cdot v(\omega) -  4\pi^2\tau |\vK|^2 (v^\perp(\theta) \cdot (v(\omega) \otimes v(\omega)) v(\theta)).
    \end{equation*}
    Using the following trigonometric identities,
    \begin{equation*}
         v^\perp(\theta)\cdot v(\omega) = -\sin(\theta-\omega), \ \ v^\perp(\theta) \cdot (v(\omega) \otimes v(\omega)) v(\theta) = -\frac{1}{2} \sin(2(\theta-\omega)),
    \end{equation*}
    %Some details
    % \begin{equation*}
    %      v^\perp(\theta)\cdot v(\omega) = -\sin(\theta)\cos(\omega) + \cos(\theta)\sin(\omega) = -\sin(\theta-\omega)
    % \end{equation*}

    % \begin{align*}
    %      v^\perp(\theta) \cdot (v(\omega) \otimes v(\omega)) v(\theta) &=  \frac{1}{2}\sin(2\theta)(-\cos^2(\omega) + \sin^2(\omega)) \\
    %      & + (\cos^2(\theta) - \sin^2(\theta))\cos(\omega)\sin(\omega)\\
    %      &=  -\frac{1}{2}\sin(2\theta)\cos(2\omega) +\frac{1}{2}\cos(2\theta) \sin(2\omega)\\
    %      &=  -\frac{1}{2} \sin(2(\theta-\omega)) \\
    % \end{align*}
    we can write,
    \begin{equation*}
            \fB_\vK(\theta) = \fB_{\vK_1}(\theta - \omega),\; \; \text{ and } \; \; \partial_\theta\fB_\vK(\theta) = \partial_\theta\fB_{\vK_1}(\theta - \omega).
    \end{equation*}
    Similarly, using that $v(\theta)\cdot v(\omega) = \cos(\theta - \omega)$, we obtain,
    \begin{equation}
        \fM_\vK(\theta) = \fM_{\vK_1}(\theta - \omega).
    \end{equation}
    This implies that $\fU^\vK_{\sigma_\theta}$ and $ \fU^{\vK_1}_{\sigma_\theta}(\cdot - \omega)$ are solutions of the same elliptic equation. Hence, by uniqueness, we conclude the proof.
\end{proof}

In the case $\sigma_\theta = 0$, we have explicitly,
\begin{equation*}
    \fU^{\vK_1}_{0} = \fOpR^{\vK_1}_{0} \partial_\theta \fB_{\vK_1}  =  \frac{-\partial_\theta \fB_{\vK_1}}{\fM_{\vK_1}},\;\;
    \fV^{\vK_1}_{0}  = -\fOpR^{-\vK_1}_{0} 1 =  \frac{1}{\fM_{-\vK_1}}.
\end{equation*}
We note that the operators $\LinOpFk^{\vK}_{\chi}$ for $\sigma_\theta=0$ allow for explicit computations, which we leverage in Sections \ref{sec:BirfucationCurves} and Appendix \ref{sec:AppendixIntegrals}. 

The full operator $\mathrm{L}_\chi$ for $\sigma_\theta=0$ is not a Fredholm operator, as its kernel contains $\{\phi\in H^2_0:\phi(x,\theta) = a(\theta)\}$. However, the operator becomes Fredholm whenever $\sigma_\theta>0$, allowing us to extend the explicit information on the operator for $\sigma_\theta=0$, by resolvent estimates.

We use the following Fourier-description of the antipodal parities of Definition~\ref{def:antipodaleven}.
\begin{definition}
    We say that $u\in L^2_\theta(\dC)$ is \textbf{central-conjugate-symmetric}, if \begin{equation*}
        u\left(\theta\right) = \left(u\left(\theta+ \pi\right)\right)^*,
    \end{equation*}
    respectively, \textbf{central-conjugate-antisymmetric}, if
    \begin{equation*}
        u\left(\theta\right) = -\left(u\left(\theta+\pi \right) \right)^*.
    \end{equation*}
    
    We denote \textbf{c.c.s} and \textbf{c.c.a}, respectively.
\end{definition}

\begin{prop}
\label{prop:equivalenceAE-CCS}
Given a \textbf{real} function  $f\in L^2_0$, then f is \textbf{antipodal even }(resp. \textbf{antipodal odd}), if and only if its Fourier modes are \textbf{central-conjugate-symmetric} (resp. \textbf{central-conjugate-antisymmetric}).
\end{prop}

The proof follows directly by expressing Definition~\ref{def:antipodaleven} in Fourier modes.

% \begin{proof}
%     The Fourier modes of an antipodal even function satisfy,
%     \begin{equation}
%         \hat{f}_k(\theta) = \hat{f}_{-k}(\theta+\pi) = \hat{f}^*_{k}(\theta+\pi),
%     \end{equation}
%     where the last equation holds from the realness of the function. The same proof holds for central-conjugate-antisymmetric.
% \end{proof}

\begin{prop}
    \label{prop:PreservationOfCentralConjSym}
    The following operators preserve the central-conjugate-symmetry: $\fOpR^\vL_{\sigma_\theta}$, $\partial_\theta$, $\int \dd \theta$, $(\partial_{\thth})^{-1}$, $\cdot^*$.
    Multiplying by a purely imaginary number inverts the central-conjugate-symmetry. Furthermore, suppose that $g, f\in L^2_0$ are real-valued, $g$ is antipodal-odd and $f$ is antipodal-even.
    Then, we have
    \begin{equation}
        \int f g \dd x\dd \theta = 0.
    \end{equation}
    In particular $\partial_\theta \fB_{\vK}$ is central-conjugate-symmetric, and thus $\fU^{\vK}_\sigma$ is central-conjugate-symmetric.
\end{prop}

\begin{proof}
    We only prove the case of $\fOpR^\vK_{\sigma_\theta}$.
    Suppose that $g$ is central-conjugate-symmetric. We denote $u = \fOpR^\vK_{\sigma_\theta} g$, and we introduce $v(\theta) = (u\left(\theta + \pi\right))^*$, so that $u$ and $v$ are solutions of the equations,
    \begin{align*}
        \left(-\fM_\vK + \sigma_\theta\partial_{\thth} \right)u &= g,\\
        \left(-\fM_\vK^*\left( \cdot + \pi\right) + \sigma_\theta\partial_{\thth} \right)v&= \left(g\left(\cdot + \pi\right)\right)^*.
    \end{align*}
    Then, using that $\fM_\vK$ and $g$ are \textbf{c.c.s}, we obtain that,
    \begin{equation}
        \left(-\fM_\vK + \sigma_\theta\partial_{\thth} \right)v = g.
    \end{equation}
    Hence, $u$ and $v$ are solutions of the same elliptic equation, so that, by uniqueness, we obtain that $u=v$.
    
    In the case where $g$ is central-conjugate-antisymmetric, we note that $i g$ is \textbf{c.c.s}, applying the above property and using that $\fOpR^\vL_{\sigma_\theta}$ is a linear operator over the complex field we obtain that $i\fOpR^\vL_{\sigma_\theta} g$ is \textbf{c.c.s}, multiplying by $-i$, we get $\fOpR^\vL_{\sigma_\theta} g$ and change the central-conjugate-symmmetry, so that $\fOpR^\vL_{\sigma_\theta} g$ is \textbf{c.c.a}.
\end{proof}

% For a whole number $k\geq 1$, let $\vK_1 = (k,0)$ and $\vK_2 = (0,k)$ , we then define the following smooth real functions $\phi^{\vK_1}_{\sigma_\theta}, \phi^{\vK_2}_{\sigma_\theta}, \psi^{\vK_1}_{\sigma_\theta},  \psi^{\vK_2}_{\sigma_\theta}, \phi^{\vK_1,\bis}_{\sigma_\theta}, \phi^{\vK_2, \bis}_{\sigma_\theta}, \psi^{\vK_1,\bis}_{\sigma_\theta},  \psi^{\vK_2, \bis}_{\sigma_\theta}$ in Fourier, let $\vM \in \dZ^2$
% \begin{align}
%     \mathcal{F}_x\{\phi^{\vK_1}_{\sigma}\}_\vM &=  \fU^{\vK_1}_\sigma \delta_{\vK_1 \vM}+\fU^{-\vK_1}_\sigma \delta_{-\vK_1 \vM}\\
%     \mathcal{F}_x\{\phi^{\vK_2}_{\sigma}\}_\vM &=  \fU^{\vK_2}_\sigma \delta_{\vK_2 \vM}+\fU^{-\vK_2}_\sigma \delta_{-\vK_2 \vM}\\
%     \mathcal{F}_x\{\psi^{\vK_1}_{\sigma}\}_\vM &= \fV^{\vK_1}_{\sigma} \delta_{\vK_1 \vM}+\fV^{-\vK_1}_{\sigma} \delta_{-\vK_1 \vM}\\
%     \mathcal{F}_x\{\psi^{\vK_2}_{\sigma}\}_\vM &= \fV^{\vK_2}_{\sigma} \delta_{\vK_2 \vM}+\fV^{-\vK_2}_{\sigma} \delta_{-\vK_2 \vM}\\
%     \mathcal{F}_x\{\phi^{\bis\vK_1}_{\sigma}\}_\vM &=  i\fU^{\vK_1}_\sigma \delta_{\vK_1 \vM}-i\fU^{-\vK_1}_\sigma \delta_{-\vK_1 \vM}\\
%     \mathcal{F}_x\{\phi^{\bis\vK_2}_{\sigma}\}_\vM &=  i\fU^{\vK_2}_\sigma \delta_{\vK_2 \vM}-i\fU^{-\vK_2}_\sigma \delta_{-\vK_2 \vM}\\
%     \mathcal{F}_x\{\psi^{\bis \vK_1}_{\sigma}\}_\vM &= i\fV^{\vK_1}_{\sigma} \delta_{\vK_1 \vM}-i\fV^{-\vK_1}_{\sigma} \delta_{-\vK_1 \vM}\\
%     \mathcal{F}_x\{\psi^{\bis\vK_2}_{\sigma}\}_\vM &= i\fV^{\vK_2}_{\sigma} \delta_{\vK_2 \vM}-i\fV^{-\vK_2}_{\sigma} \delta_{-\vK_2 \vM}.
% \end{align}

We introduce the following functions. They are the associated basis functions of the kernel and cokernel of the linearized operator at the $k$-th bifurcation point. 
\begin{center}
    \begin{minipage}{.8\textwidth}
        For a whole number $k\geq 1$, and $\vK_1 = (k,0)$ and $\vK_2 = (0,k)$, we define the following smooth real-valued functions, $\phi^{\vK_1}_{\sigma_\theta}, \phi^{\vK_2}_{\sigma_\theta}, \psi^{\vK_1}_{\sigma_\theta},  \psi^{\vK_2}_{\sigma_\theta}, \phi^{\vK_1,\bis}_{\sigma_\theta}, \phi^{\vK_2, \bis}_{\sigma_\theta}, \psi^{\vK_1,\bis}_{\sigma_\theta},  \psi^{\vK_2, \bis}_{\sigma_\theta}$ in Fourier, where for $\vM \in \dZ^2$,
        \begin{align}
            \mathcal{F}_x\{\phi^{\vK_j}_{\sigma}\}_\vM &=  \fU^{\vK_j}_\sigma \delta_{\vK_j \vM}+\fU^{-\vK_j}_\sigma \delta_{-\vK_j \vM}\\
            \mathcal{F}_x\{\psi^{\vK_j}_{\sigma}\}_\vM &= \fV^{\vK_j}_{\sigma} \delta_{\vK_j \vM}+\fV^{-\vK_j}_{\sigma} \delta_{-\vK_j \vM}.
        \end{align}
        The $\bis$ corresponds to translations of the previous functions, which in Fourier writes as,
        \begin{align}
            \mathcal{F}_x\{\phi^{\bis\vK_j}_{\sigma}\}_\vM &=  i\fU^{\vK_j}_\sigma \delta_{\vK_j \vM}-i\fU^{-\vK_j}_\sigma \delta_{-\vK_j \vM}\\
            \mathcal{F}_x\{\psi^{\bis \vK_j}_{\sigma}\}_\vM &= i\fV^{\vK_j}_{\sigma} \delta_{\vK_j \vM}-i\fV^{-\vK_j}_{\sigma} \delta_{-\vK_j \vM}.
        \end{align}
    \end{minipage}
\end{center}

\begin{prop}
    \label{prop:swapRelation}
    For the swap operator, we have
    \begin{equation}
        \swapOp \phi^{\vK_1}_{\sigma_\theta} = \phi^{\vK_2}_{\sigma_\theta}, \; \; \swapOp \psi^{\vK_1}_{\sigma_\theta} =  \psi^{\vK_2}_{\sigma_\theta}.
    \end{equation}
    The same identities hold for the $\bis$-functions.
\end{prop}

\begin{proof}
    We start by noting the following representation of the swap operator in Fourier space. Denote $\mathbf{S}$ for the rotation matrix (not to be confused with swap operator $\swapOp$),
    \begin{equation*}
        \mathbf{S} = \begin{pmatrix}
            0 & 1\\
            1 & 0 
        \end{pmatrix},
    \end{equation*}
    the following identity holds in Fourier space for the swap operator,
    \begin{align*}
        \mathcal{F}_x\{\swapOp g\}_{\vM} &= \int g(-x_2, -x_1, -\theta - \tfrac{\pi}{2}) e^{-i2\pi \vM\cdot x}\dd x,\\
        &=  \int g(x_1, x_2, -\theta - \tfrac{\pi}{2}) e^{i2\pi (\mathbf{S} \vM)\cdot x}\dd x,\\
        &= \mathcal{F}_x\{ g\}_{- (\mathbf{S}\vM)}\left(-\theta - \tfrac{\pi}{2}\right).
    \end{align*}
    Therefore, we have that,
    \begin{align*}
        \mathcal{F}_x\{\swapOp \phi^{\vK_1}\}_\vM & =  \fU^{\vK_1} \left(-\theta - \tfrac{\pi}{2}\right) \delta_{- (\mathbf{S}\vM),\vK_1 }+\fU^{-\vK_1}\left(-\theta - \tfrac{\pi}{2}\right)\delta_{-\vK_1, - (\mathbf{S}\vM)},\\
        & =  \fU^{\vK_1} \left(-\theta - \tfrac{\pi}{2}\right) \delta_{-\vK_2 \vM}+\fU^{-\vK_1}\left(-\theta - \tfrac{\pi}{2}\right)\delta_{\vK_2, \vM}.
    \end{align*}
    From Proposition~\ref{prop:RotationUKukomega} and the definition of $\fB$ and $\fM$ in Table~\ref{tb:FourierMultipliers}, we obtain that,
    \begin{align*}
        \partial_\theta \fB_{\vK_1}(-\theta-\pi/2)=& \ \partial_\theta \fB_{\vK_2}(-\theta)= \ \partial_\theta \fB_{-\vK_2}(\theta),\\
        \fM_{\vK_1}(-\theta-\pi/2)=& \ \fM_{\vK_2}(-\theta) = \ \fM_{-\vK_2}(\theta).
    \end{align*}
    Hence, from the uniqueness of the elliptic equations, $\fU^{\vK_1}(-\theta-\pi/2)= \fU^{-\vK_2}(\theta)$. Plugging this into the previous identity leads to the required result. The same proof holds for the $\psi^{\vK}$-functions.
\end{proof}

\subsection{Spectral Analysis}

% {\color{pink} add space of defintion, check also in detail equality between the dimension in the statement and \eqref{eq:kernelspaninfourier})}
We now state and prove the main properties of the linearized operator of the functional $\F$ around the constant solution, serving as the starting point of the bifurcation analysis. Specifically, we construct the bifurcation sequence at non-Pythagorean wave-numbers, we obtain the explicit basis of the kernel, and compute the inverse operator in partial Fourier space.

\begin{thm}
    \label{thm:LinOpKern}
    %For any $\sigma_\theta>0$, the operator $\mathsf{L}_{\chi}$ is Fredholm of index zero.
    For any $\sigma_x,\lambda >0$ fixed, for every whole number $k\geq 1$, there exists $\sigma^k_\theta>0$, such that for any $0<\sigma_\theta<\sigma^k_\theta$, there exist $\chi^k > 0$, such that denoting $\LinOp_{\chi^k} = \LinOp_k$, the following dimension identity holds,
    \begin{equation*}
        \dim \ker \LinOp_k = \dim \ker (\LinOp_k)^\star = \dim \coker \LinOp_k = \card \{ \vK \in \dZ^2 \setminus\{0\} , |\vK|^2 = k^2\}.
    \end{equation*}
    In particular, for non-Pythagorean whole numbers $k$, the kernel and cokernel are given by the subspaces generated by the functions,
    \begin{align*}
        \ker \LinOp_k &= \Span \{\phi^{\vK_1}_{\sigma_\theta},  \phi^{\vK_2}_{\sigma_\theta}, \phi^{\bis\vK_1}_{\sigma_\theta}, \phi^{\bis\vK_2}_{\sigma_\theta}\},\\
        \coker \mathrm{L}_k&= \ker (\LinOp_k)^* = \Span \{\psi^{\vK_1}_{\sigma_\theta},\psi^{\vK_2}_{\sigma_\theta},  \psi^{\bis\vK_1}_{\sigma_\theta}, \psi^{\bis\vK_2}_{\sigma_\theta}\}.
    \end{align*}
\end{thm}

In the rest of the analysis, $k$ denotes a non-Pythagorean wave number, and $\LinOp_k$ the linearized operator at the bifurcation point $\chi^k$ associated with some $\sigma_\theta$ sufficiently small.

% We will call $(k,\sigma_\theta, \chi^k_{\sigma_\theta})$ a \textbf{non-Pythagorean bifurcation triple}, if:
% \begin{center}
%     \begin{minipage}{0.8\textwidth} 
%      $k$ is a non-Pythagorean integer, \textit{i.e} $k^2\neq a^2+b^2$ for all $a,b\in \dN^*$; $\sigma_\theta \in (0,\sigma_\theta^k)$, and $\chi_{\sigma_\theta}^k$ is the associated bifurcation sensitivity $\chi$ as given in Theorem~\ref{thm:LinOpKern}.
%      \end{minipage}
%      %its square can not be expressed as the sum of the squares of non-zero whole numbers
% \end{center}

\begin{proof}
Let $\phi \in \ker(\LinOp_{\chi})$.
Then, in Fourier, we have that, for any $\vK \in \dZ^2 \setminus\{0\}$,
\begin{equation*}
   \mathcal{F}_x\{\LinOp_{\chi}\phi\}_\vK = \left(-\hat{\mathsf{M}}_\vK + \sigma_\theta \partial_{\thth}\right) \hat{\phi}_\vK - \frac{\chi}{2\pi} \partial_\theta\fB_\vK \int \hat{\phi}_\vK \dd \theta = 0
\end{equation*}
Using the invertibility of the elliptic operator (the resolvent as in Prop.~\ref{prop:ResolventProperties}), we obtain equivalently that, 
\begin{equation*}
    \hat{\phi}_\vK =   \frac{\chi}{2\pi} \left(-\hat{\mathsf{M}}_\vK +\sigma_\theta\partial_{\thth}\right)^{-1}\partial_\theta \fB_\vK\int \hat{\phi}_\vK \dd \theta 
\end{equation*}
From the definition of $\fU^\vK_{\sigma_\theta}$, the equation rewrites as,
\begin{equation}
    \label{eq:EigenValdef}
    \hat{\phi}_\vK  =   \frac{\chi}{2\pi} \fU^\vK_{\sigma_\theta} \int \hat{\phi}_\vK \dd \theta. 
\end{equation}
Integrating \eqref{eq:EigenValdef}, we obtain that either $\int \hat{\phi}_\vK  \dd \theta = 0 $ and thus $\hat{\phi}_\vK  = 0$, or that $\tfrac{\chi}{2\pi}\int \fU^\vK_{\sigma_\theta} \dd\theta = 1$. From Proposition~\ref{prop:RotationUKukomega}, this is equivalent to exchanging $\vK$ with $\vK_1$, $\frac{\chi}{2\pi}\int \fU^{\vK_1}_{\sigma_\theta} d\theta = 1$. As $\fU^{\vK_1}_{\sigma_\theta}$ is central-conjugate-symmetric, from Proposition~\ref{prop:PreservationOfCentralConjSym}, we have $\int \fU^{\vK_1}_{\sigma_\theta}\dd \theta \in \dR$, for any $\sigma_\theta\geq 0$.

We now construct our sequence of bifurcation points. To this purpose, we fix a wave number $k\geq 1$, and using the explicit form of the inviscid inverse operator, we obtain,
\begin{equation*}
        \int_0^{2\pi} \fU^{\vK_1}_{0} \mathrm{d}\theta = \frac{4\pi^2 k \Aek}{\cEll \lmbk} (1-|\Azink| + 2\sigk |\Azink| \tauk),
\end{equation*}using the notations as in \eqref{eq:krescaledConstants} and \eqref{eq:defekzink}.
%%% Computation in the Supplement

From the explicit formula, we note that the function $k \mapsto \int \fU^{\vK_1}_{0}\mathrm{d}\theta$ is even, and non-negative strictly decreasing on $(0,+\infty)$, converging to zero at $\pm\infty$. Now defining $\sigma^k_\theta>0$ as the largest $\sigma$ such that, for any $0< \sigma < \sigma^k_\theta$,
\begin{equation*}
    \left|\int_0^{2\pi} \fU^{\vK_1}_{0} \mathrm{d}\theta  - \int_0^{2\pi} \fU^{\vK_1}_{\sigma} \mathrm{d}\theta  \right| < \left|\int_0^{2\pi} \fU^{\vK_1}_{0} \mathrm{d}\theta \right|,
\end{equation*}
then for any $0 < \sigma_\theta < \sigma^k_\theta$, we can define
\begin{equation*}
    \tag{$\chi^k$}
    \chi^k(\sigma_\theta) \defeq \frac{2\pi}{\int_0^{2\pi} \fU^{\vK_1}_{\sigma_\theta} \mathrm{d}\theta}.
\end{equation*}
This together with equation~\eqref{eq:EigenValdef} ensure that the restriction in Fourier of operator $\LinOp_{\chi^k}$ has a non-empty kernel and is given by,
\begin{equation}
    \label{eq:kernelspaninfourier}
    \Span \{\fU^\vK_{\sigma}, i \fU^\vK_{\sigma} :\vK \in \dZ^2,|\vK| = k\}.
\end{equation}

The operator {\small $ (-\hat{\mathsf{M}}_\vK + \sigma \partial_{\thth}  )$}, is Fredholm and of index zero, as it admits a compact inverse, and {\small $\LinOpFk^\vK_{\chi^k} - (-\hat{\mathsf{M}}_\vK + \sigma \partial_{\thth})$} is compact. Therefore, the index of $\LinOpFk^{\vK}_{\chi^k}$ is zero and thus the dimension of the cokernel is equal to the dimension of the kernel, with cokernel functions satisfying,
\begin{equation*}
    \hat{\psi}_\vK = \frac{\chi^k}{2\pi}\left(-\hat{\mathsf{M}}_{-\vK} + \sigma_\theta \partial_{\thth} \right)^{-1} \left(\int \partial_\theta \mathsf{B}_{-\vK} \hat{\psi}_\vK \dd \theta\right).
\end{equation*}
So that, the cokernel of  the restriction in Fourier of operator $\LinOpFk^{\vK}_{\chi^k}$ is given by,

\begin{equation}
    \Span \left\{\fV^\vK_{\sigma}, i \fV^\vK_{\sigma}:\vK \in \dZ^2,|\vK| = k\right\}.
\end{equation}
We now check, that given a function in the kernel of $\LinOp_k$, then all its modes, for $\vL \in \dZ^2$ such that $|\vL| \neq k$, must be zero. 

First for $\vL=(0,0) = \mathsf{0}$, 
we note that since we are in the space such that $\int \phi \dd \theta \dd x = 0$, this implies that, $\int \hat{\phi}_{\mathsf{0}} \dd \theta = 0$. The equation on the wave number $\mathsf{k}=\mathsf{0}$ then writes as,
\begin{equation*}
    \mathcal{F}_x\{\LinOp_k\phi\}_{\mathsf{0}} = \LinOpFk^0_{\chi} \hat{\phi}_{\mathsf{0}} = \sigma_\theta\partial_{\theta\theta} \hat{\phi}_{\mathsf{0}} = 0,
\end{equation*}
and the Poincaré inequality implies that $\hat{\phi}_{\mathsf{0}}  = 0$.

For $|\vL| \neq k $ and $\vL \neq \mathsf{0}$, we write the estimate
{\small
\begin{equation}
    \label{est:KernelAnalysisEstimeellK}
    \begin{split}
        \hspace{-1em} \left| \int_0^{2\pi} \fU^{\vL}_{\sigma} \mathrm{d}\theta  - \int_0^{2\pi} \fU^{\vK_1}_{\sigma} \mathrm{d}\theta  \right| & \geq \left| \int_0^{2\pi} \fU^{\vL}_{0} \mathrm{d}\theta  - \int_0^{2\pi} \fU^{\vK_1}_{0} \mathrm{d}\theta  \right| - \left| \int_0^{2\pi} \fU^{\vL}_{\sigma} \mathrm{d}\theta  - \int_0^{2\pi} \fU^{\vL}_{0} \mathrm{d}\theta  \right|\\
        &\hspace{3em} - \left| \int_0^{2\pi} \fU^{\vK_1}_{\sigma} \mathrm{d}\theta  - \int_0^{2\pi} \fU^{\vK_1}_{0} \mathrm{d}\theta  \right|,
    \end{split}
\end{equation}
}and using the uniform convergence in $\sigma$ w.r.t. $\vL$, as the $\fU_0^{\vL}$-s are regular functions, we obtain that,
\begin{equation*}
    \max_{\ell \in \dZ^2 \setminus\{0\}}\left| \int_0^{2\pi} \fU^\ell_{\sigma} \mathrm{d}\theta  - \int_0^{2\pi} \fU^\ell_{0} \mathrm{d}\theta  \right| \leq C\sigma,
\end{equation*}
for some constant $C$.

We introduce the finite set of neighboring modes of the bifurcation modes as,
\begin{equation*}
    \mathcal{A}^k = \{ \vL \in \dZ^2 : |\vL|\neq k, \dist(\mathbb{S}^2_k, \vL) \leq 1\},
\end{equation*}
where $\mathbb{S}^2_k$ is the 2D sphere of radius $k$.

Then denoting,
\begin{equation*}
    C^k = \sup_{\ell \in \dZ^2 \setminus(\{0\} \cup \mathbb{S}^2_k) }\left| \int_0^{2\pi} \fU^{\vL}_{0} \mathrm{d}\theta  - \int_0^{2\pi} \fU^{\vK_1}_{0} \mathrm{d}\theta  \right|,
\end{equation*}
we obtain from the strict monotony of $\int \fU^{\vK_1}_0 \dd \theta $ in $k$, that
\begin{equation*}
    C^k = \max_{\vL \in \mathcal{A}^k} \left| \int_0^{2\pi} \fU^{\vL}_{0} \mathrm{d}\theta  - \int_0^{2\pi} \fU^{\vK_1}_{0} \mathrm{d}\theta  \right|  > 0.
\end{equation*}
Plugging this back in \eqref{est:KernelAnalysisEstimeellK}, we obtain that,
\begin{equation*}
    \begin{split}
        \left| \int_0^{2\pi} \fU^\ell_{\sigma_\theta } \mathrm{d}\theta  - \int_0^{2\pi} \fU^{\vK_1}_{\sigma_\theta} \mathrm{d}\theta  \right|  \geq C^k - C\sigma_\theta.
    \end{split}
\end{equation*}
Changing $\sigma^k_\theta >0$ to a smaller value (such that the above is strictly positive for $\sigma^k_\theta > \sigma >0$) we conclude that, by definition of $\chi^k$, (given $\sigma_k >\sigma>0$)
\begin{equation*}
    \frac{\chi^k}{2\pi} \int_0^{2\pi} \fU^{\vL}_{\sigma } \dd \theta  \neq 1,
\end{equation*}
for all $ \in \dZ^2\setminus\{0\}$ such that $|\vL| \neq k$. This implies according \eqref{eq:EigenValdef}, that for $\phi \in \ker \LinOp_k$, $\mathcal{F}_x\{\phi\}_\vL =0$,
and thus the required result. The conclusion on the cokernel follows with the same argument on the index as the one given earlier for any mode $|\vL| \neq k$, and the conclusion on the operator follows.
\end{proof}

In the following result, we give a semi-explicit form for the inverse of the linearized operator.
\begin{prop}
    \label{prop:InverseL}
    The operator $\LinOp_k$ has a bounded inverse $\left(\LinOp_k\right)^{-1} : \Range (\LinOp_k) \to H^2_0$, that preserves the $x$-modes. That is, if $g \in L^2_0$ such that, $\mathcal{F}_x\{g\}_{\vL} = \mathsf{g} \delta_{\xi \vL} + \mathsf{g}^* \delta_{-\xi \vL}$, for some $\xi \in \dZ^2 \setminus \mathbb{S}^2_k$, where $\mathbb{S}^2_k = \{\vL \in \dZ^2, |\vL| = k\}$,
    then,
    \begin{equation}
        \label{eq:modePreservingInverseL}
        \mathcal{F}_x\{ \left(\LinOp_k\right)^{-1} g\}_{\vL} = \mathsf{h} \delta_{\xi \vL} + \mathsf{h}^* \delta_{-\xi \vL} ,,
    \end{equation}
    for some complex-valued function $\mathsf{h}$.
    
    % Explicitly, for any function $g \in L^2_0$ such that $\hat{g}_{\vK} = 0$ for $\vK \in \mathcal{C}^k$, then we have,
    % \begin{align}
    %     \label{eq:InversionFormula}
    %     \mathcal{F}_x\left\{ \left(\LinOp_k\right)^{-1} g\right\}_{\boldsymbol{0}} & = \frac{1}{\sigma} (\partial_{\thth})^{-1} \hat{g}_{\boldsymbol{0}},\\
    %     \mathcal{F}_x\left\{ \left(\LinOp_k\right)^{-1} g\right\}_\vK & = 0, \text{ for } k \in \mathcal{C}^k\\
    %     \mathcal{F}_x\left\{ \left(\LinOp_k\right)^{-1} g\right\}_\vL &= \mathsf{R}^\vL_{\sigma_\theta} \hat{g}_\vL + \left(\frac{\chi^k(\sigma)}{2\pi \mathsf{K}^{|\vL|}(\chi^k(\sigma), \sigma)}\int \mathsf{R}^\ell_{\sigma} \hat{g}_\vL \dd \theta \right)\fU^\vL_{\sigma}, \text{ otherwise},
    % \end{align}
    Explicitly, for any function $g \in L^2_0$ such that $\hat{g}_{\vK} = 0$ for all $\vK \in \mathbb{S}^2_k$, then we have,
    \begin{equation}
        \label{eq:InversionFormula}
        \mathcal{F}_x\left\{ \left(\LinOp_k\right)^{-1} g\right\}_{\vL}  =
        \begin{cases}
            \frac{1}{\sigma_\theta} (\partial_{\thth})^{-1} \hat{g}_{0}, & \text{if }\vL = 0,\\
             0, & \text{if }\vL \in \mathbb{S}^2_k,\\
            \mathsf{R}^\vL_{\sigma_\theta} \hat{g}_\vL + \left(\frac{\chi^k}{2\pi \mathsf{K}^{|\vL|}(\chi^k, \sigma_\theta)}\int \mathsf{R}^\ell_{\sigma_\theta} \hat{g}_\vL \dd \theta \right)\fU^\vL_{\sigma_\theta}, & \text{otherwise},
        \end{cases}
    \end{equation}
    where we define for any $|\vL| \geq 1$, $\mathsf{K}^{|\vL|} : \dR \times \dR_+ \to \dR$ as, $\mathsf{K}^{|\vL|}(\chi, \sigma) =  1- \frac{\chi}{2\pi}\int \fU^{\vL}_\sigma \dd \theta$.
    %By definition of $ \sigma, \chi^k$, $\mathsf{K}^{|\vL|}(\chi^k(\sigma), \sigma) \neq 0$ if $|\vL| \neq k$.
    
    In particular, for a function $g$ as above, $\left(\LinOp_k\right)^{-1} $ preserves the central-conjugate-symmetry.
\end{prop}

\begin{proof}
    We denote $f = \left(\LinOp_k\right)^{-1}  g$.
    For the case $\vL \in \mathbb{S}^2_k$, by restriction of the operator, and since we suppose that $\hat{g}_\vL \equiv 0$, this implies that the mode is still zero. For the case $\vL = 0$, we have in Fourier the equation,
    \begin{equation}
        \sigma_\theta\partial_{\thth} \hat{f}_0 =\hat{g}_0,
    \end{equation}
    since $\int g = 0$ this implies that $\int \hat{g}_0 \dd \theta = 0$, and $(\partial_{\thth})^{-1}$ is well-defined from $L^2_{\theta;0}$ to $H^2_{\theta;0}$, and
    \begin{equation}
        \hat{f}_0 =\frac{1}{\sigma_\theta}(\partial_{\thth})^{-1} \hat{g}_0.
    \end{equation}
    Now for $\vL \in \dZ^2 \setminus (\mathcal{C}^k \cup \{0\})$, the equation writes as,
    \begin{equation*}
        \left(-\hat{\mathsf{M}}_\vL + \sigma_\theta\partial_{\thth}\right)\hat{f}_\vL - \frac{\chi^k}{2\pi} \partial_\theta\hat{\mathsf{B}}_\vL \int \hat{f}_\vL \dd \theta= \hat{g}_\vL.
    \end{equation*}
    Using the resolvent operator $\fOpR^\vL_{\sigma_\theta}$, this leads equivalently to
    \begin{equation}
        \label{eq:InverseLkModeEll}
        \hat{f}_\vL - \frac{\chi^k}{2\pi} \fOpR^\vL_{\sigma_\theta}\partial_\theta\fB_\vL \int \hat{f}_\vL \dd \theta = \fOpR^\vL_{\sigma_\theta}\hat{g}_\vL,
    \end{equation}
    and integrating, we obtain 
    \begin{equation*}
        \left(1 - \frac{\chi^k}{2\pi} \int \fOpR^\vL_{\sigma_\theta}\partial_\theta\fB_\vL\dd \theta \right)\int \hat{f}_\vL \dd \theta = \int\fOpR^\vL_{\sigma_\theta}\hat{g}_\vL\dd \theta,
    \end{equation*}

    \begin{equation*}
        \mathsf{K}^\vL(\chi^k(\sigma_\theta), \sigma_\theta)\int \hat{f}_\vL \dd \theta  = \int\fOpR^\vL_{\sigma_\theta}\hat{g}_\vL\dd \theta.
    \end{equation*}
    We used that $\mathsf{K}^\vL(\chi^k(\sigma_\theta), \sigma_\theta) \neq 0$ for $\sigma_\theta$ small by uniform continuity in $\vL$ of the integral. Solving for $\int \hat{f}_\ell \dd \theta$ and plugging in \eqref{eq:InverseLkModeEll}, we obtain
    \begin{equation*}
        \hat{f}_\vL = \fOpR^\vL_{\sigma_\theta}\hat{g}_\vL + \left(\frac{\chi^k}{2\pi \mathsf{K}^\ell(\chi^k, \sigma_\theta)} \int\fOpR^\vL_{\sigma_\theta}\hat{g}_\vL\dd \theta \right)\fU^\vL_{\sigma_\theta}.
    \end{equation*}
\end{proof}

We conclude this section by proving that at the first bifurcation point $k =  1$, the spectrum of the linearized operator around the homogeneous admits only $0$ as its only eigenvalue with positive real part.

\begin{prop}
    \label{prop:SpectralGapatFirstBif}
    Let $k = 1$, then there exists $\sigma^*_\theta > 0$, such that for $0 < \sigma_\theta < \sigma^*_\theta$ and $\chi^1 > 0$ as given by Theorem~\ref{thm:LinOpKern}, and there exists $c>0$, such that, 
    \begin{equation*}
        \Sigma (\LinOp_1) = \{ 0 \} \cup \Sigma_1,  \;\; \text{ with } \Sigma_1 \subset \{\mu \in \dC, \Real \mu \leq - c\}. 
    \end{equation*}
\end{prop}

\begin{proof}
    We use the decomposition $\Sigma(\LinOp_1) = \cup_{\vK \in \dZ^2} \Sigma(\LinOpFk^{\vK}_1)$, where $\LinOpFk^{\vK}_1$ is the restriction in positional Fourier of the operator. The eigenvalues of $\LinOpFk^0_1 = \sigma_\theta \partial_{\thth}$ over $L^2_{\theta;0}(\dC)$ are explicitly given by $\Sigma(\LinOpFk^0_1) =\left\{ - \sigma_\theta n^2, n \in \dN^*\right\}$. Then, for an eigenvalue $\mu$ of $\LinOpFk^{\vK}_1$ with $\vK \in \dZ^2 \setminus\{0\}$, we have the following equation,
    \begin{equation}
        \label{eq:eigenvalueProblemL1}
        \left(\LinOpFk^\vK_1 - \mu \right)\hat{\phi}_\vK  = \left(-\mu -\hat{\mathsf{M}}_\vK + \sigma_\theta \partial_{\thth}\right) \hat{\phi}_\vK - \frac{\chi^1}{2\pi} \partial_\theta\fB_\vK \int \hat{\phi}_\vK \dd \theta = 0.
    \end{equation}
    A similar result as in Proposition~\ref{prop:ResolventProperties} (see  Lemma 5.2 in~\cite{rakodeWit2025}) gives that, for $\mu \in \dC$ with $\Real(\mu) > -\sigk$, the operator $\left(-\mu -\hat{\mathsf{M}}_\vK + \sigma_\theta \partial_{\thth}\right)$ admits a bounded inverse. The eigenvalue problem for an eigenvalue $\mu$ is then equivalent to
    \begin{equation}
        \label{eq:eigenvalueProblemL1Equivalence}
        \frac{\chi^1}{2\pi} \int \left(-\mu -\hat{\mathsf{M}}_\vK + \sigma_\theta \partial_{\thth}\right)^{-1} \partial_\theta\fB_\vK\dd \theta = 1.
    \end{equation}
    We denote $\mathsf{U}^{\vK}_{\sigma_\theta}(\mu,\theta) = \left(-\mu -\hat{\mathsf{M}}_\vK + \sigma_\theta \partial_{\thth}\right)^{-1} \partial_\theta\fB_\vK$ and
    $\mathcal{J}^k_{\sigma_\theta}(\mu) = \Real \int \mathsf{U}^{\vK_1}_{\sigma_\theta}(\mu, \theta) \dd \theta$. From Proposition~\ref{prop:RotationUKukomega}, we have that $\int \mathsf{U}^{\vK}_{\sigma_\theta}(\mu, \theta) \dd \theta = \int \mathsf{U}^{\vL}_{\sigma_\theta}(\mu, \theta) \dd \theta$ if $|\vL| = |\vK|$. It is thus sufficient to prove that, for $\sigma_\theta$ sufficiently small, the following hold,
    \begin{subequations}\label{eq:j1}
    \begin{align}
        \mathcal{J}^1_{\sigma_\theta}(0) &> \mathcal{J}^1_{\sigma_\theta}\left(re^{i\phi}\right),\\
        \mathcal{J}^1_{\sigma_\theta}(0) &> \mathcal{J}^k_{\sigma_\theta}(\mu),
    \end{align}
    \end{subequations}
    for $r\in (0,+\infty)$, $\phi \in \left[-\tfrac{\pi}{2},  \tfrac{\pi}{2}\right]$,  $k \geq 2$ and $\mu \in \dC$ with $\Real \mu \geq 0$. 
    
    If these hold, we have the identity (by definition of $\chi^1$),
    \begin{equation*}
        \int \mathsf{U}^{\vK}_{\sigma_\theta}(\mu, \theta) \dd \theta = \frac{2\pi}{\chi^1},
    \end{equation*}
    if and only if $\mu = 0$ and $|\vK| = 1$, and using the equivalence between the integral equation~\eqref{eq:eigenvalueProblemL1Equivalence} and the eigenvalue problem~\eqref{eq:eigenvalueProblemL1}, $0$ is the only eigenvalue of $\LinOp_1$ with non-negative real part. The spectral gap condition,
    \begin{equation*}
        \Sigma (\LinOp_1) \setminus\{0\} \subset \{\mu \in \dC, \Real \mu \leq - c\},
    \end{equation*}
    follows from the fact that $\LinOp_1$ has a compact inverse and thus a discrete spectrum without accumulation point in any compact subset of $\dC$.

    We now prove \eqref{eq:j1}. We use the information on the explicit inviscid integral associated with $\mathsf{U}^{\vK}_{0}(\mu, \cdot ) = \partial_\theta\fB_\vK / (-\mu -\hat{\mathsf{M}}_\vK )$, given by,
    
    \begin{equation}
        \label{eq:MuIntegralEigenval}
        \mathcal{J}^k_{0}(\mu) =  \frac{4\pi^2 k \Aek(\mu)}{\cEll \lmbk} (1-|\Azink(\mu)| + 2(\sigk + \mu/\lmbk)|\Azink(\mu)| \tauk),
    \end{equation}
    with
    \begin{align*}
        \Azink(\mu) & \defeq -(2(\sigk+\mu/\lmbk)^2+1) + \sqrt{(2(\sigk+\mu/\lmbk)^2+1)^2-1},\\
        \Aek &\defeq \left(\sqrt{(2(\sigk+\mu/\lmbk)^2+1)^2-1}\right)^{-1}. 
    \end{align*}
    From this explicit form, we obtain that there exists $\delta >0$, and $d>0$, such that
    \begin{align*}
        &\sup_{r\in (0,\delta)} \sup_{\varphi \in \left[-\tfrac{\pi}{2}, \tfrac{\pi}{2}\right]} \frac{\dd}{\dd r} \mathcal{J}^1_0(r e^{i\varphi} ) \leq - d,\\
        &M \defeq \max_{\varphi \in \left[-\tfrac{\pi}{2}, \tfrac{\pi}{2}\right]} \mathcal{J}^1_0(\delta e^{i\varphi}) = \sup_{R \in (\delta, + \infty)} \sup_{\varphi \in \left[-\tfrac{\pi}{2}, \tfrac{\pi}{2}\right]} \mathcal{J}^1_0(R e^{i\varphi}) \geq  \sup_{k \geq 2}\sup_{\Real \mu \geq 0} \mathcal{J}^k_0(\mu).
    \end{align*}
    We then apply a similar $L^2$-energy estimate as in Lemma 5.2 in~\cite{rakodeWit2025}. Consider the following equations,
    \vspace{-1em}
    \begin{align*}
        \left(-\mu -\hat{\mathsf{M}}_\vK + \sigma_\theta \partial_{\thth}\right) \mathsf{U}^{\vK}_{\sigma_\theta}(\mu)  &= \partial_\theta\fB_\vK,\\
        \left(-\mu -\hat{\mathsf{M}}_\vK + \sigma_\theta \partial_{\thth}\right) \mathsf{U}^{\vK}_{0}(\mu) - \sigma_\theta \partial_{\thth} \mathsf{U}^{\vK}_{0}(\mu)  &= \partial_\theta\fB_\vK.
        \end{align*}
    %     {\color{pink}
    %     \begin{align*}
    %     &\left(-\mu -\hat{\mathsf{M}}_\vK + \sigma_\theta \partial_{\thth}\right)\left(\mathsf{U}^{\vK}_{\sigma_\theta}(\mu)-\mathsf{U}^{\vK}_{0}(\mu)\right)=-\sigma_\theta\partial_{\thth}\mathsf{U}^{\vK}_{0}(\mu)\\&
    %     \color{pink}\int(-\mu-\hat{\mathsf{M}}_\vK+\sigma\partial_{\thth})\left(\mathsf{U}^{\vK}_{\sigma_\theta}(\mu)-\mathsf{U}^{\vK}_{0}(\mu)\right)\left(\mathsf{U}^{\vK}_{\sigma_\theta}(\mu)-\mathsf{U}^{\vK}_{0}(\mu)\right)^\dagger\dd\theta\\&=-\sigma_\theta\int\partial_{\thth}\mathsf{U}^{\vK}_{0}(\mu)\left(\mathsf{U}^{\vK}_{\sigma_\theta}(\mu)-\mathsf{U}^{\vK}_{0}(\mu)\right)^\dagger\dd\theta,\\
    %   &(\Real(\mu)+4\pi^2\sigma_x k^2)\|\mathsf{U}^{\vK}_{\sigma_\theta}(\mu)-\mathsf{U}^{\vK}_{0}\|_{\dL^2}^2-\int\sigma_{\theta}\partial_{\thth}\left(\mathsf{U}^{\vK}_{\sigma_\theta}(\mu)-\mathsf{U}^{\vK}_{0}(\mu)\right)\left(\mathsf{U}^{\vK}_{\sigma_\theta}(\mu)-\mathsf{U}^{\vK}_{0}(\mu)\right)^\dagger\dd\theta\\
    %   &=\sigma_\theta\int\partial_{\thth}\mathsf{U}^{\vK}_{0}(\mu)\left(\mathsf{U}^{\vK}_{\sigma_\theta}(\mu)-\mathsf{U}^{\vK}_{0}(\mu)\right)^\dagger\dd\theta
    % \end{align*}}
    As $\mathsf{U}^{\vK}_{0}(\mu)$ is regular and explicitly given by $\mathsf{U}^{\vK}_{0}(\mu) = \partial_\theta\fB_\vK / (-\mu -\hat{\mathsf{M}}_\vK )$, we obtain that,
    \begin{equation*}
        \left\| \mathsf{U}^{\vK}_{0}(\mu) - \mathsf{U}^{\vK}_{\sigma_\theta}(\mu)\right\|_{L^2} \leq \frac{\sigma_\theta}{ \Real \mu + 4\pi^2 |\vK|^2 \sigma_x }\left\| \partial_{\theta\theta} \left(\mathsf{U}^{\vK}_{0}(\mu) \right) \right\|_{L^2}.
    \end{equation*}
    Using the explicit form of $\mathsf{U}^{\vK}_{0}(\mu)$, we derive the following uniform estimate,
    \begin{equation}
        \label{est:Uunifmuk}
       \sup_{\vK \in \dZ^2\setminus\{0\}} \sup_{\mu \in \dC_{+}} \left\| \mathsf{U}^{\vK}_{0}(\mu) - \mathsf{U}^{\vK}_{\sigma_\theta}(\mu)\right\|_{L^2} \leq C \sigma_\theta,
    \end{equation}
    for some positive constant $C>0$. Similarly, we obtain an estimate on the derivative in $\mu$, thus for $\sigma_\theta^*$ sufficiently small, we can ensure that,
    \begin{align*}
       \sup_{k \geq 2} \sup_{\mu \in \dC_+} \left|\mathcal{J}^k_{\sigma_\theta}(\mu) - \mathcal{J}^k_{0}(\mu) \right| &\leq \frac{\mathcal{J}^1_{0}(0) - M}{2},\\
        \sup_{ r\in [\delta,+\infty)} \sup_{\varphi \in \left[-\tfrac{\pi}{2}, \tfrac{\pi}{2}\right]} \left|\mathcal{J}^1_{\sigma_\theta}(re^{i\varphi}) - \tfrac{\dd}{\dd \mu}\mathcal{J}^1_{0}(re^{i\varphi})\right| &\leq \frac{\mathcal{J}^1_{0}(0) - M}{2},\\
        \sup_{ r\in [0,\delta)}\sup_{\varphi \in \left[-\tfrac{\pi}{2}, \tfrac{\pi}{2}\right]}   \left|\tfrac{\dd}{\dd r}\mathcal{J}^1_{\sigma_\theta}(re^{i\varphi}) - \tfrac{\dd}{\dd r}\mathcal{J}^1_{0}(re^{i\varphi}) \right| &\leq  \frac{d}{2},
    \end{align*}
    for all $0<\sigma_\theta < \sigma^*_\theta$. These three estimates ensure that $\mathcal{J}^1_{\sigma_\theta}(0)$ stays the strict maximum of $\sup_{k \geq 1} \sup_{\mu \in \dC_+} \mathcal{J}^k_{\sigma_\theta}(\mu)$. The conclusion follows with the equivalence between the eigenproblem and the integral condition~\eqref{eq:eigenvalueProblemL1Equivalence}.
\end{proof}

    \section{Analysis of Bifurcation Curves}
\label{sec:BirfucationCurves}
In this section, we develop the Lyapunov--Schmidt (LS) reduction method for our functional equation \eqref{eq:fullFstat}. To overcome the multidimensionality of the reduced LS equation (the kernel of the linearized operator is a four--dimensional space, see Theorem~\ref{thm:LinOpKern}), we first remove the translation symmetry by fixing an appropriate functional framework. Then, by choosing a suitable basis representation of the LS function and identifying it as a parametric vector field $\mathbb{R}^2 \times \mathbb{R} \to \mathbb{R}^2$, we show that the axes and diagonal directions are field lines (see Figure~\ref{fig:PhiXiLambda}). Restricting to these field lines reduces the problem to one dimension, allowing the application of Newton's polygon method. Solutions along the diagonal field lines correspond to spot solutions, while solutions along the axes correspond to lane solutions.

% The existence of such field lines is directly related to the fact that we expect solutions that can be constant in one of there spatial variables: Lane $\phi^{\vK_1}$ eigenfunction, and solutions that are swap symmetric: Spot solutions, composed of $\phi^{\vK_1}+\phi^{\vK_2}$.solutions, 
% Along these lines, the vector field is tangent to the lines, thus reducing the dimension of the problem further. 

\begin{figure}[!ht]
    \centering
    \includegraphics[width=0.5\linewidth]{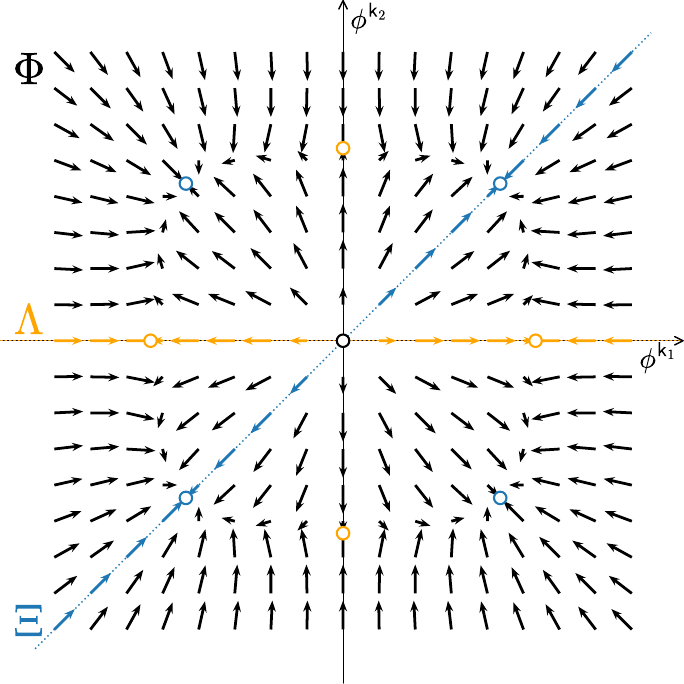}
    \caption{Representation of the Lyapunov-Schmidt functional $\bPhi$ as a vector field, with field lines $\Xi$ and $\Lambda$.}
    \label{fig:PhiXiLambda}
\end{figure}

\subsection{Lyapunov--Schmidt Reduction}

% We note that the first reduction of dimension could be preserved as an equivalence in the analysis (see Trash) but it is more tedious, and we can still come back to construct the translated family.

%\subsection{Lyapunov-Schmidt reduction}

\newcommand{\phisig}{\phi_{\sigma_{\theta}}}
\newcommand{\psisig}{\psi_{\sigma_{\theta}}}

We first start by defining the quotient spaces for the functional framework. The Lyapunov-Schmidt reduction is performed in the following functional subspaces,
\begin{equation*}
    X = H^2_0 \cap \left\{ \antipodreflOp f = f \right\}, \; Z = L^2_0 \cap \left\{ \antipodreflOp f = f \right\},
\end{equation*}
equipped with their Hilbertian structure. The space $X$ is continuously embedded in $Z$ by the Sobolev embedding theorem. % We refer to Kielofer's book~\cite{kielhofer2006bifurcation} p.219-225 for more details on the appropriate functional setting. 
We note that $\F$ and $D_f\F$ are well defined in these quotient spaces. That is, we have
\begin{equation*}
   \F \in C(X\times \dR, Z),\; D_f \F \in C(X\times \dR, L(X,Z)).
\end{equation*}
This follows from Proposition~\ref{antipodalReflectionCommutations} as, $\antipodreflOp (D_f\F(f,\chi)[\phi])=D_f\F(f,\chi)[\antipodreflOp \phi]$, given that $f,\phi\in X$. Furthermore, we note that in this functional framework, $D_f\F(0,\chi)|_X = \LinOp_\chi|_X : X \to Z$. The kernel and cokernel, at a non-Pythagorian wave number $k$, are given as the intersection with the symmetric property $\{\antipodreflOp f = f\}$, thus removing the $\bis$-eigenfunctions,
\begin{align*}
    \ker (\LinOp_k|_X) = \Span \{ \phi^{\vK_1}, \phi^{\vK_2}\},\\
    \coker (\LinOp_k|_X) = \Span \{ \psi^{\vK_1}, \psi^{\vK_2}\}.
\end{align*}
% We thus write the following decomposition, for some Hilbert space $H_0${\color{pink} check spaces and notation},
% \begin{align*}
%     X &= \ker (\LinOp_k|_X) \oplus X_0,\\
%     Z &= \coker (\LinOp_k|_X) \oplus \Range(\LinOp|_X).
% \end{align*}

For the sake of notational conciseness, we drop the restriction on $X$, and just denote $\LinOp_k$, the operator defined by restriction with the symmetry, $\LinOp_k: X\to Z$, at a bifurcation point $\chi^k$.

We now prove that $0$ is semi--simple. This gives a space decomposition for the LS reduction that is relevant for the stability analysis, as in our case, the Principle of Reduced Stability applies. We first recall the definition of a semi--simple eigenvalue.

\begin{definition}
A complex number $\mu$ is said to be a semi-simple eigenvalue of an operator $A$  of multiplicity $n \geq 1$ if  $\dim \ker (A-\mu Id) = n $ and $\ker A \cap \Range A = \{ 0 \}$. 
\end{definition}
\begin{prop}
    \label{prop:SpaceDecompositionLS}
    The following identity holds,
    \begin{equation}
        \ker \LinOp_k \cap \Range \LinOp_k = \{0\}.
    \end{equation}
    Furthermore, as $\LinOp_k$ has a zero index, this leads to the functional decompositions,
    \begin{align}
        \label{SpaceDecomposition}
        X & = \ker (\LinOp_k) \oplus (\Range (\LinOp_k )\cap X),\\
        Z & = (\ker \LinOp_k) \oplus \Range (\LinOp^k ).
    \end{align}
    Finally, we define the projection operator $Q : Z \to \ker  (\LinOp_k) $ along $\Range  (\LinOp_k) $, explicitly given by,
    \begin{equation}
        Q  = (\langle \cdot , \psi^{\vK_1}\rangle \phi^{\vK_1} + \langle \cdot , \psi^{\vK_2}\rangle \phi^{\vK_2})\Normpsiphi
    \end{equation}
    where, $\Normpsiphi = 1/\langle \phi^{\vK_1} , \psi^{\vK_1}\rangle$.
\end{prop}

In order to prove the result, we need the following Lemma. The computation is given in  Appendix~\ref{sec:AppendixIntegrals}.

\begin{lem}
We have the following identity, $\langle \psi^{\vK_1}, \phi^{\vK_1}\rangle  = \langle \psi^{\vK_2}, \phi^{\vK_2}\rangle$, and the explicit form is given by,
% \begin{align}
%     \label{eq:Phi1Psi1Integral}
%     \langle \psi^{\vK_1}, \phi^{\vK_1}\rangle  & = \frac{8\pi^2 k }{\cEll\lmbk^2}\Aek^2 (1-|\Azink|^2)^{-1}\Big(\tauk(-|\Azink|^4+4|\Azink|^2  -4|\Azink|+ 8\sigk^2|\Azink| + 1)\nonumber\\
%     &\hspace{17em}+4\sigk(1-|\Azink|)^2\Big)+ O(\sigma_\theta).
% \end{align}
\begin{equation}
    \label{eq:Phi1Psi1Integral}
    \langle \psi^{\vK_1}, \phi^{\vK_1}\rangle  = \frac{8\pi^2 k \Aek^2 }{\cEll\lmbk^2}\left(p^k_1\tauk + p^k_0\right),
\end{equation}
where,
\begin{align}
    p^k_1 & =  1+|\Azink|^2 - \frac{4|\Azink|}{1+|\Azink|} + \frac{8\sigk^2|\Azink|}{1-|\Azink|^2} + \underset{\sigma_\theta \to 0}{o}(1),\\
    p^k_0 & = 4\sigk \frac{1-|\Azink|}{1+|\Azink|} + \underset{ \sigma_\theta \to 0}{o}(1).
\end{align}

We conclude that for any $\sigk>0$, there exists $\sigma_\theta^*>0$ sufficiently small such that for every $0\leq \sigma_\theta < \sigma_\theta^*$,
\begin{equation*}
    \langle \psi^{\vK_1}, \phi^{\vK_1}\rangle  > 0.
\end{equation*}
% as $p^k_1 > 0$ and $p^k_0 > 0$ for every $\tauk$, as the leading order of the coefficients are always positive.
\end{lem}

\begin{proof}[Proof of Proposition~\ref{prop:SpaceDecompositionLS}]
    Let $u \in \ker (\LinOp_k) \cap (\Range (\LinOp_k )\cap X)$,
    \begin{equation*}
        u = \alpha_1 \phi^{\vK_1} + \alpha_2 \phi^{\vK_2} = \LinOp_k w,
    \end{equation*}
    for some $w \in X$. Taking the scalar product with the basis of the cokernel, that is, the orthogonal space of the image, we obtain,
    \begin{equation}
        \begin{pmatrix}
            \langle \psi^{\vK_1}, \phi^{\vK_1}\rangle & \langle \psi^{\vK_1}, \phi^{\vK_2} \rangle\\
            \langle \psi^{\vK_2} , \phi^{\vK_1} \rangle & \langle \psi^{\vK_2}, \phi^{\vK_2} \rangle
        \end{pmatrix} \begin{pmatrix}
            \alpha_1 \\
            \alpha_2
        \end{pmatrix} =\begin{pmatrix}
            0 \\
            0
        \end{pmatrix}.
    \end{equation}
    Then, using the $x$-Fourier-mode decomposition, we obtain $\langle \psi^{\vK_1}, \phi^{\vK_2} \rangle = \langle \psi^{\vK_2} , \phi^{\vK_1} \rangle = 0$, and the swap symmetries give $\langle \psi^{\vK_1}, \phi^{\vK_1}\rangle = \langle \psi^{\vK_2}, \phi^{\vK_2}\rangle$.
    We conclude that for $\sigma_\theta$ sufficiently small, the diagonal terms are non-zero, by using the continuity in $\sigma_\theta$ and as it is non-zero at $\sigma_\theta = 0$, see Lemma~\ref{lem:phipsiIntegral} in Appendix~\ref{sec:AppendixIntegrals}. Thus, the matrix is invertible, and since $\alpha_1 = \alpha_2 = 0$, the required result follows. The projector formula is obtained by an analog computation.
\end{proof}

In the following we are going to note,
\begin{equation*}
    Z_0  = \Range \LinOp_k,\;\text{ and }\; X_0  = (\Range \LinOp_k)\cap X.
\end{equation*}

The following theorem gives the existence of the reduced LS equation for this basis and provides some of its properties.
\begin{thm}
    \label{thm:LyapunovSchmidt}
    For any non-Pythagorean wave number $k\in \dN^\ast$ and given its associated bifurcation point $\chi^k$, there exists a smooth-analytic function $\bPhi^k : U\times U \times V \subset \dR^2\times \dR \to \dR^2$, with $U\times U \times V$ a neighborhood of $(0,0,\chi^k)$, a smooth-analytic function $G : U \times U \times V \subset \dR^2\times \dR \to X_0$ and a neighborhood $O$ of the constant solution in $X$, such that, for $\chi \in V$ and $ f \in O$, the following equivalence holds, 
    
    % \begin{center}
    %     \begin{minipage}{.8\textwidth}
    %         For $\chi \in V$ and $ f \in O$, the following equivalence holds,
    %         \begin{equation*}
    %             \F(f,\chi) = 0 \text{ in } Z, 
    %         \end{equation*}
    %         if and only if, $f = r_1 \phi^{\vK_1} + r_2 \phi^{\vK_2} + G(r_1, r_2, \chi)$, $(r_1, r_2) \in U\times U$, and
    %         \begin{equation*}
    %              \Phi^k(r_1,r_2, \chi) = 0.
    %         \end{equation*}
    %     \end{minipage}
    % \end{center}
    \begin{center}
        \begin{minipage}{.8\textwidth}
            $\F(f,\chi) = 0 \text{ in } Z$, if and only if, $f = r_1 \phi^{\vK_1} + r_2 \phi^{\vK_2} + G(r_1, r_2, \chi)$, $(r_1, r_2) \in U\times U$, and $\Phi^k(r_1,r_2, \chi) = 0$.
        \end{minipage}
    \end{center}
    
    Moreover, $\bPhi^k$ admits the following field lines: there exist functions $ \Lambda^k,\Xi^k : U \times V  \to \dR$, such that, $\bPhi^k(r, r , \chi) = (1 , 1)^{\mathsf{T}} \Xi^k(r,\chi)$, $\bPhi^k(r, 0, \chi) = (1 , 0)^{\mathsf{T}}\Lambda^k(r,\chi)$
    and $\bPhi^k$ admits the following swap symmetry,
    \begin{equation}
        \label{eq:PhiSwapSymetry}
        \bPhi^k(r_1,r_2,\chi) = \begin{pmatrix}
                                0 & 1\\
                                1 & 0
                            \end{pmatrix} \bPhi^k(r_2,r_1,\chi).
    \end{equation}
    
    Finally, the functionals defined above have the following vanishing derivatives, \\ for all $n \geq 0$ and $\chi \in V$,\begin{subequations}\label{eq:dGPhiVanishingDerivatives}\begin{align}
        \label{eq:dG0chi0}
        D_r G(0,0,\chi) &\equiv 0 \textit{ in } L(\dR^2, X_0),\\
        \label{eq:dchiG0chi0}
        D_{\chi}^n G(0,0,\chi) &\equiv 0 \textit{ in } X_0, \\
        \label{eq:dchiPhi0chi0}
        \partial_{\chi}^n \bPhi^k(0,\chi) &\equiv 0.
    \end{align}
    \end{subequations}
\end{thm}
% As stated earlier, the crucial part of the statement resides in the existence of $\Xi$ and $\Lambda$. They imply that the vector field is tangent to the axes and diagonals, and that the problem reduces to a one--dimensional problem along these lines.
\begin{proof}
    We use the space decomposition and the projection operator $Q$, of Proposition~\ref{prop:SpaceDecompositionLS}, to obtain that the problem is equivalent to,
    \begin{align*}
        Q \F(r_1 \phi^{\vK_1} + r_2 \phi^{\vK_2} + g; \chi) &= 0 \text{ in } \ker \LinOp_k,\\
        (I-Q) \F(r_1 \phi^{\vK_1} + r_2 \phi^{\vK_2} + g; \chi) &= 0  \text{ in } Z_0,
    \end{align*}
    for $g\in X_0$.
    
    Applying the Lyapunov--Schmidt procedure, using the analytic implicit function theorem as in \cite{kielhofer2006bifurcation}, from the analyticity of $F$, we obtain that there exists $U$, a neighborhood of $(0,0)$, $V$, a neighborhood of $\chi^k$, and a unique analytic function,
    \begin{equation*}
        G: U\times V \subset \dR^2\times \dR \to X_0,
    \end{equation*}
    such that,
    \begin{equation}
        \label{eq:LSFunctionEquationGequation}
        (I-Q) \F(r_1 \phi^{\vK_1} + r_2 \phi^{\vK_2}  + G(r_1, r_2,\chi); \chi) = 0 \text{ in } Z_0.
    \end{equation}
    The original problem is thus equivalent to,
    \begin{equation}
        \label{eq:LSFunctionEquation}
        Q\F(r_1 \phi^{\vK_1} + r_2 \phi^{\vK_2}  + G(r_1, r_2,\chi); \chi) = 0 \text{ in } \ker \LinOp_k.
    \end{equation}
    
    Then, using the representation of the operator $Q$ as in Proposition~\ref{prop:SpaceDecompositionLS}, since the $\psi^\vK$-functions are orthogonal, we can define, 
    \begin{equation}
        \label{eq:LSFunctionPhi}
        \bPhi^k(r_1,r_2,\chi) = \Normpsiphi\begin{pmatrix}\langle \psi^1, \F(r_1 \phi^{\vK_1} + r_2 \phi^{\vK_2} + G(r_1, r_2,\chi); \chi)\rangle\\
        \langle \psi^2, \F(r_1 \phi^{\vK_1} + r_2 \phi^{\vK_2} + G(r_1, r_2,\chi); \chi)\rangle
        \end{pmatrix}
    \end{equation}
    and then equation \eqref{eq:LSFunctionEquation} is equivalent to,
    \begin{equation}
        \label{eq:LSFunctionPhiEquation}
        \bPhi^k(r_1,r_2,\chi) = 0 \text{ in } \dR^2.
    \end{equation}
    
    Now, since $f = 0$ is a solution for all $\chi$, this implies, from the uniqueness of the definition of $G$, the property~\eqref{eq:dchiPhi0chi0}. The equivalence of the finite--dimensional problem with the equation $\F(0,\chi) = 0$ implies the property~\eqref{eq:dchiPhi0chi0}. The property~\eqref{eq:dG0chi0} follows in the same way as Corollary I.2.4 in \cite{kielhofer2006bifurcation}.

    We now prove the existence of $\Xi$. To do so, we prove the identity~\eqref{eq:PhiSwapSymetry}. We first note that, using the properties of Proposition~\ref{prop:swapRelation} and Proposition~\ref{prop:SwapOperator},\begin{equation*}
        \swapOp \psi^{\vK_1} = \psi^{\vK_2},\; \text{ and } \; \langle \swapOp  \cdot , \psi^{\vK_1}\rangle = \langle  \cdot , \swapOp \psi^{\vK_1}\rangle = \langle  \cdot , \psi^{\vK_2}\rangle.
    \end{equation*}
    This leads to the following commutation identities,
    \begin{equation*}
        \swapOp(I - Q) = (I -Q)\swapOp, \text{ and } \swapOp Q = Q\swapOp.
    \end{equation*}
    This together with the commutation of $\swapOp$ with the equation functional $F$, as in Proposition~\ref{prop:SwapOperator}, we obtain,
    \begin{equation}
        \swapOp (I-Q)\F(f,\chi) = (I-Q)\F(\swapOp f,\chi).
    \end{equation}
    By applying $S$ to equation~\eqref{eq:LSFunctionEquationGequation}, and using the previous commutations, this leads to,
    \begin{equation*}
         0 = \swapOp (I-Q)\F(r_1 \phi^{\vK_1} + r_2 \phi^{\vK_2} + G(r_1, r_2,\chi),\chi) = (I-Q)\F(r_1\phi^{\vK_2} + r_2\phi^{\vK_1} + \swapOp G(r_1,r_2,\chi),\chi),
    \end{equation*}
    and thus, since the right side of the previous equation is zero, from the uniqueness of $G$, this leads to,
    \begin{equation}
        \swapOp G(r_1,r_2,\chi) =  G(r_2,r_1,\chi).
    \end{equation}
    We plug this into the definition of $\Phi^k$. This gives
    \begin{align*}
        \bPhi^k(r_1,r_2;\chi) = &  \Normpsiphi \begin{pmatrix} \langle \psi^{\vK_1}, \F(\swapOp (r_1\phi^{\vK_2} + r_2\phi^{\vK_1} + G(r_2, r_1, \chi));\chi) \rangle \\  \langle \psi^{\vK_2}, \F(\swapOp (r_1\phi^{\vK_2} + r_2\phi^{\vK_1} + G(r_2, r_1, \chi));\chi) \rangle \end{pmatrix}\\
        = & \Normpsiphi \begin{pmatrix}  \langle \psi^{\vK_1}, \swapOp \F(r_1\phi^{\vK_2} + r_2\phi^{\vK_1} + G(r_2, r_1, \chi);\chi) \rangle \\  \langle \psi^{\vK_2}, \swapOp \F(r_1\phi^{\vK_2} + r_2\phi^{\vK_1} + G(r_2, r_1, \chi);\chi) \rangle \end{pmatrix}  \\
        = & \Normpsiphi \begin{pmatrix} \langle \swapOp^* \psi^{\vK_1},  \F(r_1\phi^{\vK_2} + r_2\phi^{\vK_1} + G(r_2, r_1, \chi);\chi) \rangle \\  \langle \swapOp^* \psi^{\vK_2},  \F(r_1\phi^{\vK_2} + r_2\phi^{\vK_1} + G(r_2, r_1, \chi);\chi) \rangle \end{pmatrix} \\
        = &  \Normpsiphi \begin{pmatrix}  \langle \swapOp \psi^{\vK_1},  \F(r_1\phi^{\vK_2} + r_2\phi^{\vK_1} + G(r_2, r_1, \chi);\chi) \rangle \\  \langle \swapOp \psi^{\vK_2},  \F(r_1\phi^{\vK_2} + r_2\phi^{\vK_1} + G(r_2, r_1, \chi);\chi) \rangle \end{pmatrix}\\
        = &  \Normpsiphi \begin{pmatrix} \langle \psi^{\vK_2},  \F(r_1\phi^{\vK_2} + r_2\phi^{\vK_1} + G(r_2, r_1, \chi);\chi) \rangle \\  \langle  \psi^{\vK_1},  \F(r_1\phi^{\vK_2} + r_2\phi^{\vK_1} + G(r_2, r_1, \chi);\chi) \rangle \end{pmatrix} \\
        = & \begin{pmatrix} 0 & 1 \\ 1 & 0 \end{pmatrix} \Phi^k(r_2, r_1;\chi),
    \end{align*}
    where we use that $\swapOp^*\psi^{\vK_1} = \swapOp\psi^{\vK_1} = \psi^{\vK_2}$. Defining $\Xi$ as, $\Xi^k(r,\chi) = \left(\bPhi^k(r,r;\chi)\right)_1$, yields the required result.
    % {\color{pink} proof of existence of $\Lambda$, fix the space of functions constant in $x_2$, do the lyapunov schmidt, get the function $\Tilde{G}$, show that it coincides with $G$, thus they coincide, then $y_1\phi^1 + G(y_1, 0, \chi)$ is constant in $x_2$ then so is $F$ of this, then the projection against $\phi^2$ is zero.}
    
    We now prove the existence of $\Lambda$. To this purpose, we define the following subspaces of $X$ and $Z$,
    \begin{align*}
        X_{\overline{x}_2} &= X \cap \left\{ f(x_1,x_2, \theta) \equiv \int f(x_1,s, \theta)\dd s \right\},\\
        Z_{\overline{x}_2} &= Z \cap \left\{ f(x_1,x_2, \theta) \equiv \int f(x_1,s, \theta)\dd s \right\},
    \end{align*}
    corresponding to functions that are constant in the $x_2$-variable. From the structure of the equation $\F$ and its linearized operator ($\F$ and $D_f\F$ preserve constancy in the second variable), we note that,
    \begin{align*}
        \F|_{X_{\overline{x}_2}} & \in C(X_{\overline{x}_2} \times \dR, Z_{\overline{x}_2})\\
        D_f \F|_{X_{\overline{x}_2}} & \in C(X_{\overline{x}_2} \times \dR, L(X_{\overline{x}_2}, Z_{\overline{x}_2})),\\
        D_f \F(0,\chi^k) |_{X_{\overline{x}_2}} & = \LinOp_k|_{X_{\overline{x}_2}} : X_{\overline{x}_2} \to Z_{\overline{x}_2}.
    \end{align*}
    This implies that under this restriction,
    \begin{align*}
        \ker (\LinOp_k|_{X_{\overline{x}_2}}) = \Span \{ \phi^{\vK_1}\},\\
        \coker (\LinOp_k|_{X_{\overline{x}_2}}) = \Span \{ \psi^{\vK_1}\},
    \end{align*}
    and the following space decompositions hold,
    \begin{align*}
        X_{\overline{x}_2} &= \Span \{ \phi^{\vK_1}\}\oplus (\Range \left(\LinOp^k|_{X_{\overline{x}_2}}\right) \cap X_{\overline{x}_2} ),\\
        Z_{\overline{x}_2} &= \Span \{ \phi^{\vK_1}\} \oplus \Range \left(\LinOp^k|_{X_{\overline{x}_2}}\right).
    \end{align*}

    Furthermore, we have that $\Range \left(\LinOp_k|_{X_{\overline{x}_2}}\right) = \Range (\LinOp_k) \cap Z_{\overline{x}_2}$. We use the following notations, $\widetilde{Z}_0  = \Range \left(\LinOp_k|_{X_{\overline{x}_2}}\right), \widetilde{X}_0 = Z_0 \cap X_{\overline{x}_2}$ and we define the projection operator $\widetilde{Q} : Z_{\overline{x}_2} \to \ker \left(\LinOp^k|_{X_{\overline{x}_2}}\right) $ along $\widetilde{Z}_0$, given by $\widetilde{Q} = \langle \cdot , \psi^{\vK_1}\rangle \Normpsiphi \phi^{\vK_1}$. We then apply the Lyapunov--Schmidt procedure on this reduced space and obtain the existence of a function,
    \begin{equation*}
        \widetilde{G} : \widetilde{U} \times \widetilde{V} \subset \dR \times \dR \to \widetilde{X}_0,
    \end{equation*}
    defined implicitly as the unique function such that for any $r_1, \lambda \in \widetilde{U} \times \widetilde{V}$,
    \begin{equation*}
        (I-\widetilde{Q}) \F(r_1\phi^{\vK_1}+ \widetilde{G}(r_1, \chi),\chi) =0.
    \end{equation*}
    Now, as $\F(r_1\phi^{\vK_1}+ \widetilde{G}(r_1, \chi),\chi) \in Z_{\overline{x}_2}$, and from the definition of $\psi^{\vK_2}$, we obtain that,
    \begin{equation}
        \label{eq:Phi2iszero}
        \langle \F(r_1\phi^{\vK_1}+ \widetilde{G}(r_1, \chi);\chi), \psi^{\vK_2}\rangle = 0.
    \end{equation}
    Using this fact in the definition of $\widetilde{Q}$ and $Q$, this leads to,
    \begin{equation*}
        0 = (I-\widetilde{Q}) \F(r_1\phi^{\vK_1}+ \widetilde{G}(r_1, \chi),\chi) = (I-Q) \F(r_1\phi^{\vK_1}+ \widetilde{G}(r_1, \chi),\chi).
    \end{equation*}
    But from the uniqueness of the function $G$ in a possibly smaller neighborhood $U'\times V' \subset U\times V$, we have that,
    \begin{equation*}
        G(r_1, 0, \chi) = \widetilde{G}(r_1, \chi) \in \widetilde{X}_0 \subset X_{\overline{x}_2}.
    \end{equation*}
    We then conclude from equation~\eqref{eq:Phi2iszero} that, $(\bPhi^k(r_1, 0 , \chi) )_2= 0$, and the statement follows by defining $\Lambda$ as $\Lambda^k(r, \chi) = (\bPhi^k(r,0,\chi))_1$.
\end{proof}

\subsection{Coefficients Equations}%(Lyapunov--Schmidt Function) 

% {\color{pink} check all the computations, especially the form of the coefficients, also there is a miss match in the defintion of the signs of b, c in differnt propositions, add the explicit forms, and conclude that there are positiv..}
In this section, we compute the coefficients of the jet of the analytic function $\Phi^k$ defined above.
\begin{prop}
    \label{prop:CoefficientEquations}
    Let $\bPhi^k$ be the associated Lyapunov--Schmidt function of Theorem~\ref{thm:LyapunovSchmidt}.
    Then, it admits the following Taylor expansion,
    \begin{equation}
        \label{eq:PhiJetAnalytic}
        \bPhi{^k}(r_1, r_2, \chi^k + \Tilde{\chi}) = \Normpsiphi\begin{pmatrix}
            r_1(a^k \Tilde{\chi} - b^k r_1^2 - c^k r_2^2)\\
            r_2(a^k \Tilde{\chi} - b^k r_2^2 - c^k r_1^2)
        \end{pmatrix} + R^\Phi(r_1, r_2, \Tilde{\chi}),
    \end{equation}
    where $a^k,b^k$ and $c^k$ are real coefficients defined in \eqref{eq:CoeffDefinitiona}--\eqref{eq:CoeffDefinitionA} and an analytic remainder $R^\Phi$, with
    \begin{equation}
        R^\Phi(r_1, r_2, \Tilde{\chi}) = o(r_1^3+ r_2^3 + r_1^2r_2 + r_1 r_2^2 + \Tilde{\chi}^3 + (r_1+r_2)\Tilde{\chi}),
    \end{equation}
    
\end{prop}

\begin{proof}
    We first note from property~\eqref{eq:PhiSwapSymetry}, that we only need to compute the jet of the first component $(\Phi^k)_1$ of $\Phi^k$, as 
    \begin{equation}
        \bPhi^k(r_1, r_2, \chi^k + \Tilde{\chi}) = \begin{pmatrix}
            \bPhi^k_1(r_1, r_2, \chi^k + \Tilde{\chi})\\
            \bPhi^k_1(r_2, r_1, \chi^k + \Tilde{\chi})
        \end{pmatrix}.
    \end{equation}
    We begin by computing the Taylor expansion of the function $G$ in the space $X_0$ at zero using its implicit definition.

    In the following, for $\mathbf{r} = (r_1, r_2) \in \dR^2$, we denote $\phi^k(\mathbf{r}) = r_1\phi^{\vK_1}+r_2\phi^{\vK_2}$, and introduce,
    \begin{equation*}
        \H(\mathbf{r},\tilde{\chi}) \defeq \F(\phi^k(\mathbf{r}) + G(\mathbf{r},\tilde{\chi}), \chi^k + \tilde{\chi}).
    \end{equation*}
    We also drop the dependence of $G$ on $\mathbf{r}$ and $\Tilde{\chi}$ for the sake of notation conciseness.
    Then, we can write
    \begin{equation*}
       \H(\mathbf{r},\tilde{\chi})  = -\frac{\tilde{\chi}}{2\pi}\partial_\theta \B_\tau[\phi^k(\mathbf{r})+G] +\LinOp^k G - (\chi^k + \tilde{\chi}) \partial_\theta (B_\tau [\phi^k(\mathbf{r}) + G](\phi^k(\mathbf{r}) + G)),
    \end{equation*}
    where we used that $\phi^k(\mathbf{r})$ is in the kernel of the linearized.
    
    The first order derivative in the $r_i$ variable reads as,
    \begin{equation}
        \label{eq:DevPhidrEq1}
        \begin{split}
            \frac{\dd}{\dd r_i} \H(\mathbf{r},\Tilde{\chi}) =& -\frac{\tilde{\chi}}{2\pi}\partial_\theta \B_\tau\left[\phi^{\vK_i}+D_{i}G\right] +\LinOp^k \left(D_{i}G\right) \\
            & - (\chi^k + \tilde{\chi}) \frac{\dd}{\dd r_i} \left( \partial_\theta (\B_\tau [\phi^k(\mathbf{r})+ G](\phi^k(\mathbf{r}) + G))\right).
        \end{split}
    \end{equation}

    Using \eqref{eq:dchiG0chi0} in the last term of equation~\eqref{eq:DevPhidrEq1}, when evaluating $\dd H/\dd r_i$ at $(0,0)$, we obtain,
    \begin{equation}
        \label{eq:DevPhidr}
        \partial_{i} \H(0,0) = \LinOp^k D_{i} G (0,\chi^k) = 0.
    \end{equation}

    Now for the cross derivative with $\Tilde{\chi}$, 
    \begin{equation}
        \begin{split}
            \frac{\dd^2}{\dd r_i \dd \tilde{\chi}}\H(\mathbf{r},\Tilde{\chi}) =& -\frac{1}{2\pi}\partial_\theta \B_\tau\left[\phi^{\vK_i}+D_{i}G\right] -\frac{\Tilde{\chi}}{2\pi}\partial_\theta \B_\tau\left[D^2_{i \tilde{\chi}}G\right] +\LinOp^k \left(D^2_{r_i  \tilde{\chi}}G\right) \\
            & - \frac{\dd}{\dd r_i} \left( \partial_\theta (\B_\tau [\phi^k(\mathbf{r})+ G](\phi^k(\mathbf{r}) + G))\right)\\
            & - (\chi^k + \tilde{\chi}) \frac{\dd^2}{\dd r_i \dd \Tilde{\chi}} \left( \partial_\theta (\B_\tau [\phi^k(\mathbf{r})+ G](\phi^k(\mathbf{r}) + G))\right).
        \end{split}
    \end{equation}

    Using equations~\eqref{eq:dchiG0chi0}, \eqref{eq:dG0chi0} and \eqref{eq:dchiG0chi0}, when evaluating at $(0,0)$ we obtain,
    \begin{equation}
        \label{eq:DevPhidrdchi}
        \partial^2_{i \tilde{\chi}} \H(0,0) = -\frac{1}{2\pi}\partial_\theta \B_\tau[\phi^{\vK_i}] +\LinOp^k D^2_{ i\tilde{\chi}} G(0,0).
    \end{equation}

    The second derivatives in $\mathbf{r}$ read
    \begin{equation}
        \label{eq:DevPhid2r2Eq1}
        \begin{split}
            \frac{\dd^2}{\dd r_i \dd r_j} \H(\mathbf{r},\Tilde{\chi}) =&  - \frac{\tilde{\chi}}{2\pi} \partial_\theta \B_\tau\left[D^2_{ij} G\right] +\LinOp^k D^2_{ij} G \\
            & - (\chi^k + \tilde{\chi}) \frac{\dd^2}{\dd r_i \dd r_j} \left( \partial_\theta (\B_\tau [\phi^k(\mathbf{r}) + G](\phi^k(\mathbf{r}) + G))\right).
        \end{split}
    \end{equation}

    Using the identity \eqref{eq:dchiG0chi0} again, when evaluating at $(0,0)$ after developing the last term of the previous equation, we obtain,
    \begin{equation}
        \label{eq:DevPhid2r2}
        \begin{split}
            \partial^2_{ij} \H(0,0) = \LinOp^k D^2_{ij} G(0,0)- \chi^k  \left( \partial_\theta (\B_\tau [\phi^{\vK_i}]\phi^{\vK_j}+\B_\tau [\phi^{\vK_j}]\phi^{\vK_i})\right).
        \end{split}
    \end{equation}
    Then, using the implicit definition of $G$,
    \begin{equation}
        0 = \frac{\dd^2}{\dd r_i\dd r_j } \left((I-Q) \F(\phi^k(\mathbf{r}) + G,\chi^k+\tilde{\chi})\right)_{(0,0)} = (I-Q) \partial_{ij} \H(0,0) = \partial_{ij} \H(0,0),
    \end{equation}
    where the last equality holds as the first term is necessarily in the image of $\LinOp_k$,
    \begin{equation}
        \label{eq:DevPhid2rIinRange}
        \left( \partial_\theta (\B_\tau [\phi^{\vK_i}]\phi^{\vK_j}+\B_\tau [\phi^{\vK_j}]\phi^{\vK_i})\right) \in \Range(\LinOp_k),
    \end{equation}
    as,
    \begin{equation}
        \mathcal{F}_x\left\{\partial_\theta (\B_\tau [\phi^{\vK_i}]\phi^{\vK_j}+\B_\tau [\phi^{\vK_j}]\phi^{\vK_i})\right\}_{\vK}=0,
    \end{equation}
    for $\vK\in\{\pm \vK_1,\pm\vK_2\}$.
    % by Proposition~\ref{prop:InverseL}, we can calculate the inverse,
    % \begin{equation}
    %     (\LinOp^k)^{-1}(\partial_\theta (\B_\tau [\phi^{\vK_i}]\phi^{\vK_j}+\B_\tau [\phi^{\vK_j}]\phi^{\vK_i})),
    % \end{equation}
    % explicitly.
    Hence, we obtain,
    \begin{equation}
        \label{eq:DevPhiAterm}
        D^2_{ij} G(0,0) = \chi^k \left(\LinOp_k\right)^{-1}\left( \partial_\theta (\B_\tau [\phi^{\vK_i}]\phi^{\vK_j}+\B_\tau [\phi^{\vK_j}]\phi^{\vK_i})\right).
    \end{equation}
    We denote these functions as $A_{ij} = D^2_{ij}G(0,0)$.
    
    The third orders in $y$ follow similarly,
    \begin{equation}
        \label{eq:DevPhid3r3Eq1}
        \begin{split}
            \frac{\dd^3}{\dd r_i^2 \dd r_j} \H(y,\Tilde{\chi}) =&  - \frac{\tilde{\chi}}{2\pi} \partial_\theta \B_\tau\left[D^2_{iij} G\right] +\LinOp^k D^2_{iij} G \\
            & - (\chi^k + \tilde{\chi}) \frac{\dd^3}{\dd r_1^2 \dd r_2} \left( \partial_\theta (\B_\tau [\phi^k(\mathbf{r}) + G](\phi^k(\mathbf{r}) + G))\right)
        \end{split}
    \end{equation}
    When evaluating at $(0,0)$ after developing the last term of the previous equation, we get,
    \begin{equation}
        \label{eq:DevPhid3r}
        \begin{split}
            \partial^3_{iij} \H(0,0) &= \LinOp^k D^2_{iij} G(0,0)\\
            &- \chi^k   \partial_\theta\left( \B_\tau [A_{ii}]\phi^{\vK_j} + \B_\tau [\phi^{\vK_j}] A_{ii}\right)\\
            &- 2 \chi^k  \partial_\theta \left(  \B_\tau [A_{ij}]\phi^{\vK_i} + \B_\tau [\phi^{\vK_i}]A_{ij}\right).
        \end{split}
    \end{equation}
    The development of the jet of $\Phi^1$ follows from \eqref{eq:dchiPhi0chi0} and by taking the scalar product of \eqref{eq:DevPhidr},  \eqref{eq:DevPhidrdchi}, \eqref{eq:DevPhid2r2} and \eqref{eq:DevPhid3r} with $\psi^{\vK_1}$.
    
    From the equation~\eqref{eq:DevPhid3r}, the definition~\eqref{eq:DevPhiAterm}, using the mode preserving property~\eqref{eq:modePreservingInverseL} of the inverse operator $(\LinOp^k)^{-1}$ from Proposition~\eqref{prop:InverseL}, and the Fourier product–convolution property,  we get that the functions $\psi^{\vK_1} \partial^3_{222}H(0,0)$ and $\psi^{\vK_1} \partial^3_{ 112}H(0,0)$ have no $0$-mode in Fourier in $x$, thus,
    \begin{equation*}
        \langle \psi^{\vK_1}, \partial^3_{222}\H(0,0)\rangle = 0 \text{ and } \langle \psi^{\vK_1}, \partial^3_{112}\H(0,0)\rangle = 0 .
    \end{equation*}

    Using the previous property and after dropping all the terms that involve derivatives of $H$ that are in the image of the linearized operator, this leads to the required result with,

    {\small 
    \begin{equation*}
        a^k  = (\Normpsiphi)^{-1}\partial_{1\widetilde{\chi}} (\bPhi^k(0,0))_1, \hspace{.5 em} b^k  = - \frac{(\Normpsiphi)^{-1}}{6}\partial_{111} (\bPhi^k(0,0))_1, \hspace{.5 em} 
        c^k  = - \frac{(\Normpsiphi)^{-1}}{2}\partial_{122} (\bPhi^k(0,0))_1,
    \end{equation*}}
    that is,
    \begin{align}
        \label{eq:CoeffDefinitiona}
        a^k & = \frac{-1}{2\pi} \left\langle \psi^{\vK_1},  \partial_\theta \B_\tau [\phi^{\vK_1}]\right\rangle,\\
        \label{eq:CoeffDefinitionb}
        b^k & =  \frac{ \chi^k}{2}  \left\langle \psi^{\vK_1},  \partial_\theta( A_{11} \B_\tau [\phi^{\vK_1}] + \phi^{\vK_1} \B_\tau [A_{11}])\right\rangle,\\
        c^k & = \frac{ \chi^k}{2}\left\langle \psi^{\vK1},  \partial_\theta( A_{22} \B_\tau [\phi^{\vK_1}] + \phi^{\vK_1} \B_\tau [A_{22}])\right\rangle \nonumber \\
        \label{eq:CoeffDefinitionc}
        & \hspace{1.em} + \chi^k\left\langle \psi^{\vK_1},  \partial_\theta(\phi^{\vK_2} \B_\tau [A_{12}] +  A_{12} \B_\tau [\phi^{\vK_2}])\right\rangle,
    \end{align}
    recalling that,
    \begin{equation}
        \label{eq:CoeffDefinitionA}
        A_{ij}(\sigma_\theta) = \chi^k \left(\LinOp_k\right)^{-1}\left( \partial_\theta (\B_\tau [\phi^{\vK_i}]\phi^{\vK_j}+\B_\tau [\phi^{\vK_j}]\phi^{\vK_i})\right).
    \end{equation}
\end{proof}

The coefficients $a^k, b^k$ and $c^k$ are semi--explicit, as they are defined as integrals involving non-explicit solutions to elliptic equations with trigonometric coefficients. In the rest of this section, we prove the behavior of these coefficients as $\sigma_\theta$ goes to zero. We obtain this in Proposition~\ref{prop:PositivorNegative}. We start with the following Lemma.

\begin{lem}
    \label{prop:CoefficentFormulae}
    The coefficients $a^k,b^k, c^k$ defined in Proposition~\ref{prop:CoefficientEquations} are third order polynomial in $\tauk$, continuous in $\sigma_\theta$ in $(0,\sigma_\theta^k)$, and there exists continuous functions $b^k_{-1}, b^k_0, c^k_{-1}$ and $c^k_0$ on $[0,\sigma_k)$, such that,
    \begin{align}
        b^k(\sigma_\theta) &= \frac{1}{\sigma_\theta} b^k_{-1}(\sigma_\theta) + b^k_0(\sigma_\theta),\\
        c^k (\sigma_\theta) &= \frac{1}{\sigma_\theta} c^k_{-1}(\sigma_\theta) + c^k_0(\sigma_\theta).
    \end{align}
    Furthermore, the limits as $\sigma_\theta$ goes to $0$, are given by 
    \begin{equation}\label{lim:abc}
        \lim_{\sigma_\theta \to 0^+} a^k(\sigma_\theta) =  \frac{4\pi}{(\chi^k_0)^2}, \; \; \lim_{\sigma_\theta \to 0^+}  b^k_{-1}(\sigma_\theta ) = \Aint^{\vK_1}, \; \; \lim_{\sigma_\theta \to 0^+}  c^k_{-1}(\sigma_\theta) = \Aint^{\vK_2}.
    \end{equation} % + \underset{\sigma_\theta \to 0 }{o}(1)
    Where the terms $\Aint^{\vK_j}$ for $j=1,2$ are defined as, 
    \begin{align}
        \Aint^{\vK_j} & = -4\pi^2 \int \fX^{\vK_1} \fW^{\vK_j} \dd \theta,
    \end{align}
    with,
    \begin{equation}
        \fX^{\vK_1}   = (\fB_{\vK_1} \partial_\theta \fV^{-\vK_1}  + \fB_{-\vK_1} \partial_\theta \fV^{\vK_1})
    \end{equation}
    and the $\fW^{\vK_j}$-functions are the unique solutions in the space $L^2_{\theta, 0}(\dC)$ of the equations,
    \begin{equation}
        \partial_{\theta\theta} \fW^{\vK_j} = \partial_\theta (\fB_{\vK_j} \fU^{-\vK_j} + \fB_{-\vK_j} \fU^{\vK_j} ).
    \end{equation}
\end{lem}
\begin{proof}
The continuity of the coefficient $a_{\sigma_\theta}$  follows directly from its definition and by the continuity in $L^2_0$ of $\phisig^{\vK_1}$ and $\psisig^{\vK_1}$ in $\sigma_\theta$, from estimte~\eqref{est:ResolventContinuity} of Proposition~\eqref{prop:ResolventProperties}.

The limit is given by,
\begin{equation}
    \lim_{\sigma_\theta \to 0^+} a^k(\sigma_\theta) =  -\frac{1}{2\pi} \left\langle \psi_0^{\vK_1},  \partial_\theta \B_\tau [\phi_0^{\vK_1},]\right\rangle.
\end{equation}
% \begin{align}
%      a_0 &= -\frac{1}{2\pi} \frac{\left\langle \psi_0^1,  \partial_\theta B_\tau [\phi_0^1]\right\rangle}{\langle \psi^1_0, \phi^1_0\rangle}
% \end{align}
Recalling that, by definition,
\begin{equation}
    \int\fU^{\vK_i}_{\sigma_\theta} \dd \theta = \frac{2\pi}{\chi^k_{\sigma_\theta}},
\end{equation}
this leads to,
\begin{equation}
    \label{eq:BphiFourier}
    \mathcal{F}\{\B_\tau[\phi^{\vK_i}]\}_\ell  = \frac{2\pi}{\chi^k_{\sigma_\theta}}(\fB_{\vK_i} \delta_{\vK_i,\ell}+\fB_{-\vK_i} \delta_{-\vK_i,\ell}),
\end{equation}
and,
\begin{align*}
    -\left\langle \psi_0^{\vK_1},  \partial_\theta \B_\tau [\phi_0^{\vK_i1}]\right\rangle & = \frac{2\pi}{\chi^k_{0}} \int \fV^{-\vK_1}_0(- \partial_\theta \fB_{\vK_1} )+ \fV^{\vK_1}_0  (-\partial_\theta \fB_{ - \vK_1} )\dd \theta \\
    & = \frac{2\pi}{\chi^k_{0}} 2 \int \Real \left(\fU^{\vK_1}_0\right) \dd \theta = \frac{8\pi^2}{(\chi^k_{0})^2}.
\end{align*}

% \begin{equation}
%     \left\langle \psi_0^1,  \partial_\theta B_\tau [\phi_0^1]\right\rangle = \frac{8\pi^2}{(\chi^k)^2}
% \end{equation}

We now treat the coefficients $b^k$ and $c^k$. We then note that, from the Fourier product-convolution property, the term
\begin{equation}
    \partial_\theta (\B_\tau [\phi^{\vK_1}]\phi^{\vK_2}+\B_\tau [\phi^{\vK_2}]\phi^{\vK_1})
\end{equation} has vanishing zero Fourier mode in $x$. Thus using the form of the inversion formulas from Proposition~\ref{prop:InverseL}, we obtain that $A_{12}(\sigma_\theta)$ is a continuous function of $\sigma_\theta$ in the space $L^2_0$. Then, using the Fourier product-convolution property again, we can decompose
\begin{equation}
    \partial_\theta \left(\B_\tau [\phi^{\vK_i}]\phi^{\vK_i}\right),
\end{equation}
as a sum of a function constant in the $x$ variable, plus a function that only has $2\vK_i$ and $-2\vK_i$ as non-zero modes in Fourier in $x$. From the inversion formulas from Proposition~\ref{prop:InverseL}, we can
decompose $A_{ii}$ as,
\begin{equation}
    \label{eq:DecompositionA}
    A_{ii}(\sigma_\theta) = \frac{1}{\sigma_\theta}(\partial_{\theta \theta})^{-1}A_{ii,-1}(\sigma_\theta) + A_{ii,0}(\sigma_\theta),
\end{equation}
with
\begin{equation}
    \label{eq:DecompositionAContinuity}
    A_{ii,-1}, A_{ii,0}\in C([0,\sigma_\theta^k), L^2_0).
\end{equation}
$A_{ii,-1}$ is explicitly given by,
\begin{equation}
    A_{ii,-1}(\sigma_\theta) = 4\pi \partial_\theta (\fB_{\vK_i}\fU^{-\vK_i}_{\sigma_\theta} + \fB_{-\vK_i}\fU^{\vK_i}_{\sigma_\theta}),
\end{equation}
that is constant in $x$, so that,
\begin{equation}
    \label{eq:DecompositionABA0}
    \B_\tau [(\partial_{\theta \theta})^{-1}A_{ii,-1}] = 0.
\end{equation}

The decomposition \eqref{eq:DecompositionA}--\eqref{eq:DecompositionAContinuity} and the property \eqref{eq:DecompositionABA0} in the definitions \eqref{eq:CoeffDefinitionb}--\eqref{eq:CoeffDefinitionc} of $b_{\sigma_\theta}$ and $c_{\sigma_\theta}$ leads to the required result with,

\begin{align*}
    b_{-1}(\sigma_\theta) &\defeq -\frac{1}{2} \chi^k_{\sigma_\theta} \left\langle \B_\tau [\phi^{\vK_1}] \partial_\theta \psi^{\vK_1}, (\partial_{\theta \theta})^{-1}A_{11,-1}(\sigma_\theta)\right\rangle,\\
    c_{-1}(\sigma_\theta) &\defeq - \frac{1}{2} \chi^k_{\sigma_\theta} \left\langle  \B_\tau [\phi^{\vK_1}] \partial_\theta \psi^{\vK_1},  (\partial_{\theta \theta})^{-1}A_{22,-1}(\sigma_\theta)\right\rangle.
\end{align*}

Using the identity~\eqref{eq:BphiFourier} and the Plancherel identity, we obtain that,

\begin{align}
    b_{-1}(\sigma_\theta) & = -4\pi^2 \int \fX_{\sigma_\theta} \fW^{\vK_1}_{\sigma_\theta} \dd \theta,\\
    c_{-1}(\sigma_\theta) & = -4\pi^2  \int \fX_{\sigma_\theta} \fW^{\vK_2}_{\sigma_\theta} \dd \theta,
\end{align}
where
\begin{equation}
    \fX_{\sigma_\theta} = (\fB_{\vK_1} \partial_\theta \fV^{-\vK_1}_{\sigma_\theta} + \fB_{-\vK_1} \partial_\theta \fV^{\vK_1}_{\sigma_\theta})
\end{equation}
and the $\fW^{\vK_i}_{\sigma_\theta}$ are given as the unique solutions in the space $L^2_{\theta;0}(\dC)$ of the equations
\begin{equation}
    \partial_{\theta\theta} \fW^{\vK_i}_{\sigma_\theta} = \partial_\theta (\fB_{\vK_i} \fU^{-\vK_i}_{\sigma_\theta} + \fB_{-\vK_i} \fU^{\vK_i}_{\sigma_\theta}).
\end{equation}

From the continuity in $\sigma_\theta$ of the inverse operator $(-\fM_{\vK}+ \sigma_\theta \partial_{\thth})^{-1}$, we obtain the required result.
% \begin{align*}
%     b^k(\sigma_\theta) &= \frac{b^k_{-1}(0)}{\sigma_\theta} + o(1/\sigma_\theta),\\
%     c^k(\sigma_\theta) &= \frac{c^k_{-1}(0)}{\sigma_\theta} + o(1/\sigma_\theta).
% \end{align*}

\end{proof}

We now state a result on the explicit form of the integrals $\Aint^{\vK_j}$. The result is obtained in the Appendix~\ref{sec:AppendixIntegrals}.

\begin{prop}
    \label{prop:IntkjExplicitForms}
    The integrals $\Aint^{\vK_j}$ admit the following explicit forms,
    \begin{align}
        \Aint^{\vK_1} &= -\frac{16\pi^6 k^3 }{|\cEll|^3\lmbk^2} \Aek^3 \left(\mathsf{g}^{\mathsf{Y}}_{k}\mathsf{g}^{\mathsf{X}}_{k} + \mathsf{a}^{\mathsf{Y}}_{k}\left(\mathsf{a}^{\mathsf{X}}_{k} + \beta^1(|\Azink|)\mathsf{b}^{\mathsf{X}}_{k}\right)\right),\\
        \Aint^{\vK_2} &= -\frac{16\pi^6 k^3}{|\cEll|^3\lmbk^2} \Aek^3\left(-\mathsf{g}^{\mathsf{Y}}_k\mathsf{g}^{\mathsf{X}}_k + \mathsf{a}^{\mathsf{Y}}_k\left(\left(\frac{1-|\Azink|^2}{1+|\Azink|^2}\right)\mathsf{a}^{\mathsf{X}}_k+ \beta^2(|\Azink|)\mathsf{b}^{\mathsf{X}}_k\right)\right),
    \end{align}
    with $\beta^i : [0,1)\to \dR_+$ defined as,
    \begin{equation}
        \beta^1(s) = -(\ln(1-s^2)+s^2)s^{-4}, \; \; \beta^2(s) = -(\ln(1+s^2)-s^2)s^{-4},
    \end{equation}
    and where $\mathsf{a}^{\fX}_k, \mathsf{b}^{\fX}_k , \mathsf{g}^{\fX}_k , \mathsf{a}^{\fY}_k, \mathsf{g}^{\fY}_k$ are second order polynomials in $\tauk$, and their respective coefficients are given in Table~\ref{tb:Intaaggbterms}.
\end{prop}

\begin{table}[!ht]
\centering
\label{tb:Intaaggbterms}
\caption{Definitions of polynomials $\mathsf{a}^{\fX}_k[\tauk], \mathsf{b}^{\fX}_k [\tauk], \mathsf{g}^{\fX}_k[\tauk] , \mathsf{a}^{\fY}_k[\tauk], \mathsf{g}^{\fY}_k[\tauk]$.}
\small
\begin{tabular}{|c||c|c|c|}
\hline
 & $\tauk^2$ & $\tauk^1$ & $\tauk^0$ \\ \hline

$\mathsf{a}^{\fY}_k$
& $-\sigk(1+|\Azink|^2)$
& $-(1+|\Azink|)^2/2$
& $2\sigk|\Azink|$
\\[2mm]\hline

$\mathsf{g}^{\fY}_k$
& $\sigk|\Azink|$
& $(2+|\Azink|)/2$
& $-2\sigk$
\\[2mm]\hline 

$\mathsf{a}^{\fX}_k$
& $0$
& $2\sigk(1-|\Azink|^2)^2$
& $\begin{aligned}
    & - 4\sigk^2|\Azink|(1+|\Azink|)^2 \\
    &\quad -(1-|\Azink|^2)^2
   \end{aligned}$
\\[4mm]\hline

$\mathsf{g}^{\fX}_k$
& $0$
& $4\sigk|\Azink|(1-|\Azink|^2)^2$
& $\begin{aligned}
    &4\sigk^2\bigl((1-|\Azink|^4)
    + 4(|\Azink|^2+|\Azink|)\bigr) \\
    &\hspace{3em} - 2|\Azink|(1-|\Azink|^2)^2
   \end{aligned}$
\\[4mm]\hline

$\mathsf{b}^{\fX}_k$
& $0$
& $-2\sigk(1-|\Azink|^2)^2(1+|\Azink|^2)$
& $\begin{aligned}
    &(1-|\Azink|^2)^2(1+|\Azink|^2) \\
    &\; - 8\sigk^2|\Azink|^2(1+|\Azink|)^2
   \end{aligned}$
\\[4mm]\hline

\end{tabular}
\end{table}

\begin{lem}
    \label{lem:AsymptoticsrootsIkj}
    The integrals $\Aint^{\vK_1}$ and $\Aint^{\vK_2}$ are third order polynomials in $\tauk$, and the coefficients of $\Aint^{\vK_1}[\tauk]$ , $(\Aint^{\vK_1}+\Aint^{\vK_2})[\tauk]$, $(\Aint^{\vK_1}-\Aint^{\vK_2})[\tauk]$ have matching leading order asymptotic expansions as $\sigk$ goes to $0$. Furthermore, their roots satisfy the following expansion as $\sigk$ goes to $0$,
    \begin{subequations}\label{eq:roots1}
    \begin{align}
        \uptau_{k,0}(\sigk )& = \sigk + o(\sigk),\\
        \uptau_{k,1}(\sigk )&= \tfrac{1}{2 \sigk \ln(\sigk)} + o\left(\tfrac{1}{\sigk \ln(\sigk)}\right),\\
        \uptau_{k,2} (\sigk )& = \tfrac{-1}{\sigk} + o\left(\tfrac{1}{\sigk}\right).  
    \end{align}
    \end{subequations}
    That is, two roots that have real part going to $-\infty$ and one real root that is positive and converges to $0$.

    Finally, we have the only positive real roots of $\Aint^{\vK_1}[\tauk]$ , $(\Aint^{\vK_1}+\Aint^{\vK_2})[\tauk]$ and $(\Aint^{\vK_1}-\Aint^{\vK_2})[\tauk]$ satisfy the following property,
    \begin{equation*}
        0<\uptau_{0,\Aint^{\vK_1}+ \Aint^{\vK_2}} < \uptau_{0,\Aint^{\vK_1}} <\uptau_{0,\Aint^{\vK_1} - \Aint^{\vK_2}}.
    \end{equation*}
\end{lem}

\begin{proof}
    The asymptotic expansions of the coefficients of the polynomials $\Aint^{\vK_1}, \Aint^{\vK_1} + \Aint^{\vK_2}$ and $\Aint^{\vK_1} + \Aint^{\vK_2}$ follow from Proposition~\ref{prop:IntkjExplicitForms}. Rescaling the polynomials by the common factor $\tfrac{1024 \pi^6 k^3 }{|\cEll|^3\lmbk^2} \Aek^3$, the expansion of the roots is equivalent to the expansion of the roots of the following polynomial.

    Let $ \polyP_{\sigk}[\tauk]$ be a third order polynomial in $\tauk$, with coefficients satisfying the following expansions,
    \begin{equation*}
        \polyP_{\sigk}[\tauk] = 2\sigk^4 \ln(\sigk) \tauk^3 + 2\sigk^2 \ln(\sigk) \tauk^2 - \sigk^2 \tauk + \sigk^3 + \textbf{h.o.t},
    \end{equation*}
    where \textbf{h.o.t} is a $\tauk$-polynomial remainder of the form,
    \begin{equation*}
        o\left(\sigk^4 \ln(\sigk) \tauk^3 + \sigk^2 \ln(\sigk) \tauk^2 + \sigk^2 \tauk + \sigk^3\right).
    \end{equation*}
    Then using the following limits,
    \begin{align*}
        &\lim_{\sigk \to 0^+} \frac{1}{\sigk^3}\polyP_{\sigk}[\sigk \widetilde{\tauk}] =  -\widetilde{\tauk} + 1,\\
        &\lim_{\sigk \to 0^+} \frac{1}{\sigk \ln(\sigk)}\polyP_{\sigk}\left[\frac{\widetilde{\tauk}}{\sigk}\right] = 2(-\widetilde{\tauk} - 1) \widetilde{\tauk}^2,\\
        &\lim_{\sigk \to 0^+} \frac{\ln(\sigk)}{2\sigk}\polyP_{\sigk}\left[\frac{\widetilde{\tauk}}{2 \sigk \ln(\sigk)}\right] =  (\widetilde{\tauk} - 1)\widetilde{\tauk},
    \end{align*}
    we obtain that the roots for $\sigk\to 0^+$ are $\widetilde{\tauk}^0 = 1$, $\widetilde{\tauk}^1 = -1$ and $\widetilde{\tauk}^2 = 1$, respectively. We then apply Rouch\'e's theorem to obtain the roots for small $\sigk$, leading to \eqref{eq:roots1}.

    The statement on the ordering of the roots of $\Aint^{\vK_1}, \Aint^{\vK_1} + \Aint^{\vK_2}$ and $\Aint^{\vK_1} + \Aint^{\vK_2}$, follows by expanding further the positive real roots, using the scaling $\sigk^{-4}\polyP_{\sigk}[\sigk+\sigk^2 \widetilde{\tauk}]$ and applying Rouch\'e's theorem. We obtain the following expansions, $ \uptau_{0,\Aint^{\vK_1}+ \Aint^{\vK_2}} = \sigk + o(\sigk^2),$ $\uptau_{0,\Aint^{\vK_1}} = \sigk + \sigk^2 + o(\sigk^2),$ and $\uptau_{0,\Aint^{\vK_1} - \Aint^{\vK_2}} = \sigk + 2\sigk^2 + o(\sigk^2)$. The required result follows.
\end{proof}
Combining the results of Lemma~\ref{prop:CoefficentFormulae} and Lemma~\ref{lem:AsymptoticsrootsIkj}, we obtain the following result on the coefficients of $\bPhi^k$. That is, we recall the limits as in \eqref{lim:abc}, and hence the behaviors of  $b^k(\sigma_\theta)$ and $c^k(\sigma_\theta)$ match the behaviors of $\Aint^{\vK_1}$ and $\Aint^{\vK_2}$ as $\sigma_\theta$ goes to $0$, which is obtained in the previous lemma for small enough $\sigk$. That is, for $2\pi k\sigma_x\lesssim \lambda$. We summarize as follows.
\begin{prop}
    \label{prop:PositivorNegative}
    There exists a constant $\mathrm{c} > 0$, such that for all $\gamma>0, \sigma_c > 0 , \sigma_x > 0, \lambda > 0$ and any non-Pythagorean wave number $k \in \dN^*$, satisfying,
    \begin{equation*}
        \label{eq:smallnesssigk}
        2\pi k \sigma_x < \mathrm{c} \lambda,
    \end{equation*}
    there exists $\sigma_\theta^k >0$ such that for any $0<\sigma_\theta < \sigma_\theta^k$, there exist $\uptau^k_b > 0, \uptau^k_{b+c} > 0 $ and $ \uptau^k_{b-c} > 0 $, such that we have the following behavior of $b^k(\sigma_\theta)$ and $c^k(\sigma_\theta)$:
    \begin{itemize}
        \item $b^k > 0 $ if $0\leq \tauk < \uptau^k_b $, \text{and} $b^k < 0 $ if $\uptau^k_b < \tauk $;
        \item $b^k + c^k > 0 $ if $0\leq \tauk < \uptau^k_{b+c} $, \text{and} $b^k + c^k  < 0 $ if $\uptau^k_{b+c} < \tauk $;
        % \item $b^k - c^k > 0 $ if $0\leq \tauk < \tau^k_{b-c} $, \textbf{and} $b^k - c^k < 0 $ if $\tau^k_{b-c} < \tauk $.
        \item $b^k - c^k > 0 $ if $0\leq \tauk < \uptau^k_{b+c}  <\uptau^k_{b-c} $.
    \end{itemize}
    % Furthermore, we have the following expansion of the roots,
    % \begin{equation*}
    %         \lim_{\sigk \to 0 } \lim_{\sigma_\theta \to 0} \tfrac{\uptau^k_b}{\sigk} = 
    %         \lim_{\sigk \to 0 } \lim_{\sigma_\theta \to 0} \tfrac{\uptau^k_{b+c}}{\sigk} = \lim_{\sigk \to 0 } \lim_{\sigma_\theta \to 0} \tfrac{\uptau^k_{b-c}}{\sigk} = 1.
    % \end{equation*}
    We denote $\uptau^k_\Lambda$, $\uptau^k_\Xi $ the rescaled roots,
    \begin{equation*}
        \uptau^k_\Lambda =  \tfrac{ \uptau^k_{b}}{2\pi k},\;\;\uptau^k_\Xi =  \tfrac{ \uptau^k_{b+c}}{2\pi k}.
    \end{equation*}
\end{prop}
\begin{proof}
   The proof follows directly from Rouch\'e's theorem and the signs of the third and zeroth order coefficients, as the roots of $b^k_{\sigma_\theta, \sigk}[\tauk]$ and $\sigma_\theta b^k_{\sigma_\theta, \sigk}[\tauk]$ are the same, and the second is continuous in $\sigma_\theta$, with
\begin{equation*}
    \lim_{\sigma_\theta \to 0} \sigma_\theta b^k_{\sigma_\theta, \sigk}[\tauk] = \Aint^{\vK_1}_{\sigk}[\tauk].
\end{equation*}
The same argument applies to $b^k+c^k $ and $b^k - c^k$. 
\end{proof}
\subsection{Existence of Spots and Lanes}
We here conclude this section by proving the existence of bifurcation branches, given the Taylor expansion of the LS function obtained in the previous subsection.
\begin{thm}
    \label{thm:ExistenceSolCurves}
    For all $\gamma>0, \sigma_c > 0 , \sigma_x > 0, \lambda > 0$ and $k \in \dN^*$ a non-Pythagorean wave number. Let $\mathrm{c} >0$ and $\sigma_\theta^k >0$ be given by Proposition~\ref{prop:PositivorNegative}, assuming $2\pi k \sigma_x < \mathrm{c}\lambda$. For all $0<\sigma_\theta < \sigma_\theta^k$, 
    suppose that $0\leq \tau$ and $\tau \neq \uptau^k_\Lambda$ (resp. $\tau \neq \uptau^k_\Xi$), then there exists a non-trivial bifurcation curve, $\zeta^\Lambda \in C((-\delta, \delta), \dR^2\times\dR)$, (resp. $\zeta^\Xi \in C((-\delta, \delta), \dR^2\times\dR)$) of the form,
    \begin{equation}
        \zeta^\Lambda(s) = (s,0, \chi^k +s^2\widetilde{\chi}^\Lambda(s)),
    \end{equation}
    (resp. $\zeta^\Xi(s) = (s,s, \chi^k +s^2\widetilde{\chi}^\Xi(s))$) with $\widetilde{\chi}^\Lambda$ (resp. $\widetilde{\chi}^\Xi$) analytic around $0$ and $\widetilde{\chi}^\Lambda(0) = b^k/a^k $ (rep. $\widetilde{\chi}^\Xi(0) = (b^k + c^k)/a^k$), such that,
    \begin{equation*}
        \bPhi^k\left(\zeta^{\Lambda}(s)\right) = 0 \; \;  \forall s \in  (-\delta, \delta ),
    \end{equation*}
    (resp. $\bPhi^k(\zeta^\Xi(s)) = 0 )$.
\end{thm}

% Recalling the equivalence with the original problem, and the space decomposition used in the Lyapunov-Schmidt reduction. We obtain two solution curves $u^\Lambda, u^\Xi \in C((-\delta,\delta), X \times \dR)$ for the equation $F$, of the form,
% \begin{align}
%     u^{\Lambda}(s) &= (f^\Lambda(s), \chi^k_{\sigma_\theta} +s^2\widetilde{\chi}^\Lambda(s))\\
%     u^{\Xi}(s) &= (f^\Xi(s), \chi^k_{\sigma_\theta} +s^2\widetilde{\chi}^\Xi(s))
% \end{align}
% with,
% \begin{align}
%    f^\Lambda(s) &= s\phi^1_{\sigma_\theta} + o(s) \text{ in } X_{\overline{x}_2},\\
%    f^\Xi(s) &= s(\phi^1_{\sigma_\theta} + \phi^2_{\sigma_\theta}) + o(s) \text{ in } X.
% \end{align}

% We emphasize that $f^\Lambda \in X_{\overline{x}_2}$, that is, functions that are constant in the $x_2$ variable, as a corollary of preservation of constancy proved in Theorem~\ref{thm:LyapunovSchmidt}.

\begin{proof}
    We prove the existence of non-trivial solutions to $\bPhi^k$, by using the reduced dimension functions $\Lambda$, $\Xi$ of Theorem~\ref{thm:LyapunovSchmidt}. Then according to their definitions, Proposition~\ref{prop:CoefficientEquations} and Proposition~\ref{prop:PositivorNegative}, we obtain by assumption on $\tau$, that both $\Lambda$ and $\Xi$ can be represented as a function $\Upsilon : \dR\times \dR \to \dR$, in the form,
    \begin{equation}
    \label{eq:Ups1}
        \Upsilon(r,\tilde{\chi}) = r(\omega_1 \tilde{\chi} - \omega_2 r^2) + R^\Upsilon(r,\chi),
    \end{equation}
    with $\omega_1>0, \omega_2 \neq 0$. For both $\Lambda$ and $\Xi$ we have $\omega_1 = a^k$ and for $\Lambda$ we have $\omega_2 = b^k$ and for $\Xi$ we have $\omega_2 = b^k+ c^k$. Finally, $R^\Upsilon$ is an analytic remainder, such that
    \begin{equation}
        R^\Upsilon(r,\tilde{\chi}) = o\left(r^3 + r\Tilde{\chi} +\Tilde{\chi}^2\right).
    \end{equation}
    We will therefore prove in generality the result for $\Upsilon$. We introduce the following ansatz curve $(s\mapsto (s, s^2\chi')) \in C((-\delta, \delta),\dR\times \dR)$. Plugging this ansatz into equation \eqref{eq:Ups1}, yields
    \begin{equation}
        \Upsilon(s,s^2\chi') = s^3\left((\omega_1 \chi' - \omega_2) + R_1(s,\chi')\right),
    \end{equation}
    with $R_1$ an analytic function, defined as $R_1(s,\chi') = R^\Upsilon(s,\chi' s^2)/s^3 = o(\chi')$. By the Newton's polygon method, $R_1$ satisfies $R_1(0,\chi') \equiv \partial_{\chi'}R_1(0,\chi') \equiv 0$, within the domain of convergence of the analytic series. We thus apply the analytic implicit function theorem to the equation,
    \begin{equation}
        \label{eq:ImplicitChiUpsilon}
        (\omega_1\chi' - \omega_2) + R_1(s,\chi') = 0,
    \end{equation}
    as $(0,\omega_2/\omega_1)$ is a solution, and the derivative of the equation with respect to $\chi'$ evaluated at $(0,\omega_2/\omega_1)$, is equal to $\omega_1>0$. This implies the existence of a non-trivial curve $(-\delta,\delta)\ni s\mapsto (s,s^2\chi'(s))$, with $\chi'$ the analytic implicit function defined by equation~\eqref{eq:ImplicitChiUpsilon}, such that $\chi'(0) = \omega_2/\omega_1$.
    The curve is a solution of,
    \begin{equation*}
        \Upsilon(s,s^2\chi'(s)) = 0 \ \ \forall s \in (-\delta, \delta),
    \end{equation*}
    this concludes the proof.
\end{proof}
    
    \section{Stability of the First Supercritical Branch}
\label{sec:StabilityAnalysis}
In this section, we apply the Principle of Reduced Stability \cite[Chapiter 1, I.18]{kielhofer1983principle} to study the stability of the stationary solutions of the bifurcation curves constructed in Section~\ref{sec:BirfucationCurves}. Firstly, we note that at a bifurcation point for $k\geq 2$, there already exist $(k-1)$--unstable eigenvalues of the linearized operator around the homogeneous solution. As the stability of bifurcation curves inherits the stability of the constant solution from a continuity argument, as far as our local analysis is concerned, we are restricted to studying the stability of the first bifurcation curves for $k = 1$, the first time when the constant solution loses its stability. Secondly, as mentioned in Section~\ref{sec:MainRes}, because of the translation invariance of the solutions, the stability must be studied in the quotient space $X$. 
%Thirdly, the stability analysis can be performed in the supercritical case, that is, when $0\leq \tau < \uptau^1_\Xi$.
% These correspond to the loss of stability of the homogeneous solution during the previous $(k-1)$-bifurcations. In \cite{rakodeWit2025}, it was already proven that each of these unstable eigenmodes corresponds to an unstable solution of the non-linear time-dependent problem \eqref{eq:ftPDE}, giving a dimensional lower bound of the global attractor.
% As new bifurcation curves emerge from the homogeneous one, the linearized operators around these solutions are continuous along the curves and thus inherit locally the instability of the homogeneous one. 
% At this point there is only one eigenvalue of the linearized operator around the homogeneous solution with positive real part.
% {\color{pink}
% \footnotesize
% \begin{equation}
%     \int \frac{-\partial_\theta B_{\vK_1}}{\mu + \fM_{\vK_1}}\dd \theta = 2\pi\left(1 - \frac{\mu/\lmbk  + \sigk}{\sqrt{1+(\mu/\lmbk  + \sigk)^2}} + \left(\frac{1+2(\mu/\lmbk + \sigk)^2}{\sqrt{1+(\mu/\lmbk + \sigk)^2}} - 2(\mu/\lmbk + \sigk)^2\right)\tauk\right)
% \end{equation}
% }
\begin{thm}
    \label{thm:StabilitySupercritical}
    Let $\gamma>0, \sigma_c > 0 , \sigma_x > 0, \lambda > 0$, and let $k = 1$. Let $\mathrm{c} >0$ and $\sigma_\theta^1 >0$ be given by Proposition~\ref{prop:PositivorNegative}, assuming $2\pi  \sigma_x < \mathrm{c}\lambda$. Suppose that $0<\sigma_\theta < \sigma_\theta^1$, and $0\leq \tau < \uptau^1_\Xi$.
    
    Then, denoting $\mathrm{L}^\Lambda(s)$ and $\mathrm{L}^\Xi(s)$ the linearized operators along the bifurcation curves,
    \begin{align}
        \mathrm{L}^\Lambda(s) & = D_f\F(f^\Lambda(s), \chi^\Lambda(s)) \in L(X,Z),\\
        \mathrm{L}^\Xi(s) & = D_f\F(f^\Xi(s), \chi^\Xi(s)) \in L(X,Z).
    \end{align}
    there exists a constant $0<\delta_2$, such that,
        \begin{equation}
            \Sigma(\mathrm{L}^\Lambda(s)) \cap \dC_+ \neq \emptyset \text{ and } \Sigma(\mathrm{L}^\Xi(s))\subset \dC_-   \ \ \forall s \in (-\delta_2, \delta_2), |s| > 0,
        \end{equation}
\end{thm}

\begin{proof}
    We first note that the functional framework of this section is consistent with the Principle of Reduced Stability (P.R.S.) \cite{kielhofer1983principle,kielhofer2006bifurcation}. In particular, $\mathrm{0}$ is semi-simple and the space decomposition for the P.R.S. coincides with our Lyapunov--Schmidt functional. The spectral analysis in Proposition~\ref{prop:SpectralGapatFirstBif} shows that, at $k = 1$, the only eigenvalue with positive real part is $0$. Since the operator is sectorial with a compact inverse, its spectrum consists of isolated eigenvalues with no accumulation points except at $-\infty$. Consequently, the local linear stability along the bifurcation curves is governed solely by the variation of the critical eigenvalue $0$, because the rest of the spectrum stays uniformly away from the imaginary axis
    % (see Kato~\cite{kato2013perturbation} p.212 Theorem 3.16, or using that the operators $\LinOp(s)$ (Theorem 1.3.2 p.19 \cite{henry2006geometric}) are sectorial thus the continuity of the Riesz projector will preserve the spectrum separated {\color{pink} phrasing}). 
    (using that $\B$ is a relatively compact perturbation and the resolvent of $\mathrm{L}$ satisfies a decay estimate).
    By the P.R.S., the expansion of the critical eigenvalue coincides with the eigenvalue of the linearized Lyapunov--Schmidt functional.

    Using the analytic expansions of $\bPhi^1$~\eqref{eq:PhiJetAnalytic} and the solution curves from Theorem~\ref{thm:ExistenceSolCurves}, a straight forward computation leads to,
    \begin{equation*}
        D_y \bPhi^1 ( \zeta^\Lambda(s)) = s^2 \begin{pmatrix}
            -2b^k & 0\\
            0 & b^k - c^k
        \end{pmatrix} + R^{D\Phi, \Lambda}(s)
    \end{equation*}
    and
    \begin{equation*}
        D_y \bPhi^1 ( \zeta^\Xi(s)) = s^2 \begin{pmatrix}
            -2b^k & -2c^k\\
            -2c^k & -2b^k
        \end{pmatrix} + R^{D\Phi, \Xi}(s).
    \end{equation*}
    The spectra of the dominating terms are respectively given by 
    \begin{equation*}
        \left\{\mu^\Lambda_1 = - 2b^k, \mu^\Lambda_2 = b^k- c^k\right\} \text{ and } \left\{\mu^\Xi = - 2(b^k + c^k), \mu^\Lambda_2 = -2(b^k- c^k)\right\}.
    \end{equation*}
    So that, since $\mathrm{0}$ is at most algebraic multiplicity one, Theorem I.18.1 of~\cite{kielhofer2006bifurcation} applies, and the critical eigenvalues of $\mathrm{L}^\Lambda(s)$ and $\mathrm{L}^\Xi(s)$ are respectively given by,
    \begin{equation*}
        s^2\mu^\Lambda_i + o(s^2) \text{ and } s^2\mu^\Xi_i + o(s^2).
    \end{equation*}
    According to Proposition~\ref{prop:PositivorNegative} and the hypothesis on $\tau$, we obtain that $\zeta^\Lambda$ is locally linearly unstable and $\zeta^\Xi$ is locally linearly stable.
\end{proof}

    \section*{Acknowledgement}
    The authors would like to thank Prof. Johannes Zimmer from the Technical University of Munich for hosting the second author at TUM and thereby fostering this collaboration.

    \begingroup

\renewcommand{\thesubsection}{\Alph{subsection}}
\renewcommand{\thethm}{\thesubsection.\arabic{thm}}
\renewcommand{\theprop}{\thesubsection.\arabic{prop}}
\renewcommand{\thelem}{\thesubsection.\arabic{lem}}
\renewcommand{\theequation}{\thesubsection.\arabic{equation}}

\appendix
\section*{Appendix}
\addcontentsline{toc}{section}{Appendix}
\setcounter{equation}{0}

% If you don't want "Appendix" to appear in the TOC twice, you can do:
%\addcontentsline{toc}{section}{Appendix}

%%%%%%%%%%%%%%%%%%%%%%%%%%%%%%%%%%%%%%%%%%%%%%%%%%%%%%%%%%%%%%%%%%%%%%%%%%%%
%
%                       Analysis of Integral Terms
%
%%%%%%%%%%%%%%%%%%%%%%%%%%%%%%%%%%%%%%%%%%%%%%%%%%%%%%%%%%%%%%%%%%%%%%%%%%%%

\subsection{Analysis of Integral Quantities}
\label{sec:AppendixIntegrals}

In this section, we will use the following fact; if $\psi \in L^2_\theta(\dC)$ is such that,
there exists a meromorphic function $\Psi : \dC \to \dC $, such that, $\psi(\theta) = \Psi(e^{i\theta})$, then, from the residue theorem,
\begin{equation*}
    \mathcal{F}\{\psi\}_n = \frac{1}{2\pi}\int_0^{2\pi}\psi(\theta) e^{-in\theta}\dd \theta = \frac{1}{2\pi i }\ointop_{|z| = 1} \Psi(z) z^{-n-1}\dd z = \Residu_{|z|=1}( \Psi(z) z^{-n-1}).
\end{equation*}

% Recall that, let $c$ be a pole of order $p$, then
% \begin{equation}
%     \Residu (\Psi, c) = \frac{1}{(p-1)!}\lim_{z \to c} \frac{\dd^{p-1}}{\dd z^{p-1}} \left((z-c)^p \Psi(z)\right).
% \end{equation}
%
%
%   Previous form in the Trash
%
%

% {\color{pink}
% We now treat the integrals,
% \begin{align}
%     b_{-1}(\sigma_\theta) & = -4\pi^2 \int \fX^{\vK_1}_{\sigma_\theta} \fW^{\vK_1}_{\sigma_\theta} \dd \theta,\\
%     c_{-1}(\sigma_\theta) & = -4\pi^2 \int \fX^{\vK_1}_{\sigma_\theta} \fW^{\vK_2}_{\sigma_\theta} \dd \theta,
% \end{align}

% where we recall the following defintions,
% \begin{align}
%     \fX^{\vK_1}_{\sigma_\theta} &= (\fB_{\vK_1} \partial_\theta \fV^{-\vK_1}_{\sigma_\theta} + \fB_{-\vK_1} \partial_\theta \fV^{\vK_1}_{\sigma_\theta}),\\
%     \partial_{\theta\theta} \fW^{\vK_i}_{\sigma_\theta} &= \partial_\theta \fY^{\vK_i}_{\sigma_\theta},\\
%     \fY^{\vK_i}_{\sigma_\theta} &= (\fB_{\vK_i}\fU^{-\vK_i}_{\sigma_\theta} + \fB_{-\vK_i}\fU^{\vK_i}_{\sigma_\theta}).
% \end{align}
% }
\subsubsection{Integral Coefficients}

In this section, we compute explicitly the integrals $\Aint^{\vK_1}, \Aint^{\vK_2}$ defined as follows,

\begin{equation*}
    \Aint^{\vK_j}= -4\pi^2 \int \fX^{\vK_1}\fW^{\vK_j} \dd \theta,
\end{equation*}
where we recall the following definitions,
\begin{align*}
    \fX^{\vK_1} &= (\fB_{\vK_1} \partial_\theta \fV^{-\vK_1}_0 + \fB_{-\vK_1} \partial_\theta \fV^{\vK_1}_0),\\
    \partial_{\theta\theta} \fW^{\vK_i} &= \partial_\theta \fY^{\vK_i},\\
    \fY^{\vK_i}&= (\fB_{\vK_i}\fU^{-\vK_i}_0+ \fB_{-\vK_i}\fU^{\vK_i}_0).
\end{align*}

We use Parseval's identity and the explicit form of the inverse Laplace in Fourier in the $\theta$-variable to obtain closed forms. We start by introducing the Fourier transform in $\theta$ of the functions of interest,

\begin{equation*}
    \mathsf{y}^{\vK_i}_{\omega} \defeq \mathcal{F}_\theta\{\fY^{\vK_i}\}_{\omega},\ \  \mathsf{x}^{\vK_1}_{\omega} \defeq \mathcal{F}_\theta\{\fX^{\vK_i}\}_{\omega},
\end{equation*}

with the Fourier transform given by $\mathcal{F}_\theta\{g(\theta)\}_\omega=\frac{1}{2\pi}\int_0^{2\pi}e^{-i\omega\theta}g(\theta)\dd\theta$. The equation for $\fW^{\vK_i}$ writes in Fourier as,
\begin{align*}
     -\omega^2 \mathsf{w}^{\vK_i}_{\omega}  &=  i \omega\mathsf{y}^{\vK_i}_{\omega}, \forall \omega \in \dZ/\{0\}\\
     \mathsf{w}^{\vK_i}_{0} &= 0,
\end{align*}

by definition of the inverse Laplacian over the $L^2$-functions with average $0$.

From Parseval's identity, we rewrite the integrals as the following series, 
\begin{equation*}
     \Aint^{\vK_j}  =  -8\pi^3 \sum_{\omega \in \dZ} \mathsf{w}^{\vK_j}_{\omega}\mathsf{x}^{\vK_1}_{-\omega} =  8\pi^3 \sum_{\omega \in \dZ/\{0\}} \frac{i}{\omega}\mathsf{y}^{\vK_j}_{\omega}\mathsf{x}^{\vK_1}_{-\omega} = 16\pi^3 \sum_{1\leq\omega} \Real \left(\frac{i}{\omega}\mathsf{y}^{\vK_j}_{\omega}\mathsf{x}^{\vK_1}_{-\omega}\right).
\end{equation*}
Where we used that $\fX^{\vK_1}, \fY^{\vK_i}$ and $\fW^{\vK_i}$ are real and thus their opposite Fourier modes are complex conjugated. We finally obtain, that,
\begin{equation}
    \label{eq:IntkjSeriesform}
    \Aint^{\vK_j}= 16\pi^3 \sum_{1\leq\omega} \frac{1}{\omega} \Imag \left(\overline{\mathsf{y}}^{\vK_j}_{\omega}\mathsf{x}^{\vK_1}_{\omega}\right).
\end{equation}

We also note, from Proposition~\ref{prop:RotationUKukomega} that,
\begin{equation}
    \fY^{\vK_2}_{\sigma_\theta} = \fY^{\vK_1}_{\sigma_\theta}\left(\cdot - \pi/2\right),
\end{equation}
for every $\sigma_\theta\geq 0$, and in particular for $\sigma_\theta = 0$. So that we have the following relation in Fourier space,

\begin{equation}
    \label{eq:appendixrelationyk1yk2}
    \mathsf{y}^{\vK_2}_{n} = e^{-i\frac{n}{2}\pi}\mathsf{y}^{\vK_1}_{n}= (-1)^{\frac{n}{2}}\mathsf{y}^{\vK_1}_{n}.
\end{equation}

We thus only need to compute the Fourier coefficients of $ \fY^{\vK_1}$. We start by developing $\fX^{\vK_j}$ and $\fY^{\vK_j}$ as the resolvent $\fOpR^{\vK_j}_{0}$ is explicit at $\sigma_\theta = 0$, this leads to,

\begin{align*}
    \fX^{\vK_j} & = (\fB_{\vK_j} \partial_\theta \fV^{-\vK_j}_{0} + \fB_{-\vK_j} \partial_\theta \fV^{\vK_j}_{0}) = -\left( \frac{ \fB_{\vK_j} \partial_\theta \fM_{\vK_j}}{(\fM_{\vK_j})^2} +  \frac{\fB_{-\vK_j} \partial_\theta \fM_{-\vK_j}}{(\fM_{-\vK_j})^2}\right) \\ 
    & = \frac{-2 \Real ( \fB_{\vK_j} \partial_\theta \fM_{\vK_j} (\fM_{-\vK_j})^2 )}{|\fM_{\vK_j}|^4}  \eqdef \frac{-\fC^{\vK_j}_{\fX}}{|\fM_{\vK_j}|^4}\\
    \fY^{\vK_j} &= (\fB_{\vK_j}\fU^{-\vK_j}_{0} + \fB_{-\vK_j}\fU^{\vK_j}_{0})= -\left(\frac{\fB_{\vK_j} \partial_\theta \fB_{-\vK_j}}{\fM_{-\vK_j}} +  \frac{ \fB_{-\vK_j} \partial_\theta \fB_{\vK_j}}{\fM_{\vK_j}}\right)   \\ 
    &= \frac{-2\Real(\fB_{\vK_j} \partial_\theta \fB_{-\vK_j} \fM_{\vK_j} )}{|\fM_{\vK_j}|^2} \eqdef \frac{-\fC^{\vK_j}_{\fY}}{|\fM_{\vK_j}|^2}
\end{align*}

%\defeq \mathsf{d}^{\fX^{\vK_i}}_{n} 

We then decompose the functions as products and introduce their Fourier modes as,
\begin{align}
    \label{eq:dM2dM4defprod}
    \mathsf{d}^{\vK_j,2}_{n} &\defeq \mathcal{F}_\theta\left\{|\fM_{\vK_j}|^{-2}\right\}_{n},\hspace{-7em}
    & \mathsf{d}^{\vK_j,4}_{n} &\defeq \mathcal{F}_\theta\left\{|\fM_{\vK_j}|^{-4}\right\}_{n}, \\
    \label{eq:CYjCXjdefprod}
    \mathsf{c}^{\fY_j}_n &\defeq \mathcal{F}_\theta\left\{\fC^{\vK_j}_{\fY}\right\}_{n},\hspace{-7em}
    & \mathsf{c}^{\fX_j}_n &\defeq \mathcal{F}_\theta\left\{\fC^{\vK_j}_{\fX}\right\}_{n}.
\end{align}

The Fourier modes of the products is given by the Cauchy Products,
\begin{equation}
    \label{eq:Appynxnconvdef}
    \mathsf{y}^{\vK_j}_{n} = -\sum_{l\in \dZ} \mathsf{c}^{\fY_j}_l \mathsf{d}^{\vK_j,2}_{n-l},\; \; \mathsf{x}^{\vK_j}_{n} =  -\sum_{l\in \dZ} \mathsf{c}^{\fX_j}_l \mathsf{d}^{\vK_j,4}_{n-l}.
\end{equation}

In the following lemmas, we give explicitly the Fourier coefficients \eqref{eq:dM2dM4defprod} and \eqref{eq:Appynxnconvdef}.

\begin{lem}
    \label{lem:MhatmdFourier}
    The Fourier coefficients $ \mathsf{d}^{\vK_1,2}_{n},  \mathsf{d}^{\vK_1,4}_{n}$ defined in~\eqref{eq:dM2dM4defprod} admit the following explicit forms,
    \begin{align}
        \mathsf{d}^{\vK_1,2}_{n} &= \frac{\Aek}{\lmbk^2} |\Azink|^{|n|/2}\begin{cases}
            2(-1)^{n/2} \text{ if $n$ is even},\\
            0 \text{ if $n$ is odd}.
        \end{cases}\\
        \mathsf{d}^{\vK_1,4}_{n} &= (|n| + \Afk)\frac{\Aek^2}{\lmbk^4} |\Azink|^{|n|/2}\begin{cases}
            2(-1)^{n/2} \text{ if $n$ is even},\\
            0 \text{ if $n$ is odd},
        \end{cases}
    \end{align}
    where $\Azink$, $\Afk$ and $\Aek$ are defined as,
    { \small
    \begin{equation}
        \Azink = -(2\sigk^2+1) + \sqrt{(2\sigk^2+1)^2-1},\;
        \Afk =
    2\left(\frac{1+|\Azink|^2}{1-|\Azink|^2}\right), \; \Aek = \frac{1}{\sqrt{(2\sigk^2+1)^2-1}}.
    \end{equation}}
\end{lem}

\begin{proof}
    Let $ n \geq 0 $,
    \begin{equation*}    
        \mathsf{d}^{\vK_1,2}_{-n} = \frac{1}{2\pi} \int \frac{e^{i n \theta}}{|\fM_{\vK_1}|^{2}} \dd \theta = \frac{-i}{2\pi} \ointop_{|z| = 1} f(z) z^{n-1} dz = \Residu_{|z| = 1}(f(z) z^{n-1}).
    \end{equation*}
    
    with,
    \begin{align*}
        f(z) & =  \frac{4z^{2}}{\lmbk^2 (z - i \sqrt{|\Azink|})(z + i \sqrt{|\Azink|})(z^2 - \Azoutk)},
    \end{align*}

    % The computations to check that f(ei\theta) = 1/|Mk1(\theta)|^2 is in the trash at #AppendixM2computation
    
    % The computations for 1/|Mk2(\theta)|^2 is in the trash at #AppendixMk22computation even if it is not required here anymore
    
    where,
    \begin{align}
        \Azink =& -(2\sigk^2+1) + \sqrt{(2\sigk^2+1)^2-1},\\
        \Azoutk =& -(2\sigk^2+1) - \sqrt{(2\sigk^2+1)^2-1}.
    \end{align}
    
    So that, the following identity holds,
    \begin{equation*}
        |\fM_{\vK_1}(\theta)|^{-2} = f\left(e^{i\theta}\right).
    \end{equation*}
    
    Applying the Residue theorem to each of the two simple poles of $f(z)z^{n-1}$ inside the unit circle, this leads to,
    \begin{align*}
        \Residu_{\UnitCircleResThm}(f^1(z) z^{n-1}) &= \frac{\Aek}{\lmbk^2}|\Azink|^{n/2}[(i)^n + (-i)^n]
    \end{align*}
    with, 
    \begin{equation*}
        \Aek \defeq \frac{2}{(\Azink - \Azoutk)} = \frac{1}{\sqrt{(2\sigk^2+1)^2-1}}.
    \end{equation*}
    
    Using the mode symmetry for the positive modes, as they are associated with a real-valued function, we obtain,
    
    \begin{equation*}
        \mathsf{d}^{\vK_1,2}_{n} = \frac{\Aek}{\lmbk^2} |\Azink|^{|n|/2}\begin{cases}
            2(-1)^{n/2} \text{ if $n$ is even},\\
            0 \text{ if $n$ is odd}.
        \end{cases}
    \end{equation*}
    % The fourier modes of dk2 are in the trash at #appendixdk2
    Similarly, for $\mathsf{d}^{\vK_1,4}_{n}$, we start with the negative modes.
    
    Let $n\geq 0$, and consider,
    \begin{align*}
        \mathsf{d}^{\vK_1,4}_{-n} &= \frac{1}{2\pi} \int \frac{e^{i n \theta}}{|\fM_{\vK_1}|^{4}} \dd \theta = \frac{-i}{2\pi} \ointop_{|z| = 1 } g(z) z^{n-1} dz = \Residu_{|z| = 1 }(g(z) z^{n-1}).
    \end{align*}
    
    where,
    \begin{equation*}
        g(z) = \frac{16z^{4}}{\lmbk^4 (z - i \sqrt{|\Azink|})^2(z + i \sqrt{|\Azink|})^2(z^2 - \Azoutk)^2},
    \end{equation*}
    
    so that $g(e^{i\theta}) = |\fM_{\vK_1}(\theta)|^{-4}$.
    
    Applying the Residue theorem, on the two poles of order two inside the unit circle for negative modes, and using the mode symmetry for positive modes, we obtain,
    
    \begin{equation*}
        \mathsf{d}^{\vK_1,4}_{n} = \Residu_{\UnitCircleResThm}(g(z) z^{n-1}) = (|n| + \Afk)\frac{\Aek^2}{\lmbk^4} |\Azink|^{|n|/2}\begin{cases}
            2(-1)^{n/2} \text{ if $n$ is even},\\
            0 \text{ if $n$ is odd},
        \end{cases}
    \end{equation*}
    
    with,
    \begin{equation*}
        \Afk = %2 + 2\frac{|\Azink|}{\sqrt{(2\sigk^2+1)^2 - 1}} =2 \frac{2\sigk^2+1}{\sqrt{(2\sigk^2+1)^2 - 1}} = 
        2\left(\frac{1+|\Azink|^2}{1-|\Azink|^2}\right).
    \end{equation*}
    
\end{proof}

\begin{lem}
    \label{lem:xyfourier}
    The Fourier modes $\Ay^{\vK_j}_n, \Ax^{\vK_1}_n$ of $\fY^{\vK_j}_{0} , \fX^{\vK_1}_{0}$ admit the following explicit forms, for any $n \geq 4$ even, we have that,

    \begin{align}
        \Ax^{\vK_1}_{n} & = -\frac{\pi k }{2\cEll\lmbk} \Aek^2(-1)^{\frac{n}{2}}|\Azink|^{\frac{n-4}{2}}(\mathsf{a}^{\fX}_k n + 2(1-|\Azink|^2)^{-1}\mathsf{b}^{\fX}_k),\\
        \Ay^{\vK_1}_{n} & = -i\frac{2\pi^2k^2 }{|\cEll|^2\lmbk} \Aek (1-|\Azink|^2)(-1)^{\frac{n}{2}}|\Azink|^{\frac{n-4}{2}} \mathsf{a}^{\fY}_k,\\
        \Ay^{\vK_2}_{n} & = -i\frac{2\pi^2k^2 }{|\cEll|^2\lmbk} \Aek (1-|\Azink|^2)|\Azink|^{\frac{n-4}{2}} \mathsf{a}^{\fY}_k,
    \end{align}

    and for $n = 2$, 

    \begin{align}
        \Ax^{\vK_1}_{n} & = -\frac{\pi k }{2\cEll\lmbk} \Aek^2(-1)^{\frac{n}{2}}|\Azink|^{\frac{n-4}{2}}(\mathsf{a}^{\fX}_k n + 2(1-|\Azink|^2)^{-1}\mathsf{b}^{\fX}_k),\\
        \Ay^{\vK_1}_{n} & = -i\frac{2\pi^2k^2 }{|\cEll|^2\lmbk} \Aek (1-|\Azink|^2)(-1)^{\frac{n}{2}}|\Azink|^{\frac{n-4}{2}} \mathsf{a}^{\fY}_k,\\
        \Ay^{\vK_2}_{2} &= i\frac{2\pi^2k^2 }{|\cEll|^2\lmbk}\Aek(1-|\Azink|^2) \mathsf{g}^{\mathsf{Y}}_k,
    \end{align}
    
    where $\mathsf{a}^{\fX}_k, \mathsf{b}^{\fX}_k , \mathsf{g}^{\fX}_k , \mathsf{a}^{\fY}_k, \mathsf{g}^{\fY}_k$ are second order polynomials in $\tauk$, and their respective coefficients are given in Table~\ref{tb:Intaaggbterms}. The negative modes are obtained by complex conjugation, and the odd modes are zero.
\end{lem}

\begin{proof}
    Using the definition in Table~\ref{tb:FourierMultipliers}, in the form 
    \eqref{eq:CYjCXjdefprod}, with the following renormalizations,
    \begin{equation}
        \Ac^{\mathsf{Y}_j}_{n} = i\frac{\pi^2k^2 \lmbk}{|\cEll|^2} \Tilde{\Ac}^{\mathsf{Y}_j}_{n}, \ \ \ \Ac^{\mathsf{X}_1}_{n} = \frac{\pi k \lmbk^3}{4\cEll} \Tilde{\Ac}^{\mathsf{X}_1}_{n},
    \end{equation}
    from explicit computation we obtain,
    \begin{center}
        \begin{equation*}
            \begin{array}{|c||c|c|c|}
                \hline
                 & -4 & -2 & 0 \\ \hline
                \Tilde{\Ac}^{\mathsf{Y}_1}_{n} & (\sigk\tauk^2+\tauk/2) &(-\tauk+2\sigk) & 0
                % \\[2mm]\hline
                % \Tilde{\Ac}^{\mathsf{Y}_2}_{n} & (\sigk\tauk^2-\tauk/2) & -(\tauk+2\sigk) & 0
                \\[2mm]\hline
                \Tilde{\Ac}^{\mathsf{X}_1}_{n}& 2\sigk\tauk-1 & 4\sigk^2 & -4\sigk\tauk+2-8\sigk^2\\
            \hline
            \end{array}
        \end{equation*}
    \end{center}
    
    Where we only specify the negative non-vanishing modes, the modes $4$ and $2$ are obtained as the complex conjugates, as the functions are real, and the other modes are zero. Then applying Cauchy product formula as in~\eqref{eq:Appynxnconvdef}, with the explicit form obtained in Lemma~\ref{lem:MhatmdFourier}, leads to the required result.
\end{proof}

Summing up the above lemmas~\ref{lem:MhatmdFourier}-\ref{lem:xyfourier} in formula~\eqref{eq:IntkjSeriesform} leads to the stated result in Proposition~\ref{prop:IntkjExplicitForms}.

% \underline{6. Asymptotic behavior of the roots}
% The roots of the polynomials $\Aint^1$ and $\Aint^2$ do not admit a polished form. We thus propose to study their asymptotic behavior as $\sigma_k$ goes to $0$, that is either $\sigma_x$ goes to $0$ or $\lambda$ to $+\infty$.

% From the definition of the coefficients, we derive the following asymptotic expansions:

% \begin{center}
%     \begin{equation*}
%         \begin{array}{|c||c|c|c|c|}
%             \hline
%             & \tauk^3 & \tauk^2 & \tauk^1 & \tauk^0 \\ \hline
%             -\mathsf{a}^{\fY}_k  \mathsf{a}^{\fX}_k  / 64 & \sigk^4 + o(\sigk^4) & -\sigk^5 + o(\sigk^5) & -\sigk^2 + o(\sigk^2) & \sigk^3 + o(\sigk^2)
%             \\[2mm]\hline
%             -\mathsf{a}^{\fY}_k \mathsf{b}^{\fX}_k / 64 & -2\sigk^4 + o(\sigk^4) & -2\sigk^3 + o(\sigk^3) & 4\sigk^4 + o(\sigk^4) & -2\sigk^5 + o(\sigk^5) 
%             \\[2mm]\hline
%             -\mathsf{g}^{\fY}_k \mathsf{g}^{\fX}_k / 64   & -\sigk^4 + o(\sigk^4) &-\frac{3}{2}\sigk^3 + o(\sigk^3) & -3\sigk^3 + o(\sigk^3) & 4\sigk^4 + o(\sigk^4)\\
%         \hline
%         \end{array}
%     \end{equation*}
% \end{center}
% For the functions $\beta^1$ and $\beta^2$, using that $1-|\Azink|^2 = 4\sigk+o(\sigk)$, we obtain that,
% \begin{align*}
%     \beta^1(|\Azink|) &= -(\ln(\sigk)+\ln(4)+1)+o(1),\\
%     \beta^2(|\Azink|) &= -(\ln(2)-1)+o(1).
% \end{align*}

\subsubsection{Projection Operator}

\begin{lem}
    \label{lem:phipsiIntegral}
    The integral $\langle \psi^{\vK_1}, \phi^{\vK_1}\rangle$ admits the following explicit form,
    \begin{equation*}
        \langle \psi^{\vK_1}, \phi^{\vK_1}\rangle  = \frac{8\pi^2 k \Aek^2 }{\cEll\lmbk^2}\left(p^k_1 \tauk + p^k_0\right).
    \end{equation*}
    where,
    \begin{align*}
        p^k_1 & =  1+|\Azink|^2 - \frac{4|\Azink|}{1+|\Azink|} + \frac{8\sigk^2|\Azink|}{1-|\Azink|^2} + \underset{\sigma_\theta \to 0}{o}(1),\\
        p^k_0 & = 4\sigk \frac{1-|\Azink|}{1+|\Azink|} + \underset{ \sigma_\theta \to 0}{o}(1).
    \end{align*}
\end{lem}

\begin{proof}
    Using the Fourier product for the $\x$-variable we obtain,
    \begin{equation*}    
        \langle \psi^{\vK_1}, \phi^{\vK_1}\rangle = \int (\fU^{\vK_1}_{\sigma_\theta} \fV^{-\vK_1}_{\sigma_\theta} + \fU^{-\vK_1}_{\sigma_\theta} \fV^{\vK_1}_{\sigma_\theta} )\dd \theta.
    \end{equation*}
    Expanding the integral at $\sigma_\theta = 0$, as the resolvent $\fOpR^l_0$ is explicit, we obtain, 
    \begin{align*}
        \langle \psi^{\vK_1}_0, \phi^{\vK_1}_0\rangle &= \int (\fU^{\vK_1}_{0} \fV^{-\vK_1}_{0} + \fU^{-\vK_1}_{0} \fV^{\vK_1}_{0} )\dd \theta = \int (\fU^{\vK_1}_{0} \fV^{-\vK_1}_{0} + \fU^{-\vK_1}_{0} \fV^{\vK_1}_{0} )\dd \theta, \\
        &=  \int\left(\frac{-\partial_\theta \fB_{\vK_1}}{(\fM_{\vK_1})^2}+\frac{-\partial_\theta \fB_{-\vK_1}}{(\fM_{-\vK_1})^2}\right)\dd \theta,\\
        &=  -\int\left(\frac{\partial_\theta \fB_{\vK_1}(\fM_{-\vK_1})^2+\partial_\theta \fB_{-\vK_1}(\fM_{\vK_1})^2}{|\fM_{\vK_1}|^4}\right)\dd \theta,\\
        &= -4\pi \sum_{n\in \dZ}\mathsf{d}^{\vK_1,4}_n \mathsf{m}_{-n}^{\phi\psi}
    \end{align*}
    with,
    \begin{equation*}
        \mathsf{m}_n^{\phi\psi} \defeq \mathcal{F}_\theta \{ \partial_\theta \fB_{\vK_1}(\fM_{-\vK_1})^2\}_{n}. 
    \end{equation*}

    Using the explicit form of Table~\ref{tb:FourierMultipliers}, and introducing the rescaling,

     \begin{equation*}
        \mathsf{m}_n^{\phi\psi} = \frac{\pi |k|}{2\mathcal{E}^c_\mathbf{k}}\lmbk^2\Tilde{\mathsf{m}}_n^{\phi\psi},
    \end{equation*}
    such that the Fourier modes $\Tilde{\mathsf{m}}_n^{\phi\psi}$ are explicitly given by
    
    \begin{center}
        \begin{equation*}
            \begin{array}{|c||c|c|c|}
                \hline
                 & -4 & -2 & 0 \\ \hline
                % \Tilde{\mathsf{c}}_n^{\phi\psi} & \tauk/2 & -\sigk - 2\sigk^2 \tauk + \tauk & - 2\sigk + \tauk
                \Tilde{\mathsf{m}}_n^{\phi\psi} & -\tauk/2  & -2\sigk -\tauk-2\sigk^2\tauk & -4\sigk-\tauk \\
            \hline
            \end{array}
        \end{equation*}
    \end{center}

    We conclude with the final required form by using the Cauchy product between $\mathsf{d}^{\vK_1,4}_n$ and $\mathsf{m}_{-n}^{\phi\psi}$, with the result of Lemma~\ref{lem:MhatmdFourier}.
    
    % \begin{equation}
    %     % \label{eq:AppPhi1Psi1Integral}
    %     % \langle \psi^1_{0}, \phi^1_{0}\rangle = \frac{64\pi^2 k }{\cEll}\lmbk^2\Aek^2 (1-|\Azink|^2)^{-1}\left(\tauk (|\Azink|^2(1-|\Azink|^2)+ 4(1-\Azink)^2+16\sigk^2|\Azink|-2)-4\sigk(1-|\Azink|)^2\right)
    %     \label{eq:AppPhi1Psi1Integral}
    %     \langle \psi^{\vK_1}_0, \phi^{\vK_1}_{0}\rangle = \frac{8\pi^2 k }{\cEll \lmbk^2}\Aek^2 (1-|\Azink|^2)^{-1}\left(\tauk\left(-|\Azink|^4+4|\Azink|^2 -4|\Azink|+ 8\sigk^2|\Azink| + 1\right)+4\sigk(1-|\Azink|)^2\right)
    % \end{equation}
\end{proof}

%%%%%%%%%%%%%%%%%%%%%%%%%%%%%%%%%%%%%%%%%%%%%%%%%%%%%%%%%%%%%%%%%%%%%%%%%%%%
%
%                       Symmetries of the Equation
%
%%%%%%%%%%%%%%%%%%%%%%%%%%%%%%%%%%%%%%%%%%%%%%%%%%%%%%%%%%%%%%%%%%%%%%%%%%%%

\subsection{Symmetries of the Equation}
\label{sec:AppendixSymetries}
\begin{proof}[Proof of Proposition~\ref{antipodalReflectionCommutations}]
    The product property can be checked directly, similarly one  easily obtains that $\Delta_x (\antipodreflOp f )= \antipodreflOp (\Delta_x f)$ and that $\partial_{\thth}( \antipodreflOp f )= \antipodreflOp (\partial_{\thth} f)$. For the transport term, writing explicitly the change of variable, we obtain,
    \begin{align*}
        (v_\theta \cdot \nabla_x (\antipodreflOp f))(x,\theta) & = - v(\theta) \cdot \nabla_x f (-x,\theta + \pi)\\
        & =  v(\theta + \pi ) \cdot \nabla_x f (-x,\theta + \pi)\\
        & =  (\antipodreflOp (v_\theta \cdot \nabla_x  f))(x,\theta).
    \end{align*}
    Finally, the orientation transport term follows from the product distributivity and the commutation $\antipodreflOp (\partial_\theta f) = \partial_\theta (\antipodreflOp f)$. Only the commutation with the operator $\B$ is left to be checked. By definition of $\B$, we have that,
    \begin{equation*}
        \B[\Tilde{c}] = v^\perp (\theta) \cdot \nabla_x \Tilde{c} + v^\perp (\theta) \cdot \nabla_x^2 \Tilde{c} v(\theta),
    \end{equation*}
    with $\Tilde{c}$ a solution of,
    \begin{equation*}
        \Delta_x \Tilde{c} - \gamma \Tilde{c} + \int \antipodreflOp f \dd \theta = 0.
    \end{equation*}
    Defining $c = \antipodreflOp \Tilde{c}$, we obtain that $c$ is a solution of the equation,
    \begin{equation*}
        \Delta_x c - \gamma c + \int f \dd \theta = 0,
    \end{equation*}
    and that,
    \begin{align*}
        \B[\antipodreflOp c ](x,\theta) & = \B[\Tilde{c}](x,\theta)\\
        & =  - v^\perp (\theta) \cdot \nabla_x c (-x) + v^\perp (\theta) \cdot \nabla_x^2 c(-x) v(\theta)\\
        & =  v^\perp (\theta + \pi) \cdot \nabla_x c (-x) + v^\perp (\theta + \pi) \cdot \nabla_x^2 c(-x) v(\theta + \pi),\\
        & = (\antipodreflOp (\B[c]))(x,\theta).
    \end{align*}
    This concludes the proof.
\end{proof}

\begin{proof}[Proof of Proposition~\ref{prop:SwapOperator}]
The involution property follows from the following direct computation,
\begin{equation*}
    \swapOp\swapOp f(x_1,x_2,\theta)=\swapOp f(-x_2,-x_2,-\theta-\tfrac{\pi}{2})=f(x_1,x_2,\theta).
\end{equation*}
The self-adjointness of $\swapOp$ follows by the change of variable formula; for $f,g \in L^2_0$, we have that, 
\begin{align*}
    \ \int \swapOp f(x_1,x_2,\theta) g(x_1,x_2,\theta)\dd x_1\dd x_2\dd\theta=&\ \int f(-x_2,-x_1,-\theta-\tfrac{\pi}{2})g(x_1,x_2,\theta)\dd x_1\dd x_2\dd\theta,\\
    =& \ \int f(x_1,x_2,\theta)g(-x_2,-x_1,-\theta-\tfrac{\pi}{2})\dd x_1 \dd x_2\dd\theta,\\
    =&\ \int f(x_1,x_2,\theta)\swapOp g(x_1,x_2,\theta)\dd x_1\dd x_2\dd\theta.
\end{align*}
For the commutation property, we expand as follows,
\begin{align*}
    \F(\swapOp f,\chi)=& \ \sigma_x\Delta_x \swapOp f+\sigma_\theta\partial_{\theta\theta} \swapOp f-\lambda v_\theta\cdot\nabla_x \swapOp f-\chi\partial_\theta(\B_\tau[\swapOp f](\tfrac{1}{2\pi}+\swapOp  f)).
\end{align*}
% \begin{align*}
%         \cos(\theta) &= -\sin(-\theta - \pi/2)\\
%         \sin(\theta) &= -\cos(-\theta - \pi/2)\\
%         \cos(2\theta)&= -\cos(-2\theta - \pi)\\
%         \sin(2\theta)&= \sin(-2\theta - \pi).
% \end{align*}
Then, using basic trigonometric identities, we treat each term separately as follows. Firstly, the transport term,
\begin{align*}
    (v_\theta\cdot \nabla_x (\swapOp f)) (x_1,x_2,\theta) = & - \cos(\theta) \partial_{x_2} f(-x_2,-x_1,-\theta-\pi/2),\\
    & - \sin(\theta) \partial_{x_1} f(-x_2,-x_1,-\theta-\pi/2)\\
    = & \  \sin(-\theta - \pi/2) \partial_{x_2} f(-x_2,-x_1,-\theta-\pi/2),\\
    &+ \cos(-\theta - \pi/2)\partial_{x_1} f(-x_2,-x_1,-\theta-\pi/2),\\
    =& \  (\swapOp  (v_\theta\cdot \nabla_x f)) (x_1,x_2,\theta).
\end{align*}
Secondly, we easily verify that, $\swapOp\Delta_x f = \Delta_x \swapOp f, \swapOp\partial_{\thth} f = \partial_{\thth} \swapOp f, \swapOp\partial_{\theta} f = -\partial_{\theta} \swapOp f, \swapOp(\B_\tau[f]f) = (\swapOp \B_\tau[f]) (\swapOp f)$. Thirdly, expanding the $\B$ operator gives,
\begin{align*}
    (\B_\tau [\swapOp f]) (x_1,x_2,\theta) = & \sin(\theta) \partial_{x_2} c(-x_2,-x_1) - \cos(\theta) \partial_{x_1}c(-x_2,-x_1), \\
    &+ \frac{\tau}{2}\sin(2\theta) (\partial_{x_1 x_1}c - \partial_{x_2 x_2}c )(-x_2,-x_1), \\
    &+ \tau \cos(2\theta) \partial_{x_1 x_2} c(-x_2,-x_1),\\
    =& -\cos(-\theta-\pi/2) \partial_{x_2} c(-x_2,-x_1), \\
    &+ \sin(-\theta-\pi/2) \partial_{x_1}c(-x_2,-x_1), \\
    &- \frac{\tau}{2}\sin(-2\theta - \pi) (\partial_{x_2 x_2}c - \partial_{x_1 x_1}c)(-x_2,-x_1), \\
    &- \tau \cos(-2\theta - \pi) \partial_{x_1 x_2} c(-x_2,-x_1)\\
    =&- (\swapOp \B_\tau [f]) (x_1,x_2,\theta).
\end{align*}
Finally, using the distributive property and the above, we obtain that,
\begin{equation*}
    \swapOp\partial_\theta\left(\B_\tau [f]\left(\tfrac{1}{2\pi}+ f\right)\right) = \partial_\theta\left(\B_\tau [\swapOp f]\left(\tfrac{1}{2\pi}+\swapOp f\right)\right),
\end{equation*}
and this concludes the proof.
\end{proof}

%%%%%%%%%%%%%%%%%%%%%%%%%%%%%%%%%%%%%%%%%%%%%%%%%%%%%%%%%%%%%%%%%%%%%%%%%%%%
%
%                       Numerics
%
%%%%%%%%%%%%%%%%%%%%%%%%%%%%%%%%%%%%%%%%%%%%%%%%%%%%%%%%%%%%%%%%%%%%%%%%%%%%

\subsection{Numerical simulation of sub and super-critical bifurcations}
\label{sec:AppendixNumerics}
In Figure~\ref{fig:numericalbif} we show a graph for a numerical result accompanying the super- and subcriticality of the bifurcation as in the main result. In the graph we show two lines, with the blue line corresponding to the existence of numerical non-trivial stationary solutions and the red line corresponding to the value of the largest real eigenvalue of the linearized operator around the constant solution. The numerical solutions are approximated using a finite volume discretization as in \cite{bruna2025convergence}. After the forward-in-time finite volume scheme has reached a small enough value of the residual, we stop the simulation and start a new one with a smaller value of $\chi$ with initial data $f^{0,\chi_{n+1}}=f^{\chi_n}(t_{\mathrm{final}})$ from the previous simulation. For the spectrum, we approximate the Fr\'echet derivative of the discrete forward operator in each direction of the finite-dimensional finite volume domain. These two numerical results illustrate the behavior predicted by the bifurcation diagram in Figure~\ref{fig:subsupercriticiality}. The simulation on the left corresponds to the supercritical case, and the simulation on the right corresponds to the subcritical case.

\begin{figure}[!ht]
    \centering
    \includegraphics[width=1.\linewidth]{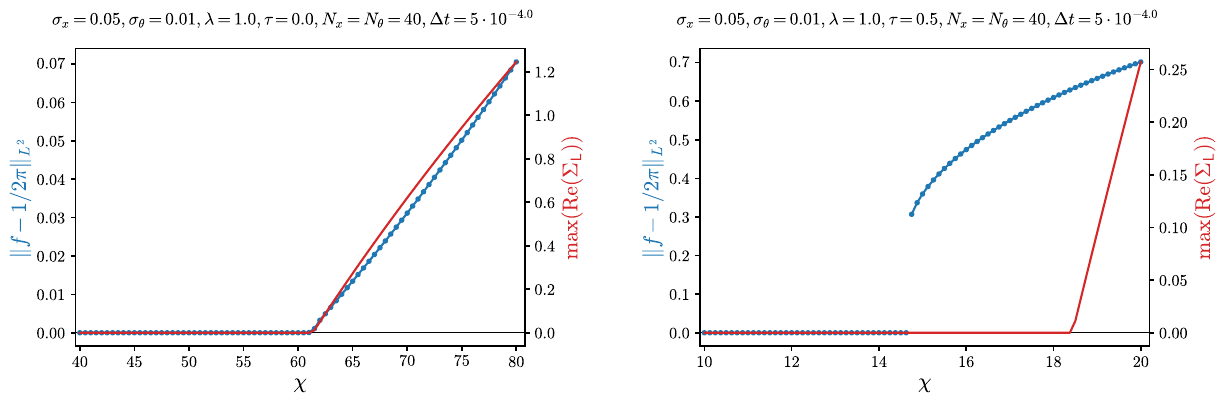}
    \caption{Numerical bifurcation diagram for $\tau=0$ on the left and $\tau = 0.5$ on the right. We used a forward-in-time version of the finite volume scheme as in \cite{bruna2025convergence}. The residuals $\|\F(f,\chi)\|_{L^\infty}$ are of the order $10^{-6}$ or less.}
    \label{fig:numericalbif}
\end{figure}
% \begin{figure}
%     \centering
%     \includegraphics[width=0.7\textwidth]{Figures/bifuplot_spec_l=1.0_ts=0.0_chi=80.0-0.540.0_sx=0.05_st=0.01_g=1.0_sc=1.0_Nx=40_Nt=40_steps=60000_skips=5000_dt=0.0005.pdf}
%     \caption{Numerical bifurcation diagram for $\tau=0$. We used a forward-in-time version of the finite volume scheme as in \cite{bruna2025convergence}. The residuals $\|F(f,\chi)\|_{L^\infty}$ are of the order $10^{-6}$ or less.}
%     \label{fig:tau0}
% \end{figure}
% \begin{figure}
%     \centering
%     \includegraphics[width=0.7\textwidth]{Figures/bifuplot_spec_l=1.0_ts=0.5_chi=20.0-0.12510.0_sx=0.05_st=0.01_g=1.0_sc=1.0_Nx=40_Nt=40_steps=60000_skips=5000_dt=0.0005.pdf}
%     \caption{Numerical bifurcation diagram for $\tau=0.5$. We used a finite volume scheme as above. The residuals are of the order $10^{-6}$ or less.}
%     \label{fig:tau05}
% \end{figure}

\endgroup
    
    \newpage
    \section*{References}
    \addcontentsline{toc}{section}{References}
    \printbibliography[heading=none]

\end{document}